\documentclass [11pt,oneside,a4paper,mathscr]{amsart}

\usepackage{amscd,amsmath,amssymb,euscript}
\usepackage[frame,cmtip,curve,arrow,matrix,line,graph]{xy}
\usepackage{tikz-cd}
\usepackage{hyperref}
\usepackage{stmaryrd}
\usepackage{bm}

\title[]{Topological K-theory of quasi-BPS categories for Higgs bundles}
\author{Tudor P\u adurariu and Yukinobu Toda}

\newtheorem{thm}{Theorem}[section]
\newtheorem{cor}[thm]{Corollary}
\newtheorem{prop}[thm]{Proposition}
\newtheorem{conj}[thm]{Conjecture}
\newtheorem{lemma}[thm]{Lemma}

\theoremstyle{definition}
\newtheorem{defn}[thm]{Definition}

\newtheorem{thm*}[thm]{Theorem$^*$}
\newtheorem{remark}[thm]{Remark}

\newcommand{\comment}[1]{}

\renewcommand{\leq}{\leqslant}
\renewcommand{\geq}{\geqslant}

\newcommand{\X}{\mathcal{X}}

\newcommand{\rank}{\operatorname{rank}}

\newcommand{\Coh}{\operatorname{Coh}}

\newcommand{\Ker}{\operatorname{Ker}}
\newcommand{\id}{\operatorname{id}}
\newcommand{\Ext}{\operatorname{Ext}}

\newcommand{\Hom}{\operatorname{Hom}}

\newcommand{\Spec}{\operatorname{Spec}}

\newcommand{\GL}{\operatorname{GL}}

\newcommand{\inclusion}{\ar@<-0.3ex>@{^{(}->}[r]}
\newcommand{\iinclusion}{\ar@<-0.3ex>@{^{(}->}[rr]}

\newcommand{\ssslash}{/\!\!/}

\usetikzlibrary{positioning,fit,calc}
\tikzstyle{block}=[draw=black, width=1cm, minimum height=2cm, align=center] 
\tikzstyle{block2}=[draw=black, text width=2cm, minimum height=1cm, align=center] 
\tikzstyle{block3}=[draw=black, text width=2cm, minimum height=1cm, align=center] 

\makeatletter

\@addtoreset{equation}{section}	
\makeatother
\begin{document}
\begin{abstract}
In a previous paper, we
introduced quasi-BPS categories for moduli 
stacks of semistable Higgs bundles. 
Under a certain condition on the rank, Euler characteristic, and weight, 
the quasi-BPS categories (called BPS in this case) are non-commutative analogues of Hitchin integrable systems. We proposed a conjectural equivalence between 
BPS categories which swaps Euler characteristics and weights. 
The conjecture is 
inspired by the Dolbeault Geometric Langlands equivalence of Donagi--Pantev, by the Hausel--Thaddeus mirror symmetry, and by the $\chi$-independence phenomenon for BPS invariants of curves on Calabi-Yau threefolds.

In this paper, we show that the above conjecture holds at 
the level of topological K-theories. When the rank and the Euler characteristic are coprime, such an isomorphism was proved by Groechenig--Shen. Along the way, we show that the topological K-theory of BPS categories is isomorphic to the BPS cohomology of the moduli of semistable Higgs bundles. 
\end{abstract}

	\maketitle
\setcounter{tocdepth}{2}
\tableofcontents

\section{Introduction}

\subsection{Hausel--Thaddeus mirror symmetry of Higgs bundles}
Let $C$ be a smooth projective curve of genus $g$, 
and let the group $G$ be either $\mathrm{GL}(r), \mathrm{SL}(r)$ or $\mathrm{PGL}(r)$. 
We denote by 
\begin{align*}
\pi \colon 
    M_G(\chi) \to B
\end{align*}
the moduli space of semistable 
$G$-Higgs bundles 
with Euler characteristic $\chi$ together with the Hitchin fibration with Hitchin base $B$. 
In the case of $G=\mathrm{GL}(r)$, it consists of pairs
\begin{align*}
(F, \theta), \ \theta \colon F \to  F\otimes \Omega_C,
\end{align*}
where $F$ is a vector bundle on $C$ with $(\rank(F), \chi(F))=(r, \chi)$
and $(F, \theta)$ satisfies a stability condition. 

In general, the stack $M_G(\chi)$ is singular, however it is a smooth Deligne-Mumford stack
if $(r, \chi)$ are coprime. 
Suppose that both of 
$\chi$ and $w$ are coprime with $r$. 
In this case,
Hausel--Thaddeus~\cite{HauTha} proposed that the pair of smooth 
Deligne-Mumford stacks (together with some Brauer classes) 
\begin{align*}(M_G(\chi), M_{G^L}(w))
\end{align*}
is a mirror pair~\cite{SYZ}. In particular, Hausel--Thaddeus proposed the equality of the 
stringy Hodge numbers of $M_G(\chi)$ and $M_{G^L}(w)$, which was proved by Groechenig--Wyss--Ziegler~\cite{GrWy}. 
At the categorical level, one expects a derived equivalence~\cite{DoPa} (called ``the Dolbeault Langlands equivalence"~\cite{ZN}):
\begin{align}\label{intro:mirror}
    D^b(M_G(\chi), \alpha^{-w}) \stackrel{\sim}{\to} D^b(M_{G^L}(w), \beta^{\chi}),
\end{align}
where $\alpha, \beta$ are some canonical Brauer classes, 
see for example~\cite[Section~2.4]{HauICM}.
Heuristically (following Donagi--Pantev~\cite{DoPa}), 
the conjectural equivalence (\ref{intro:mirror}) may be regarded as a classical 
limit of the geometric de Rham Langlands correspondence \cite{AG}:
\begin{align}\label{GLC}
\mathrm{Ind}_{\mathcal{N}}D^b(\mathcal{L}\mathrm{ocSys}_{G}) \simeq \text{D-mod}(\mathcal{B}\mathrm{un}_{G^L}). 
\end{align}
Here $\mathcal{L}\mathrm{ocSys}_{G}$ is the moduli stack of $G$-flat connections on $C$ 
and $\mathcal{B}\mathrm{un}_{G^L}$ is the moduli stack of $G^L$-bundles on $C$. 
The equivalence (\ref{intro:mirror}) may be regarded as an extension over the full Hitchin base of the Fourier-Mukai equivalence between dual abelian schemes, which are constructed from the Picard schemes of the smooth spectral curves.
The equivalence \eqref{intro:mirror} should match Hecke and Wilson operators on the two sides~\cite{KapWit, DoPa, HauICM}, but we do not discuss this aspect in the paper.

The existence of an equivalence (\ref{intro:mirror}) is an open problem 
in most of the cases. Besides the equality of Hodge stringy numbers conjectured in \cite{HauTha}, it is known that versions of \eqref{intro:mirror} hold in cohomology
(by Groechenig--Wyss--Ziegler~\cite{GrWy}, see also \cite{LoWy}), for Hodge structures
(by Maulik--Shen~\cite{MSend}), for Chow and Voevodsky motives (by Hoskins--Pepin Lehalleur \cite{HoskinsPL}), and for relative Chow motives (by Maulik--Shen--Yin~\cite{MaShYi}).
Groechenig--Shen considered a comparison in topological K-theory~\cite{GS}, and showed an equivalence of (integral) topological K-theories spectra:
\begin{align}\label{intro:equiv:K}
    K^{\rm{top}}(M_G(\chi), \alpha^{-w}) \stackrel{\sim}{\to} K^{\rm{top}}(M_{G^L}(w), \beta^{\chi}). 
\end{align}

The first main difficulty in proving the equivalence \eqref{intro:mirror} is the construction of a candidate Fourier-Mukai kernel, which should be an extension of the Poincaré sheaf, which is initially defined only over the locus of smooth spectral curves. Such an extension was constructed by Arinkin \cite{Ardual} over the elliptic locus (when the spectral curve is reduced and irreducible), by Melo--Rapagnetta--Viviani \cite{MRVF, MRVF2} over the locus of reduced spectral curves, and by Li~\cite{MLi} for rank two meromorphic Higgs bundles. 
However, this is not an impediment in proving the statements in \cite{GrWy, LoWy, MSend, HoskinsPL, GS, MaShYi}, which follow from a good understanding of the comparison of the two sides over the elliptic locus. For example, any extension over the full Hitchin base of Arinkin's kernel induces the isomorphism \eqref{intro:equiv:K}.

The purpose of this paper is to give a version of 
the equivalence (\ref{intro:equiv:K}) for integers 
$\chi, w$ which are not necessary coprime with $r$, 
using the quasi-BPS categories studied in the previous paper~\cite{PThiggs}. 

\subsection{Symmetry of quasi-BPS categories}\label{subsection12}
So far in the literatures, the studies of mirror symmetry of Higgs bundles
have been restricted to the case when $(r, \chi)$ are coprime. 
In this case, the moduli space $M_G(\chi)$ is a smooth Deligne-Mumford stack, 
and its derived category has several nice properties, for example it is smooth over $\mathbb{C}$ and 
proper over the Hitchin base $B$. 

However, it is important to study the (derived) moduli stacks of semistable Higgs bundles $\mathcal{M}_G(\chi)$ for general $\chi$. First, all such moduli stacks are used in the definition of the categorical Hall algebra of the surface $\mathrm{Tot}_C(\Omega_C)$~\cite{PoSa}, and needed to be studied in order to categorify theorems known for (Kontsevich--Soibelman~\cite{MR2851153}) cohomological Hall algebras~\cite{KinjoKoseki, DHSM}.

Second, the stack $\mathcal{L}\mathrm{ocSys}_{G}$ degenerates to $\mathcal{M}_{G}(0)$, see \cite[Proposition 4.1]{Simp}.
Thus, when studying the limit of the de Rham Langlands equivalence \eqref{GLC} (following Donagi--Pantev), 
the limit of the right hand side should be a category of ind-coherent sheaves on $\mathcal{M}_{G}(0)$.

Third, for $\mathrm{G}=\mathrm{SL}(r)$,  
principal $\mathrm{SL}(r)$-Higgs bundles have degree zero, thus their Euler characteristic
is equal to $r(1-g)$, 
which is divisible by $r$. Therefore, considering quasi-BPS categories for $\chi=r(1-g)$ 
is 
essential in the categorical study of principal $\mathrm{SL}$-Higgs bundles.

In~\cite{PThiggs}, we introduced some admissible 
subcategories, called \textit{quasi-BPS categories}:
\begin{align}\label{intro:qbps}
    \mathbb{T}_G(\chi)_w \subset D^b(\mathcal{M}_G(\chi))_w
\end{align}
for $w \in \mathbb{Z}$ corresponding to a weight with respect to the action of the center of $G$. 
The construction of the categories (\ref{intro:qbps}) is part of the problem of categorifying BPS invariants on Calabi-Yau 3-folds~\cite{PTK3, PTtop, PTquiver}, or more generally of categorifying the constructions and theorems from (numerical or cohomological) Donaldson-Thomas theory~\cite{T}. Indeed, the quasi-BPS categories~\eqref{intro:qbps} have analogous properties to the BPS cohomology (defined by Kinjo--Koseki~\cite{KinjoKoseki} and Davison--Hennecart--Schlegel Mejia~\cite{DHSM}) of the local Calabi-Yau threefold $X=\mathrm{Tot}_C(\Omega_C)\times \mathbb{A}^1_\mathbb{C}$, see \cite{PThiggs} for more details. In this paper, we make this relation precise by computing the topological K-theory of quasi-BPS categories in terms of BPS cohomology, see Proposition \ref{prop:topK}, Proposition \ref{prop:topKquasi}, and Theorem \ref{prop:BPS}.

If the vector $(r, \chi, w)$ is primitive, i.e. $\gcd(r, \chi, w)=1$, the category (\ref{intro:qbps})
is smooth over $\mathbb{C}$ and proper 
over the Hitchin base $B$. In this case, we regard it as a non-commutative 
analogue of the Hitchin system. 
Note that neither $\chi$ nor $w$ may be coprime with $r$, even if $(r, \chi, w)$ is primitive. 
In~\cite{PThiggs}, we conjectured the existence of an equivalence 
\begin{align}\label{intro:equiv2}
    \mathbb{T}_G(w)_{-\chi} \stackrel{\sim}{\to} \mathbb{T}_{G^L}(\chi)_{w}
\end{align}
extending the Donagi--Pantev equivalence over the locus of smooth spectral curves,
which is nothing but the equivalence (\ref{intro:mirror}) if both $\chi$ and $w$ are coprime with $r$.

\subsection{Main theorem}
The purpose of this paper is to provide evidence towards the equivalence (\ref{intro:equiv2}), namely to prove  that extensions of the Poincaré sheaf induce an isomorphism of the (rational) topological K-theories of the two categories in (\ref{intro:equiv2}). We thus obtain a generalization of the Groechenig--Shen theorem (\ref{intro:equiv:K})
beyond
the coprime case. Note that (\ref{intro:equiv:K}) holds integrally, and that we also prove versions for integral topological K-theory for twisted Higgs bundles. 

For a dg-category $\mathscr{D}$, Blanc~\cite{Blanc} introduced its topological K-theory spectrum 
\begin{align*}
    K(\mathscr{D}) \in \mathrm{Sp}. 
\end{align*}
We denote by $K(\mathscr{D})_{\mathbb{Q}} :=K(\mathscr{D}) \wedge H\mathbb{Q}$ its 
rationalization. 

We use the following notations 
\begin{align}\notag
    \mathbb{T}(r, \chi)_w :=\mathbb{T}_{\mathrm{GL}(r)}(\chi)_w, \ 
    \mathbb{T}_{\mathrm{SL}(r), w}:=\mathbb{T}_{\mathrm{SL}(r)}(r(1-g))_w, \ 
    \mathbb{T}_{\mathrm{PGL}(r)}(\chi):=\mathbb{T}_{\mathrm{PGL}(r)}(\chi)_0. 
\end{align}
We also use the notation $\mathbb{T}(r, \chi)_w^{\rm{red}}$
for the \textit{reduced quasi-BPS category}, 
which is a category obtained from the usual quasi-BPS categories
by removing a redundant derived structure from $\mathcal{M}_G(\chi)$. 
The following is the main theorem in this paper: 
\begin{thm}\emph{(Theorem~\ref{thm:induceK2}, Theorem~\ref{thm:topK:slpgl2})}\label{thm:intro2}

    (1) Suppose that the vector $(r, \chi, w)$ is primitive. For $G=\mathrm{GL}(r)$, there is an equivalence 
    \begin{align}\label{firstiso}
        K^{\rm{top}}(\mathbb{T}(r, w)_{-\chi}^{\rm{red}})_{\mathbb{Q}} \stackrel{\sim}{\to} 
         K^{\rm{top}}(\mathbb{T}(r, \chi)_{w}^{\rm{red}})_{\mathbb{Q}}.
    \end{align}

    (2) Suppose that $\gcd(r, w)=1$. 
    For $(G, G^{L})=(\mathrm{PGL}(r), \mathrm{SL}(r))$, there is an equivalence 
    \begin{align*}
        K^{\rm{top}}(\mathbb{T}_{\mathrm{PGL}(r)}(w))_{\mathbb{Q}} \stackrel{\sim}{\to}
        K^{\rm{top}}(\mathbb{T}_{\mathrm{SL}(r), w})_{\mathbb{Q}}. 
    \end{align*}
\end{thm}

The main ingredient in the proof of the above theorem is a computation of the (rational) topological K-theory of BPS categories. For example, for $G=\mathrm{GL}(r)$, we show in Theorem \ref{prop:BPS} that 
\begin{align}\label{thempropBPS}
\mathcal{K}_M^{\rm{top}}(\mathbb{T}(r, \chi)_{w}^{\rm{red}})_{\mathbb{Q}}\cong \mathcal{BPS}_{M}[\beta^{\pm 1}],
\end{align}
where $\mathcal{K}_M^{\rm{top}}(-)$ is the relative topological K-theory~\cite{Moulinos} over $M$, $\mathcal{BPS}_{M}$ is the BPS sheaf defined by Kinjo--Koseki~\cite{KinjoKoseki} for $M=M_{GL(r)}(\chi)$, and $\beta$ is of degree $2$. 
The isomorphism (\ref{thempropBPS}) explains the use of the name of \textit{BPS categories}.
Note that $\mathcal{BPS}_M$ is (a shift of) the constant sheaf $\mathbb{Q}_{M}$ if $(r,\chi)$ are coprime, and that in general $\mathcal{BPS}_M$ contains $\mathrm{IC}_M$ as a direct summand.
The two sides on \eqref{firstiso} are thus isomorphic because of (cohomological) $\chi$-independence~\cite[Theorem 1.2]{KinjoKoseki}. We use the methods of \cite{GS} to show that an extension of 
Arinkin's kernel induces the isomorphism \eqref{firstiso}. 

Note that the isomorphism \eqref{thempropBPS} shows the weight independence of topological K-theory of BPS categories, which is a phenomenon we also discussed in \cite{PTtop, PTK3}.
The equivalence \eqref{intro:equiv2} thus interchanges the weight and $\chi$-independence phenomena for Higgs bundles.

\subsection{The $L$-twisted case}

The result of Theorem~\ref{thm:intro} is deduced from analogous results 
for $L$-twisted (i.e. meromorphic) Higgs bundles, where \[L \to C\] is a line bundle with 
$\deg L>2g-2$, using the method of vanishing cycles as in \cite{KinjoKoseki, MSend}. 
We denote by $\mathcal{M}^L_G(\chi)$ the moduli stack of 
$L$-twisted semistable $G$-Higgs bundles. 
Note that $\mathcal{M}^L_G(\chi)$ is a smooth stack, whereas $\mathcal{M}_G(\chi)$ is, in general, singular and non-separated.
In the case of $G=\mathrm{GL}(r)$, the moduli stack $\mathcal{M}^L_G(\chi)$ consists of pairs
\begin{align*}
    (F, \theta), \ \theta \colon F \to F\otimes L,
\end{align*}
where $F$ is a vector bundle and $(F, \theta)$ satisfies a
stability condition. 
We can similarly define the quasi-BPS category 
\begin{align*}
    \mathbb{T}_G^L(\chi)_w \subset D^b(\mathcal{M}^L(\chi))_w. 
\end{align*}
We use the notation as in the previous subsection, e.g. 
$\mathbb{T}^L(r, \chi)_w:=\mathbb{T}_{\mathrm{GL}(r)}^L(\chi)_w$. 

As before, we regard quasi-BPS categories as the categorical replacement of the BPS cohomology for the local Calabi-Yau threefold $\mathrm{Tot}_C(L\oplus L^{-1}\otimes \Omega_C)$.
Indeed, the derived category of coherent sheaves on $\mathcal{M}^L(\chi)$ has a semiorthogonal decomposition in Hall products of quasi-BPS categories, analogous to the decomposition of the cohomology of $\mathcal{M}^L(\chi)$ in terms of the BPS cohomology of $M^L:=M^L_{GL(r)}(\chi)$, which is isomorphic to its intersection cohomology~\cite{Mein}. 
In Propositions \ref{prop:topK} and \ref{prop:topKquasi}, we compute the (rational) topological K-theory of quasi-BPS categories in terms of BPS cohomology using the results and methods for quivers~\cite{PTtop}. In particular, we show that, if $(r, \chi, w)$ satisfies the BPS condition, then 
\begin{equation}\label{a}
\mathcal{K}_{M^L}^{\rm{top}}(\mathbb{T}^L(r, \chi)_w)_{\mathbb{Q}} \cong \mathrm{IC}_{M^L}[-\dim M^L][\beta^{\pm 1}].
\end{equation}
The vector $(r, \chi, w)$ satisfies the \textit{BPS condition} 
if the vector 
\begin{align*}
    (r, \chi, w+1-g^{\rm{sp}}) \in \mathbb{Z}^3
\end{align*}
is primitive, where $g^{\rm{sp}}$ is the genus of the 
spectral curve, see the formula (\ref{formula:gD}). 
Note that, if $\deg L$ is even, the BPS condition is equivalent to the vector $(r, \chi, w)$ being
primitive. 
We say that $(r, w)$ satisfies the BPS condition if $(r, 0, w)$ satisfies the BPS condition. 
One can also formulate, for $L$-twisted Higgs bundles, a conjectural derived equivalence analogous to \eqref{intro:mirror}, \eqref{intro:equiv2}, see \cite[Conjecture~4.3]{PThiggs}.
We prove its version for topological K-theory:

\begin{thm}\emph{(Theorem~\ref{thm:induceK}, Corollary~\ref{cor:topK:untsited})}\label{thm:intro}
Let $L \to C$ be a line bundle of $\deg L>2g-2$. 

    (1) Suppose that $(r, \chi, w)$ satisfies the BPS condition. For $G=\mathrm{GL}(r)$, there is an equivalence
    \begin{align}\label{isoL}
        K^{\rm{top}}(\mathbb{T}^L(r, w+1-g^{\rm{sp}})_{-\chi+1-g^{\rm{sp}}}) \stackrel{\sim}{\to} 
         K^{\rm{top}}(\mathbb{T}^L(r, \chi)_{w}).
    \end{align}

    (2) Suppose that $(r, w)$ satisfies the BPS condition. 
    For $(G, G^{L})=(\mathrm{PGL}(r), \mathrm{SL}(r))$, there is an equivalence
    \begin{align*}
        K^{\rm{top}}(\mathbb{T}^L_{\mathrm{PGL}(r)}(w)) \stackrel{\sim}{\to}
        K^{\rm{top}}(\mathbb{T}^L_{\mathrm{SL}(r), w}). 
    \end{align*}
\end{thm}

As in the case $L=\Omega_C$, a first step in proving part (1) is to show that both sides in \eqref{isoL} have isomorphic rational topological K-theory. This follows from \eqref{a} and the cohomological $\chi$-independence for twisted Higgs bundles proved by Maulik--Shen~\cite{DMJS}. Similarly, part (2) uses a cohomological SL/PGL-duality for twisted Higgs bundles~\cite{MSint}.

Note that the 
equivalences of topological K-theories in Theorem~\ref{thm:intro} hold integrally. 
Indeed, we show that the topological K-groups 
in Theorem~\ref{thm:intro} are torsion-free.  
In Theorem~\ref{thm:topK:SLPGL}, we give a slightly stronger statement of part (2) of Theorem~\ref{thm:intro}
which also involves weight/Euler characteristics on 
left/hand hand sides. 

\subsection{Complements}
As we mentioned in the beginning of Subsection \ref{subsection12}, it is important to study \text{all} quasi-BPS categories. Thus it is natural to inquire whether a derived equivalence \eqref{intro:equiv2}, or isomorphisms \eqref{firstiso} or \eqref{isoL}, may holds for all quasi-BPS categories. An analogous such comparison was discussed in \cite[Section 1.3]{PT2} for quasi-BPS categories of points on a surface, which includes points (i.e. rank zero Higgs sheaves) on $\mathrm{Tot}_C(\Omega_C)$. 
We mention a rational analogue of the isomorphism \eqref{isoL} in Subsection \ref{subsec42}.

It is also interesting to pursue an integral version of Theorem \ref{thm:intro2}. 

The methods of this paper (and of \cite{PTtop}) may be used to compute the topological K-theory of quasi-BPS categories (or of noncommutative resolutions defined by \v{S}penko--Van den Bergh~\cite{SVdB}) for other smooth (or quasi-smooth) symmetric stacks, for example for the moduli stack $\mathcal{B}\mathrm{un}(r,d)^{\mathrm{ss}}$ of semistable vector bundles of rank $r$ and degree $d$ on a smooth projective curve $C$. 
In Subsection \ref{subsec:bun}, we briefly discuss the computation of topological K-theory of quasi-BPS categories of $\mathcal{B}\mathrm{un}(r,d)^{\mathrm{ss}}$ in terms of the intersection cohomology of the good moduli space $\mathrm{Bun}(r,d)^{\mathrm{ss}}$. Note that in \cite[Subsection 3.4]{PThiggs} we explained that $D^b(\mathcal{B}\mathrm{un}(r,d)^{\mathrm{ss}})$ has a semiorthogonal decomposition in Hall products of quasi-BPS categories. Thus the results in loc.cit. and in Subsection \ref{subsec:bun} provide a K-theoretic (or categorical) version of a theorem of Mozgovoy--Reineke~\cite[Theorem 1.3]{MR}.

\subsection{Acknowledgements}
T.~P. thanks MPIM Bonn and CNRS for their support during part of the preparation of this paper.
This material is partially based upon work supported by the NSF under Grant No. DMS-1928930 and by the Alfred P. Sloan Foundation under grant G-2021-16778, while T.~P. was in residence at SLMath in Berkeley during the Spring 2024 semester. T.~P. thanks Kavli IPMU for their hospitality and excellent working conditions during a visit in May 2024.

This work started while Y.~T.~was visiting to Hausdorff Research Institute for Mathematics in Bonn on November 12-18, 2023. 
Y.~T.~thanks the hospitality of HIM
Bonn during his visit. 
Y.~T.~is supported by World Premier International Research Center
	Initiative (WPI initiative), MEXT, Japan, and JSPS KAKENHI Grant Numbers JP19H01779, JP24H00180.

\subsection{Notation and convention}
In this paper, all the (derived) stacks are defined over $\mathbb{C}$. 
For a stack $\X$, we use the notion of \textit{good moduli space}
from~\cite{MR3237451}. It generalizes the notion of GIT quotient 
\[R/G \to R\ssslash G,\] where $R$ is an affine variety on which 
a reductive group $G$ acts.

For a (classical) stack $\X$, denote by $D(\mathrm{Sh}_{\mathbb{Q}}(\X))$
the derived category of complexes of $\mathbb{Q}$-constructible sheaves~\cite{MR2312554}.

For a torus $T$, a character $\chi$, and cocharacter $\lambda$, 
we denote by $\langle \lambda, \chi \rangle \in \mathbb{Z}$ its natural 
pairing. 
For a $T$-representation $V$, we denote by $V^{\lambda>0} \subset V$ the 
subspace spanned by $T$-weights $\beta$ with $\langle \lambda, \beta\rangle>0$. 
We also write $\langle \lambda, V^{\lambda>0}\rangle:=\langle \lambda, \det(V^{\lambda>0})\rangle$.

For a variety or stack $\X$, we denote by $D^b(\X)$ the bounded 
derived category of coherent sheaves, which is a pre-triangulated 
dg-category. 
We denote by 
$\mathrm{Perf}(\X)$ the category of perfect complexes 
and $D_{\rm{qcoh}}(\X)$ the 
unbounded derived category of quasi-coherent sheaves. 
For pre-triangulated dg-subcategories $\mathcal{C}_i \subset D^b(\X_i)$ for $i=1, 2$, 
we denote by $\mathcal{C}_1\boxtimes \mathcal{C}_2$ the smallest 
pre-triangulated dg-subcategory of $D^b(\X_1 \times \X_2)$ which contains 
objects $E_1 \boxtimes E_2$ for $E_i \in \mathcal{C}_i$ and closed under direct summands. 

For a dg-category $\mathcal{D}$ with $\mathrm{Perf}(B)$-module structure for a scheme $B$, its semiorthogonal 
decomposition $\mathcal{D}=\langle \mathcal{C}_i\,|\, i\in I\rangle$ is called 
\textit{$\mathrm{Perf}(B)$-linear} if $\mathcal{C}_i \otimes \mathrm{Perf}(B) \subset \mathcal{C}_i$ for all $i\in I$. 

For a smooth stack $\X$ and a regular function 
$f \colon \X \to \mathbb{C}$, 
we denote by $\mathrm{MF}(\X, f)$
the $\mathbb{Z}/2$-graded dg-category of matrix 
factorizations of $f$. 
If there is a $\mathbb{C}^{\ast}$-action on $\X$
such that $f$ is of weight one, 
we also consider 
the dg-category of graded matrix 
factorizations $\mathrm{MF}^{\rm{gr}}(\X, f)$. 
We refer to~\cite[Section~2.6]{PT0} for a review of 
(graded) matrix factorizations.

\section{Quasi-BPS categories for GL-Higgs bundles}
In this section, we first recall some basic properties of moduli stacks of semistable Higgs bundles for $G=\mathrm{GL}(r)$, such as its local description using quivers and the BNR spectral correspondence. We then recall the definition of quasi-BPS categories for Higgs bundles and their conjectural symmetry. We finally briefly discuss Joyce-Song pairs for Higgs bundles.

\subsection{Twisted Higgs bundles}
Let $C$ be a smooth projective curve of genus $g$. 
Let $L$ be a line bundle on $C$ such that either  
\begin{align*}l:=\deg L>2g-2 \mbox{ or }
L=\Omega_C.
\end{align*}
By definition, a $L$-\textit{twisted Higgs bundle} is a pair 
$(F, \theta)$, where $F$ is a vector bundle on 
$C$ and $\theta$ is a morphism 
\begin{align*}
    \theta \colon F \to F \otimes L. 
\end{align*}
When $L=\Omega_C$, it is just called a \textit{Higgs bundle}. 
The (semi)stable $L$-twisted Higgs bundle is defined 
using the slope $\mu(F)=\chi(F)/\rank(F)$
in the usual way: a $L$-twisted Higgs bundle $(F, \theta)$ 
is (semi)stable if we have 
\begin{align*}
    \mu(F') <(\leq) \mu(F), 
\end{align*}
for any sub-Higgs bundle $(F', \theta') \subset (F, \theta)$
such that $\rank(F')<\rank(F)$. 

A $L$-twisted Higgs bundle corresponds to a 
compactly supported pure one-dimensional 
coherent sheaf on the non-compact surface
\begin{align*}
p \colon 
    S=\mathrm{Tot}_C(L) \to C. 
\end{align*}
The correspondence (called \textit{the spectral construction}~\cite{BeNaRa})
is given as follows: 
for a given $L$-twisted Higgs pair $(F, \theta)$, 
the Higgs field $\theta$ determines 
the $p_{\ast}\mathcal{O}_S$-module 
structure on $F$, which in turn 
gives a coherent sheaf on $S$. 
Conversely, a pure one-dimensional 
compactly supported sheaf $E$ on 
$S$ pushes forward to a vector bundle $F$ 
with Higgs field $\theta$ given by the 
$p_{\ast}\mathcal{O}_S$-module structure on it. 

\subsection{Moduli stacks of Higgs bundles}
We denote by 
\begin{align}\label{H:stack}\mathcal{M}^L(r, \chi)
\end{align}
the derived moduli stack of semistable $L$-twisted Higgs bundles 
$(F, \theta)$ with 
\begin{align*}(\rank(F), \chi(F))=(r, \chi).
\end{align*}
It is smooth (in particular classical)
when $\deg L>2g-2$, and 
quasi-smooth when $L=\Omega_C$. 
We omit $L$ in the notation when $L=\Omega_C$, 
i.e. 
$\mathcal{M}(r, \chi):=\mathcal{M}^{\Omega_C}(r, \chi)$.
We also denote by 
$(\mathcal{F}, \vartheta)$ the universal Higgs bundle 
\begin{align}\label{univ:F}
    \mathcal{F} \in \mathrm{Coh}(C \times \mathcal{M}^L(r, \chi)), \ 
    \vartheta \colon \mathcal{F} \to \mathcal{F} \boxtimes L. 
\end{align}

The stack (\ref{H:stack}) is equipped with the Hitchin map
\begin{align*}
  h \colon \mathcal{M}^L(r, \chi) \to B^L(r,\chi) :=\bigoplus_{i=1}^r
    H^0(C, L^{\otimes i})
\end{align*}
sending $(F, \theta)$ 
to $\mathrm{tr}(\theta^{i})$
for $1\leq i\leq r$. When $r, \chi$ are clear from the context, we write $B^L:=B^L(r,\chi)$, $B:=B^{\Omega_C}$. 
We have the factorization 
\begin{align*}
    h \colon \mathcal{M}^L(r, \chi)^{\rm{cl}} 
    \stackrel{\pi}{\to} M^L(r, \chi) \to B^L,
\end{align*}
where the first morphism is the good moduli 
space morphism. 
A closed point $y \in M^L(r, \chi)$ corresponds to a polystable 
Higgs bundle 
\begin{align}\label{polystable}
E=\bigoplus_{i=1}^k V_i \otimes E_i,
\end{align}
where $E_i$ is a stable $L$-twisted Higgs bundle 
such that $(\rank(E_i), \chi(E_i))=(r_i, \chi_i)$ 
satisfies $\chi_i/r_i=\chi/r$
and 
$V_i$ is a finite dimensional vector space. 
By abuse of notation, we also denote by $y \in \mathcal{M}^L(r, \chi)$ the closed point represented by (\ref{polystable}).
It is the unique closed point in the fiber of 
$\mathcal{M}^L(r, \chi)^{\rm{cl}} \to M^L(r, \chi)$ at $y$. 
We also have the Cartesian square
\begin{align}\label{dia:stable}
\xymatrix{
\mathcal{M}^L(r, \chi)^{\rm{st}} \inclusion \ar[d] & \mathcal{M}^L(r, \chi)^{\rm{cl}} \ar[d] \\
M^L(r, \chi)^{\rm{st}} \inclusion & M^L(r, \chi), 
}
\end{align}
where $(-)^{\rm{st}}$ is the open locus of stable points, the horizontal 
arrows are open immersions, and the left vertical arrow is a good 
moduli space morphism which is a $\mathbb{C}^{\ast}$-gerbe. 
\begin{lemma}\emph{(\cite[Lemma~2.3]{PThiggs})}
If $l>0$, then $M^L=M^L(r, \chi)$ is Gorenstein with trivial dualizing sheaf 
$\omega_{M^L}=\mathcal{O}_{M^L}$, and the right vertical 
arrow in (\ref{dia:stable}) is generically a $\mathbb{C}^{\ast}$-gerbe. 
\end{lemma}

A point $b \in B^L$ corresponds to a support
$\mathcal{C}_b \subset S$ of 
the sheaf on $S$, called the \textit{spectral curve}. 
We denote by $g^{\rm{sp}}$ the 
arithmetic genus of the spectral curve, which is given by  
\begin{align}\label{formula:gD}
    g^{\rm{sp}}:=1+(g-1)r-\frac{rl}{2}+\frac{r^2 l}{2}. 
\end{align}

Let $\mathcal{C} \to B^L$ be the universal spectral curve, 
which is a closed subscheme of $S \times B^L$. 
By the spectral construction, the universal Higgs bundle 
corresponds to a universal sheaf 
\begin{align}\label{univ:E}
    \mathcal{E} \in \Coh(\mathcal{C}\times_{B^L} \mathcal{M}^L(r, \chi))
\end{align}
which is also regarded as a coherent sheaf on $S \times \mathcal{M}^L(r, \chi)$
by the closed immersion 
$\mathcal{C}\times_{B^L} \mathcal{M}^L(r, \chi) \hookrightarrow S \times \mathcal{M}^L(r, \chi)$.

\subsection{The local description of moduli stacks of Higgs bundles}\label{subsec:loc}
The good moduli space 
\begin{align}\label{good:pi}
\pi \colon \mathcal{M}^L(r, \chi)^{\rm{cl}} \to M^L(r, \chi)
\end{align}
is, locally near a point $y\in M^L(r, \chi)$, 
described in terms of the representations of the Ext-quiver of $y$. 
In what follows, we write 
\begin{align}\label{rchid}
(r, \chi)=d(r_0, \chi_0),
\end{align}
where $d\in \mathbb{Z}_{>0}$ and $(r_0, \chi_0)$ are coprime. 

Let $y \in M^L(r, \chi)$ be a closed point corresponding 
to the polystable object (\ref{polystable}).
The associated Ext-quiver $Q_y$ consists of 
vertices $\{1, \ldots, k\}$ 
with the number of arrows given by 
\begin{align*}
    \sharp(i \to j)=\dim \Ext_S^1(E_i, E_j)=
    \begin{cases}
    r_i r_j l+\delta_{ij}, & l>2g-2, \\
    r_i r_j (2g-2)+2\delta_{ij}, & L=\Omega_C. 
    \end{cases}
\end{align*}
Here, by the spectral construction, we regard $E_i$ as a coherent 
sheaf on $S$. 
The representation space of $Q_y$-representations of dimension vector $\bm{d}=(d_i)_{i=1}^k$ 
for 
$d_i=\dim V_i$ is given by 
\begin{align*}
    R_{Q_y}(\bm{d}):=\bigoplus_{(i\to j) \in Q_y}\Hom(V_i, V_j) =\Ext_S^1(E, E). 
\end{align*}
Let $G(\bm{d})$ be the algebraic group
\begin{align*}
    G(\bm{d}):=\prod_{i=1}^k GL(V_i)=\mathrm{Aut}(E)
\end{align*}
which acts on $R_{Q_y}(\bm{d})$ by the conjugation. 

If $l>2g-2$, by the Luna étale slice theorem, 
étale locally at $y$, the 
map \[\pi \colon \mathcal{M}^L(r, \chi) \to M^L(r, \chi)\]
is isomorphic to 
\begin{align}\label{etslice}
\X_{Q_y}(\bm{d}):=R_{Q_y}(\bm{d})/G(\bm{d}) \to X_{Q_y}(\bm{d}):=R_{Q_y}(\bm{d}) \ssslash G(\bm{d}).     
\end{align}
Note that the above quotient stack is the moduli stack 
of $Q_y$-representations of dimension $\bm{d}$. 

If $L=\Omega_C$, we can write 
$R_{Q_y}(\bm{d})=R(\bm{d}) \oplus R(\bm{d})^{\vee}$
for some $G(\bm{d})$-representation $R(\bm{d})$. 
Let $\mu$ be the moment map 
\begin{align*}
    \mu \colon 
    R_{Q_y}(\bm{d}) \to \mathfrak{g}(\bm{d})^{\vee},
\end{align*}
where $\mathfrak{g}(\bm{d})$ is the Lie algebra of $G(\bm{d})$. 
Then $\pi \colon \mathcal{M}(r, \chi)^{\rm{cl}}
\to M(r, \chi)$ is 
\'{e}tale locally at $y$ isomorphic to (\cite{Sacca}, \cite[Section 5]{Dav}, \cite[Section 4.2]{HalpK32}): 
\begin{align}\label{loc:P(d)}
    \mathscr{P}(d)^{\rm{cl}}:=\mu^{-1}(0)^{\rm{cl}}/G(\bm{d}) \to P(d):=\mu^{-1}(0)^{\rm{cl}}\ssslash G(\bm{d}),  
\end{align}
see~\cite[Section~4.3]{PTK3} for more details.

The moduli space $M^L(r, \chi)$ is stratified with strata indexed by the data $(d_i, r_i, \chi_i)_{i=1}^k$
of the polystable object (\ref{polystable}). 
The deepest stratum corresponds to 
$k=1$ with $(d_1, r_1, \chi_1)=(d, r_0, \chi_0)$, 
which consist of polystable objects
$V\otimes E_0$ where $\dim V=d$ and $E_0$ is stable 
with 
\begin{align*}(\rank(E_0), \chi(E_0))=(r_0, d_0).
\end{align*}
The associated Ext-quiver at the deepest 
stratum has one vertex and $(1+lr_0^2)$-loops. 
We have the following lemma, 
also see~\cite[Lemma~4.3]{PTK3} for an analogous
    statement for K3 surfaces. 

\begin{lemma}\emph{(\cite[Lemma~2.2]{PThiggs})}\label{lem:equiver}
    For each closed point $y \in M^L(r, \chi)$, 
    and a closed point 
    $x \in M^L(r, \chi)$ which lies in the deepest
    stratum, there exists a closed point 
    $y' \in M^L(r, \chi)$ which is sufficiently close to $x$
    such that $Q_y=Q_{y'}$. 
\end{lemma}

\subsection{Quasi-BPS categories}\label{subsec:qbps}
We consider the bounded derived category of coherent sheaves 
$D^b(\mathcal{M}^L(r, \chi))$. 
Note that there is an orthogonal decomposition
\[D^b(\mathcal{M}^L(r, \chi))=\bigoplus_{w\in\mathbb{Z}}D^b(\mathcal{M}^L(r, \chi))_w,\] where $D^b(\mathcal{M}^L(r, \chi))_w$ is the subcategory of $D^b(\mathcal{M}^L(r, \chi))$ of weight $w$ complexes with respect to the action of the scalar automorphisms $\mathbb{C}^*$
at each point of $\mathcal{M}^L(r, \chi)$.
In this subsection, 
we define the quasi-BPS categories, which are subcategories  
of $D^b(\mathcal{M}^L(r, \chi))_w$.

We first construct a line bundle on $\mathcal{M}^L(r, \chi)$. 
We write $(r, \chi)=d(r_0, \chi_0)$ as in (\ref{rchid}), 
and take $(a, b) \in \mathbb{Z}^2$ such that 
\begin{align}\label{cond:uv}
    a\chi_0+br_0+(1-g)ar_0=1
\end{align}
which is possible as $(r_0, \chi_0)$ are coprime. 
Let  
$u\in K(C)$ be such that $(\rank(u), \chi(u))=(a, b)$. 
Define the following line bundle on $\mathcal{M}^L(r, \chi)$
\begin{align}\label{delta}
    \delta:=\det(Rp_{\mathcal{M}\ast}(u\boxtimes \mathcal{F})) \in 
    \mathrm{Pic}(\mathcal{M}^L(r, \chi)), 
\end{align}
where $\mathcal{F}$ is the universal Higgs bundle (\ref{univ:F}) and $p_{\mathcal{M}}$ is the 
projection onto $\mathcal{M}$. It has diagonal $\mathbb{C}^{\ast}$-weight $d$ 
because of the condition (\ref{cond:uv})
and the Riemann-Roch theorem. 

Before defining quasi-BPS categories, we need to introduce some more notations.
An object $A \in D^b(B\mathbb{C}^{\ast})$ decomposes into a direct sum \[A=\bigoplus_{w\in \mathbb{Z}}A_w,\] where $A_{w}$ is of $\mathbb{C}^{\ast}$-weight
$w$. 
Denote by $\mathrm{wt}(A)$ the set of $w \in \mathbb{Z}$
such that $A_{w} \neq 0$. 
In the case that $A$ is a line bundle on $B\mathbb{C}^{\ast}$, then $\mathrm{wt}(A)$ 
consists of one element $\mathrm{wt}(A) \in \mathbb{Z}$. 
We also write $A^{>0}:=\bigoplus_{w>0}A_{w}$. 

\begin{defn}\label{def:qbps}
In the case of $l>2g-2$, define the 
quasi-BPS category 
\begin{align}\label{T:qbps}
    \mathbb{T}^L(r, \chi)_w \subset D^b(\mathcal{M}^L(r, \chi))_w. 
\end{align}
to be 
consisting of objects $\mathcal{E}$ such that, for any 
map $\nu \colon B\mathbb{C}^{\ast} \to \mathcal{M}=\mathcal{M}^L(r, \chi)$, we have 
\begin{align}\label{cond:qbps}
\mathrm{wt}(\nu^{\ast}\mathcal{E}) \subset 
\left[-\frac{1}{2}\mathrm{wt} \det ((\nu^{\ast}\mathbb{L}_{\mathcal{M}})^{>0}), 
\frac{1}{2}\mathrm{wt} \det ((\nu^{\ast}\mathbb{L}_{\mathcal{M}})^{>0})
\right]+\frac{w}{d}\mathrm{wt}(\nu^{\ast}\delta). 
\end{align}
\end{defn}

\begin{remark}\label{rmk:qbps:loc}
In the notation of (\ref{etslice}), let 
\begin{align}\label{qbps:loc}
    \mathbb{T}_{Q_y}(\X_{Q_y}(\bm{d}))_{w} \subset D^b(\X_{Q_y}(\bm{d}))
\end{align}
be generated by $\Gamma(\chi)\otimes \mathcal{O}_{\X_{Q_y}(\bm{d})}$,
where $\Gamma(\chi)$ is the irreducible $G(\bm{d})$-representation whose highest
weight $\chi$ is such that 
\begin{align*}
    \chi+\rho \in \frac{1}{2}\mathrm{sum}[0, \beta]+\frac{w}{d}\delta_y.
\end{align*}
Here, $M(\bm{d})$ is the weight lattice of the maximal torus $T(\bm{d}) \subset G(\bm{d})$, $\mathrm{sum}[0, \beta]$ is the Minkowski sum of weights in $R_{Q_y}(\bm{d})$, 
$\rho$ is the half the sum of positive roots, and $\delta_y:=\delta|_{y}$. 
The above category (\ref{qbps:loc}) is the quasi-BPS category of the 
quiver $Q_y$ and gives an \'{e}tale local model of $\mathbb{T}^L(r, \chi)_w$. 
For more details, see~\cite[Lemma~3.5, Remark~3.7]{PThiggs}. 
\end{remark}

We next recall the definition of (reduced or not) quasi-BPS categories in the case of $L=\Omega_C$. 
Below we fix $p\in C$. There is a closed embedding 
\begin{align}\label{emb:M}
j \colon  \mathcal{M}(r, \chi) \hookrightarrow 
    \mathcal{M}^{\Omega_C(p)}(r, \chi)
\end{align}
sending 
$(F, \theta)$ 
for $\theta \colon F \to F \otimes \Omega_C$ to 
$(F, \theta')$, where $\theta'$ is the composition 
\begin{align*}
\theta' \colon 
    F \stackrel{\theta}{\to} F \otimes \Omega_C \hookrightarrow 
    F \otimes \Omega_C(p). 
\end{align*}
Given $(F, \theta')$ in $\mathcal{M}^{\Omega_C(p)}(r, \chi)$, 
it comes from the image of (\ref{emb:M})
if and only if $\theta'|_{p} \colon F|_p \to F|_{p}\otimes \Omega_C(p)|_{p}$ is zero. 
Globally, let $(\mathcal{F}, \vartheta)$ be the universal 
Higgs bundle (\ref{univ:F}) for $L=\Omega_C(p)$, 
and set 
\begin{align*}
\mathcal{F}_p :=\mathcal{F}|_{p\times \mathcal{M}^{\Omega_C(p)}(r, \chi)} \in \mathrm{Coh}(\mathcal{M}^{\Omega_C(p)}(r, \chi)).
\end{align*}
By fixing an isomorphism $\Omega_C(p)|_{p} \cong \mathbb{C}$, 
the correspondence 
$(F, \theta') \mapsto (F|_{p}, \theta'|_{p})$ gives a section 
$s$ of the vector bundle 
\begin{align}\label{sec:s}
\xymatrix{
    \mathcal{V}:=\mathcal{E}nd(\mathcal{F}_p) \ar[r] & \ar@/_18pt/[l]^s \mathcal{M}^{\Omega_C(p)}(r, \chi). 
    }
\end{align}
Then we have an equivalence of derived stacks 
\begin{align}\label{equiv:stack}
    \mathcal{M}(r, \chi) \stackrel{\sim}{\to} s^{-1}(0),
\end{align}
where the right hand side is the derived zero locus of the section $s$.

\begin{defn}\label{def:qbps2}
Suppose that $L=\Omega_C$. We define the subcategory 
\begin{align}\label{def:qbps:Omega}
    \mathbb{T}(r, \chi)_w \subset D^b(\mathcal{M}(r, \chi))_w
\end{align}
to be consisting of
objects $\mathcal{E} \in D^b(\mathcal{M}(r, \chi))$ such 
that, 
for all $\nu \colon B\mathbb{C}^{\ast} \to \mathcal{M}(r, \chi)$, 
we have 
\begin{align}\label{cond:nu}
\mathrm{wt}(\nu^{\ast}j^{\ast}j_{\ast}\mathcal{E}) \subset 
\left[-\frac{1}{2}\mathrm{wt} \det (\nu^{\ast}\mathbb{L}_{\mathcal{V}})^{>0}, 
\frac{1}{2}\mathrm{wt} \det (\nu^{\ast}\mathbb{L}_{\mathcal{V}})^{>0}
\right]+\frac{w}{d}\mathrm{wt}(\nu^{\ast}\delta). 
\end{align}
Here, $j$ is the closed immersion (\ref{emb:M}). 
\end{defn}

The section $s$ is indeed a section $s_0$ of the subbundle 
$\mathcal{V}_0 \subset \mathcal{V}$
consisting of traceless endomorphisms. The reduced stack is the 
derived zero locus of $s_0$
\begin{align}\label{sec:s0}
    \mathcal{M}(r, \chi)^{\rm{red}} :=s_0^{-1}(0) \subset \mathcal{M}(r, \chi). 
\end{align}
Note that the classical truncations of $\mathcal{M}(r, \chi)^{\rm{red}}$ and 
$\mathcal{M}(r, \chi)$ are the same. 
The reduced quasi-BPS category 
\begin{align}\label{def:redqbps}
    \mathbb{T}(r, \chi)_w^{\rm{red}} \subset D^b(\mathcal{M}(r, \chi)^{\rm{red}})
\end{align}
is also similarly defined using the closed immersion 
$\mathcal{M}(r, \chi)^{\rm{red}} \hookrightarrow \mathcal{M}^{\Omega_C(p)}(r, \chi)$. 
\begin{remark}\label{rmk:open}
The construction of the category (\ref{T:qbps}) is local over $M^L=M^L(r, \chi)$. It follows that, for any open subset $U \subset M^L$,
there is an associated subcategory 
\begin{align*}
    \mathbb{T}^L(r, \chi)_{w}|_{U} \subset D^b(\mathcal{M}^L(r, \chi)\times_{M^L} U), 
\end{align*}
see~\cite[Remark~3.11]{PThiggs}.     
The same remark also applies to the reduced quasi-BPS 
categories. 
\end{remark}

We say that the tuple $(r, \chi, w)$ satisfies \textit{the BPS condition} if the 
tuple 
\begin{align*}
    (r, \chi, w+1-g^{\rm{sp}}) \in \mathbb{Z}^3
\end{align*}
is primitive, i.e. if $\gcd(r, \chi, w+1-g^{\rm{sp}})=1$. If $(r, \chi, w)$ satisfies the BPS condition,
we say that the category (\ref{def:qbps:Omega}), (\ref{def:redqbps}) is a \textit{(reduced or not) BPS category}. 
We now recall the basic properties of BPS categories. 
\begin{thm}\emph{(\cite[Theorem~1.2]{PThiggs})}\label{thm:intro1}
Suppose that the tuple $(r, \chi, w)$ satisfies the BPS condition. 

(1) If $l>2g-2$, then $\mathbb{T}^L(r, \chi)_w$ is a smooth dg-category over $\mathbb{C}$, 
which is proper and Calabi-Yau over $B^L$.

(2) If $L=\Omega_C$, then $\mathbb{T}(r, \chi)^{\rm{red}}_{w}$
is a smooth dg-category over $\mathbb{C}$, 
which is proper and Calabi-Yau over $B$.
\end{thm}

\subsection{The conjectural symmetry}
We now recall the main conjecture about the symmetry of BPS categories for $G=\mathrm{GL}(r)$ proposed in~\cite{PThiggs}. 

We denote by $(B^L)^{\rm{sm}} \subset B^L$
the open subset corresponding to smooth and irreducible spectral 
curves, and let
$\mathcal{E}^{\rm{sm}}$ be the restriction of $\mathcal{E}$ to $(B^L)^{\rm{sm}}$. 
We denote by 
\begin{align*}
    \mathcal{M}^L(r, \chi)^{\rm{sm}}:=
    \mathcal{M}^L(r, \chi)\times_{B^L} (B^{L})^{\rm{sm}}, \ 
    M^L(r, \chi)^{\rm{sm}}:=M^L(r, \chi)\times_{B^L} (B^L)^{\rm{sm}}. 
\end{align*}
Note that $\mathcal{M}^L(r, \chi)^{\rm{sm}}$ is the relative Picard stack of the family of smooth spectral curves $\mathcal{C}^{\mathrm{sm}}\to (B^L)^{\rm{sm}}$.
Let $\mathcal{P}^{\rm{sm}}$ be the Poincaré line bundle 
\begin{align*}
    \mathcal{P}^{\rm{sm}} \to \mathcal{M}^L(r, w+1-g^{\rm{sp}})^{\rm{sm}}\times_{(B^L)^{\rm{sm}}}
    \mathcal{M}^L(r, \chi)^{\rm{sm}}
\end{align*}
defined by 
\begin{align}\label{line:P22}
    \mathcal{P}^{\rm{sm}}:=\det Rp_{13\ast}(p_{12}^{\ast}\mathcal{E}' \otimes
    p_{23}^{\ast}\mathcal{E}) &\otimes 
 \det Rp_{13\ast}(p_{12}^{\ast}\mathcal{E}')^{-1} \\ \notag 
 &\otimes 
  \det Rp_{13\ast}(p_{23}^{\ast}\mathcal{E})^{-1} 
  \otimes \det Rp_{13\ast}\mathcal{O},
\end{align}
where $p_{ij}$ are the projections from 
\begin{align*}
    \mathcal{M}^{L}(r, w+1-g^{\rm{sp}})^{\rm{sm}}\times_{(B^L)^{\rm{sm}}} \times \mathcal{C}^{\rm{sm}} \times_{(B^L)^{\rm{sm}}}
\times\mathcal{M}^L(r, \chi)^{\rm{sm}}
\end{align*}
onto the corresponding factors, 
and 
$\mathcal{E}'$ is the universal sheaf on the product
$\mathcal{C}\times_{B^L} \mathcal{M}^{L}(r, w+1-g^{\rm{sp}})$. 
It determines the Fourier-Mukai equivalence~\cite{Mu1, DoPa}:
\begin{align}\label{equiv:family}
 \Phi_{\mathcal{P}^{\rm{sm}}} \colon   D^b(\mathcal{M}^L(r, w+1-g^{\rm{sp}})^{\rm{sm}})_{-\chi+1-g^{\rm{sp}}}\stackrel{\sim}{\to} D^b(\mathcal{M}^L(r, \chi)^{\rm{sm}})_w.
    \end{align}
    For $l>2g-2$, we rewrite the equivalence above as the following equivalence:
\begin{align}\label{equiv:Bsm}
    \mathbb{T}^L(r, w+1-g^{\rm{sp}})_{-\chi+1-g^{\rm{sp}}}|_{(B^L)^{\rm{sm}}}
\stackrel{\sim}{\to} 
    \mathbb{T}^L(r, \chi)_w|_{(B^L)^{\rm{sm}}}. 
\end{align}
The following is the main conjecture in~\cite{PThiggs}, which says that the above equivalence extends over the full Hitchin base $B^L$:
\begin{conj}\emph{(\cite[Conjecture~4.3]{PThiggs})}\label{conj:T0}
Suppose that $l>2g-2$ and that the tuple $(r, \chi, w)$ satisfies the BPS condition. 
Then there is a $B^L$-linear equivalence 
\begin{align}\label{equiv:conjT}
\mathbb{T}^L(r, w+1-g^{\rm{sp}})_{-\chi+1-g^{\rm{sp}}}
\stackrel{\sim}{\to} 
    \mathbb{T}^L(r, \chi)_w 
\end{align}
which extends the equivalence (\ref{equiv:Bsm}), i.e. it commutes with 
the restriction functors to $(B^L)^{\rm{sm}}$. 
\end{conj}
In the case of $L=\Omega_C$, we propose the following 
conjecture: 
\begin{conj}\emph{(\cite[Conjecture~4.6]{PThiggs})}
Suppose that the tuple $(r, \chi, w)$ is primitive. 
Then there is a $B$-linear equivalence 
\begin{align}\label{mirrorBPS}
    \mathbb{T}(r, w)_{-\chi}^{\rm{red}} \stackrel{\sim}{\to} 
    \mathbb{T}(r, \chi)_w^{\rm{red}}
\end{align}
which extends the equivalence (\ref{equiv:Bsm}). 
    \end{conj}

Note that, when $w$ and $\chi$ are both coprime with $r$ and $G=\mathrm{GL}(r)$, the equivalence \eqref{mirrorBPS} recovers the (still conjectural) equivalence \eqref{intro:mirror}.

\subsection{Semiorthogonal decompositions for moduli of Higgs bundles}\label{subsec:sod}
We write $(r, \chi)=d(r_0, \chi_0)$ for $d \in \mathbb{Z}_{>0}$ and 
$(r_0, \chi_0)$ coprime. 
For a decomposition 
$d=d_1+\cdots+d_k$, 
let $\mathcal{F}il^L(d_1, \ldots, d_k)$ be the 
(derived) moduli stack of filtrations
of $L$-twisted Higgs bundles
\begin{align}\label{higgs:filt}
    0=E_0 \subset E_1 \subset \cdots \subset E_k,
\end{align}
where $E_i/E_{i-1}$ has $(\rank(E_i/E_{i-1}), \chi(E_{i}/E_{i-1}))=d_i(r_0, \chi_0)$. 
There are natural evaluation morphisms 
\begin{align*}
    \times_{i=1}^k \mathcal{M}^L(d_i(r_0, \chi_0)) \stackrel{q}{\leftarrow} 
    \mathcal{F}il^L(d_1, \ldots, d_k) \stackrel{p}{\to} \mathcal{M}^L(r, \chi), 
\end{align*}
where $p$ sends a filtration (\ref{higgs:filt}) to $E_k$
and $q$ sends (\ref{higgs:filt}) to $(E_i/E_{i-1})_{i=1}^k$. 
The categorical Hall product is defined by (see \cite{PoSa}): 
\begin{align}\label{cathall}
\ast:=Rp_{\ast}Lq^{\ast} \colon 
\boxtimes_{i=1}^k
D^b(\mathcal{M}^L(d_i(r_0, \chi_0)))   
\to D^b(\mathcal{M}^L(r, \chi)). 
\end{align}
We have the following semiorthogonal decomposition:
\begin{thm}\emph{(\cite[Theorem 3.12]{PThiggs})}\label{thm:sod}
For each $w\in \mathbb{Z}$, there is a semiorthogonal decomposition 
\begin{align}\label{sod:main}
D^b(\mathcal{M}^L(r, \chi))_w=
    \left\langle 
\boxtimes_{i=1}^k \mathbb{T}^L(d_i r_0, d_i \chi_0)_{w_i}
\,\Big|\, \frac{v_1}{d_1}<\cdots<\frac{v_k}{d_k} \right\rangle,
\end{align}
where the right hand side is after all partitions
$(d,w)=(d_1,w_1)+\cdots+(d_k,w_k)$ and where $v_i \in \frac{1}{2}\mathbb{Z}$ is given by 
\begin{align}\label{def:wi}
 v_i=
     w_i-\frac{lr_0^2}{2}d_i\left( \sum_{i>j}d_j-\sum_{i<j}d_j \right).
\end{align}
The fully-faithful functor 
\begin{align*}
\boxtimes_{i=1}^k \mathbb{T}^L(d_i r_0, d_i \chi_0)_{w_i}
\to D^b(\mathcal{M}^L(r, \chi))
\end{align*}
is given by the restriction of the categorical Hall product (\ref{cathall}). 
\end{thm}

\subsection{Moduli spaces of Joyce-Song pairs}
In this subsection, we mention a framed version of Theorem~\ref{thm:sod}. 
We first discuss a version of Joyce-Song (JS) stable pairs~\cite{JS} in the context of Higgs bundles. 
Fix an ample line bundle $\mathcal{O}_C(1)$ 
on $C$ of degree one and let $m\gg 0$. 
\begin{defn}\label{def:JSpair}
A tuple $(F, \theta, s)$ is called a \textit{JS pair} if 
$(F, \theta)$
is a semistable $L$-twisted Higgs bundle and 
$s \in H^0(F(m))$ is such that, for any 
surjection $j \colon (F, \theta) \twoheadrightarrow (F', \theta')$ of 
$L$-twisted Higgs bundle with $\mu(F)=\mu(F')$, 
we have $0\neq j \circ s \in H^0(F'(m))$. 
    \end{defn}
    We denote by 
    \begin{align}\label{JS:moduli}
        \mathcal{M}^{L\dag}=\mathcal{M}^L(r, \chi)^{\rm{JS}}
    \end{align}
    the moduli space of JS pairs $(F, \theta, s)$
    such that $(\mathrm{rank}(F), \chi(F))=(r, \chi)$. 
    It is a quasi-projective scheme with morphisms:
    \begin{align}\label{mor:phi}
    \phi \colon 
        \mathcal{M}^{L\dag} \to \mathcal{M}^L \to M^L,
    \end{align}
    where $\mathcal{M}^L=\mathcal{M}^L(r, \chi)$ and
    $M^L=M^L(r, \chi)$, see~\cite[Section~12.1]{JS}. The first morphism is 
    given by forgetting $s$ which is a smooth morphism, 
    the second one is a good moduli space map, 
    and the composition $\phi$ is a projective morphism. 
    In particular, $\mathcal{M}^{L\dag}$ is smooth for $l>2g-2$. 

    Let $(r, \chi)=d(r_0, \chi_0)$ with coprime $(r_0, \chi_0)$, 
    and set 
    \begin{align}\label{N:E0}
        N:=\chi(E_0(m))=mr_0+\chi_0+r_0(1-g)
    \end{align}
    where $E_0$ is a vector bundle with $(\mathrm{rank}(E_0), \chi(E_0))=(r_0, \chi_0)$. 
Let $y \in M^L$ be the point corresponding to a polystable object (\ref{polystable}), 
and let $Q_y$ be its associated Ext-quiver. 
The morphism (\ref{mor:phi}) is étale locally on $M$ at $y$ described 
as follows. 
Let $Q_y^{\dag}$ be the quiver framed quiver of the Ext-quiver $Q_y$, with vertices $\{0, 1, \ldots, k\}$ and with 
\begin{align*}
    \sharp(0 \to i)=\dim h^0(E_i(m))=Nm_i,
\end{align*}
edges between the added vertex $\{0\}$ and a vertex $i$ of $Q_y$, and
where $m_i$ is such that $(\mathrm{rank}(E_i), \chi(E_i))=m_i(r_0, d_0)$. 
Let $d_i=\dim V_i$ and $\bm{d}=(d_i)_{i=1}^k$. 
The space of $Q_y^{\dag}$-representations of dimension 
vector $(1, \bm{d})$ is given by 
\begin{align*}
R_{Q_y^{\dag}}(\bm{d})=\bigoplus_{i=1}^k V_i^{\oplus Nm_i} \oplus R_{Q_y}(\bm{d}). 
\end{align*}
Then the morphism (\ref{mor:phi}) is étale locally at $y$ isomorphic to 
\begin{align*}
  \mathcal{X}_{Q_y}^{\dag}(\bm{d}):=
  R_{Q_y^{\dag}}(\bm{d})^{\rm{ss}}/G(\bm{d}) \to R_{Q_y}(\bm{d})/G(\bm{d}) \to 
    R_{Q_y}(\bm{d})\ssslash G(\bm{d}). 
\end{align*}
Here, the semistable locus is with respect to the character 
$G(\bm{d}) \to \mathbb{C}^{\ast}$ given by 
$(g_i)_{i=1}^k \mapsto \prod_{i=1}^k \det g_i$. 
The semistable locus consists of $Q_y^{\dag}$-representations 
which are generated by the images of the edges from the vertex $\{0\}$, 
see~\cite[Lemma~6.1.9]{T}. 
    
    The following is a framed version of Theorem~\ref{thm:sod}. 
    \begin{thm}\label{thm:sod:JS}
    Suppose that $l>2g-2$. 
    There is a semiorthogonal decomposition 
        \begin{align}\label{SODJS}
            D^b(\mathcal{M}^{L}(r, \chi)^{\rm{JS}})=        
 \left\langle 
\boxtimes_{i=1}^k \mathbb{T}^L(d_i r_0, d_i \chi_0)_{w_i}
\,\Big|\, 0\leq \frac{v_1}{d_1}<\cdots<\frac{v_k}{d_k}<N \right\rangle,
\end{align}
where the right hand side is after all partitions
$d=d_1+\cdots+d_k$ and all $(w_i)_{i=1}^k\in \mathbb{Z}^k$, and where $v_i \in \frac{1}{2}\mathbb{Z}$ is given by 
\begin{align}\label{def:wi2}
 v_i=
     w_i-\frac{lr_0^2}{2}d_i\left( \sum_{i>j}d_j-\sum_{i<j}d_j \right).
\end{align}        
    \end{thm}
    \begin{proof}
    Let $y \in M^L$ lie in the deepest stratum, and let $Q_y^{\dag}$ be the 
    framed quiver as above. 
    There is a semiorthogonal decomposition, 
    see~\cite[Theorem~4.18]{PTquiver}:
    \begin{align*}
        D^b(\mathcal{X}_{Q_y^{\dag}}(d))=\left\langle\boxtimes_{i=1}^k 
        \mathbb{T}_{Q_y}(d_i)_{w_i} \,\Big|\, 
        0\leq \frac{v_1}{d_1}<\cdots<\frac{v_k}{d_k}<N \right\rangle,
    \end{align*}
    where $v_i$ is given by (\ref{def:wi2}). 
    Similarly to the proof of Theorem~\ref{thm:sod} 
    in~\cite[Theorem~3.12]{PThiggs} (also see~\cite[Theorem~5.1]{PTK3}), we can reduce the proof of the semiorthogonal decomposition \eqref{SODJS} to
    the above local statement at $y$. 
    \end{proof}

\section{Topological K-theory}

In this section, we review the topological K-theory 
of dg-categories due to Blanc \cite{Blanc}
and its relative version due to Moulinos~\cite{Moulinos}. 
We also prove some technical lemmas which will be used later. 

\subsection{Topological K-theory}\label{subsec:topK:review}
We use the notion of \textit{spectrum}, which is an object 
representing a generalized cohomology theory and plays a central role in stable homotopy theory. 
Basic references are~\cite{Adams2, May}. 
A spectrum consists of a sequence of pointed spaces $E=\{E_n\}_{n\in \mathbb{N}}$ 
together with maps $\Sigma E_n \to E_{n+1}$ where $\Sigma$ is the suspension 
functor. 
There is a notion of homotopy groups of a spectrum 
denoted by $\pi_{\bullet}(E)$.
The $\infty$-category of spectra is 
constructed in~\cite[Section~1.4]{LSAG} and 
denoted by $\mathrm{Sp}$.
A map of spectra 
$E \to F$ 
in $\mathrm{Sp}$
is an \textit{equivalence} if it induces
isomorphisms on homotopy groups.

For a topological space $X$, we denote by $K^{\rm{top}}_{i}(X)$ its $i$-th
topological K-group.
The K-groups $K^{\rm{top}}_{\bullet}(X)$ 
are obtained as homotopy groups of a spectrum: 
there is a spectrum $KU=\{BU\times \mathbb{Z}, U, \ldots\}$, 
called \textit{the topological K-theory spectrum}, 
such that 
we have (see~\cite[Sections~22-24]{May})
\begin{align*}
	K_{\bullet}^{\rm{top}}(X)=\pi_{\bullet}KU_X, \ 
	KU_X:=[\Sigma^{\infty}X, KU]. 
\end{align*}
Here $[-, KU]$ is the mapping spectrum with target $KU$
and $\Sigma^{\infty}(-)$ is the infinite suspension functor. 
We call $K^{\rm{top}}(X)$ the \textit{topological K-theory spectrum of }$X$. 

For a dg-category $\mathscr{D}$,
Blanc \cite{Blanc}
defined the
\textit{topological K-theory spectrum
	of $\mathscr{D}$}, denoted by
\begin{align}\label{def:topK:dg}K^{\mathrm{top}}(\mathscr{D}) \in \mathrm{Sp}.
\end{align}
If $\mathcal{D}=\mathrm{Perf}(X)$ for a separated $\mathbb{C}$-scheme $X$ of finite type, 
then Blanc's topological K-theory spectrum is 
$K^{\rm{top}}(\mathcal{D}) \simeq KU_X$. 
For $i\in \mathbb{Z}$, consider its $i$-th rational homotopy group, which is a $\mathbb{Q}$-vector space:
\begin{align*}K^{\mathrm{top}}_i(\mathscr{D})_{\mathbb{Q}}:=K^{\mathrm{top}}_i(\mathscr{D})\otimes_\mathbb{Z}\mathbb{Q}, \ 
K^{\mathrm{top}}_{i}(\mathscr{D}):=\pi_i(K^{\mathrm{top}}(\mathscr{D})).
\end{align*}
There are isomorphisms $K^{\mathrm{top}}_i(\mathscr{D})\cong K^{\mathrm{top}}_{i+2}(\mathscr{D})$ for every $i\in \mathbb{Z}$ obtained by multiplication with a Bott element, see \cite[Definition 1.6]{Blanc}.
The topological K-theory spectrum sends exact triangles of dg-categories to exact triangles of spectra \cite[Theorem 1.1(c)]{Blanc}.

\subsection{Relative topological K-theory}
In this subsection, we review the relative version of topological
K-theory developed in~\cite{Moulinos}. 
Below we use the notation in~\cite[Section~2]{GS}
for sheaves of spectra. 

For a complex variety $M$ and an $\infty$-category with arbitrary 
small limits, we denote by 
$\mathrm{Sh}_{\mathcal{C}}(M^{\rm{an}})$ the 
$\mathcal{C}$-valued hypersheaves on $M^{\rm{an}}$
as in~\cite[Section~1.3.1]{LSAG}, 
where $M^{\rm{an}}$ is the underlying complex analytic 
space. 
There is a rationalization functor, see~\cite[Definition~2.6]{GS}:
\begin{align*}
	\mathrm{Rat} \colon \mathrm{Sh}_{\mathrm{Sp}}(M^{\rm{an}}) \to 
	\mathrm{Sh}_{D(\mathbb{Q})}(M^{\rm{an}})=D(\mathrm{Sh}_{\mathbb{Q}}(M^{\rm{an}}))
\end{align*}
given by $\mathcal{F} \mapsto \mathcal{F} \wedge H\mathbb{Q}$. 
In the above, $H\mathbb{Q}$ is the Eilenberg-Maclane spectrum
of $\mathbb{Q}$ and 
the right hand side is the derived category of sheaves of $\mathbb{Q}$-vector 
spaces on $M^{\rm{an}}$. 
We write $\mathcal{F}_{\mathbb{Q}} :=\mathrm{Rat}(F)_{\mathbb{Q}}$. 
It satisfies 
$\pi_{\bullet}(\mathcal{F}_{\mathbb{Q}})=\pi_{\bullet}(\mathcal{F})\otimes \mathbb{Q}$. 

By~\cite[Lemma~2.7]{GS}, 
the sheaf of topological K-theory spectra $\underline{KU}_M$ on $M^{\rm{an}}$ satisfies
\begin{align}\label{isom:KU}
	(\underline{KU}_{M})_{\mathbb{Q}} \cong \mathbb{Q}_M[\beta^{\pm 1}]
	=\bigoplus_{n\in \mathbb{Z}} \mathbb{Q}_M[2n],
\end{align}
where $\beta$ is of degree $2$. 
The above isomorphism follows from the degeneration of 
the Atiyah--Hirzebruch spectral sequence.

Let 
$\mathrm{Cat}^{\rm{perf}}(M)$ be
the $\infty$-category of $\mathrm{Perf}(M)$-linear stable 
$\infty$-categories. 
By \cite{Moulinos}, there is a functor 
\begin{align*}
	\mathcal{K}_{M}^{\rm{top}} \colon 
	\mathrm{Cat}^{\rm{perf}}(M) \to \mathrm{Sh}_{\mathrm{Sp}}(M^{\rm{an}})  
\end{align*}
such that, for a Brauer class $\alpha$ on $M$, 
we have 
\begin{align}\label{isom:ktopalpha}
	\mathcal{K}_M^{\rm{top}}(\mathrm{Perf}(M, \alpha))=\underline{KU}^{\hat{\alpha}}_M. 
\end{align}
In the above, 
$\mathrm{Perf}(M, \alpha)$ is the dg-category of $\alpha$-twisted 
perfect complexes, 
$\hat{\alpha} \in H^3(M, \mathbb{Z})$ is the class associated with $\alpha$
by the natural map $\mathrm{Br}(M) \to H^3(M, \mathbb{Z})$, 
and the right hand side is the sheaf of $\hat{\alpha}$-twisted topological K-theory spectrum.
For $M=\Spec \mathbb{C}$, 
it 
agrees with the spectrum (\ref{def:topK:dg}):
$\mathcal{K}_{\mathrm{Spec}\mathbb{C}}^{\rm{top}}(\mathscr{D})=K^{\rm{top}}(\mathscr{D})$. 

\subsection{Topological G-theory}\label{subsec:topG}
For a complex variety $M$, 
there is an embedding $\mathrm{Perf}(M) \subset D^b(M)$ which 
is not equivalence unless $M$ is smooth. 
A version of topological K-theory for $D^b(M)$ is 
called \textit{topological G-theory}. It was 
introduced by Thomason in~\cite{ThomasonAS}, 
and we denote it by $G^{\rm{top}}_{\bullet}(M)$. 

There is also its spectrum 
version $KU_{M, c}^{\vee}$, 
called \textit{locally compact supported K-homology}, see~\cite[Lemma~2.6]{HLP}. 
In~\cite[Theorem~2.10]{HLP}, it is proved 
that there is an equivalence 
\begin{align}\label{equiv:csupp}
	K^{\rm{top}}(D^b(M)) \simeq KU_{M, c}^{\vee}. 
\end{align}
We denote by $\underline{KU}_{M, c}^{\vee}$ the sheaf of locally 
compact supported K-homology. 
If $M$ is smooth, then 
we have $\underline{KU}_{M, c}^{\vee} \cong \underline{KU}_M$. 
The following is the sheaf version of the equivalence (\ref{equiv:csupp}):
\begin{lemma}\emph{(\cite[Lemma~4.1]{PTtop})}\label{lem:KBM}
	For a quasi-projective scheme $M$, 
	there is a natural equivalence 
	\begin{align}\label{equiv:singular}
		\mathcal{K}_M^{\rm{top}}(D^b(M)) \stackrel{\sim}{\to} \underline{KU}_{M, c}^{\vee}. 
	\end{align}
	In particular, there is an equivalence 
			$\mathcal{K}_M^{\rm{top}}(D^b(M))_{\mathbb{Q}} \simeq \omega_M[\beta^{\pm 1}]$, 
   where $\omega_M=\mathbb{D} \mathbb{Q}_M$ is the dualizing 
complex. 
		\end{lemma}

\subsection{Push-forward of relative topological K-theories}
We will use the following property of 
topological K-theory under proper push-forward. 

\begin{thm}\emph{(\cite[Proposition~7.8]{Moulinos}, \cite[Theorem~2.12]{GS})}\label{thm:phiproper}
	Let $\phi \colon M \to M'$ be a proper morphism.
	Then for $\mathscr{D} \in \mathrm{Cat}^{\rm{perf}}(M)$, we have 
	\begin{align*}
		\mathcal{K}_{M'}^{\rm{top}}(\phi_{\ast}\mathscr{D}) \cong \phi_{\ast}\mathcal{K}_M^{\rm{top}}(\mathscr{D}). 
	\end{align*}
	Here, $\phi_{\ast}\mathscr{D}$ is the category $\mathscr{D}$
	with $\mathrm{Perf}(M')$-linear structure induced by the pullback
	$\phi^{\ast} \colon \mathrm{Perf}(M') \to \mathrm{Perf}(M)$. 
\end{thm}
Below we often write $\mathcal{K}_{M'}^{\rm{top}}(\phi_{\ast}\mathscr{D})$ as 
$\mathcal{K}_{M'}^{\rm{top}}(\mathscr{D})$ when $\phi$ is clear from the context. 
The following is a version of Theorem~\ref{thm:phiproper}
for open immersions.

\begin{lemma}
\emph{(\cite[Lemma~4.3]{PTtop})}\label{lem:openimm}
	Let $j \colon U \subset M$ be an open immersion. 
	Then there is a natural equivalence 
	\begin{align}\label{equiv:open}
		\mathcal{K}_M^{\rm{top}}(j_{\ast}\mathrm{Perf}(U)) \stackrel{\sim}{\to} 
		j_{\ast}\mathcal{K}_U^{\rm{top}}(\mathrm{Perf}(U)). 
	\end{align}
\end{lemma}

\subsection{Relative topological K-theories of semiorthogonal summands}
We also need a version of Theorem~\ref{thm:phiproper} for global sections over non-proper 
schemes. We prove it for semiorthogonal 
summands of $\mathrm{Perf}(M)$, $D^b(M)$, or 
of categories of matrix factorizations. 
We first state the following lemma: 

\begin{lemma}
\emph{(\cite[Lemma~4.4]{PTtop})}\label{lem:gsection}
	Let $B$ be a quasi-projective scheme. 
	Then, for $\phi \colon B \to \Spec \mathbb{C}$
	and $\mathscr{D} \in \mathrm{Cart}^{\rm{perf}}(B)$, there is a natural 
	morphism 
	\begin{align}\label{nat:eta}
		\eta \colon K^{\rm{top}}(\mathscr{D}) \to \phi_{\ast}\mathcal{K}_{B}^{\rm{top}}(\mathscr{D}). 
	\end{align}
\end{lemma}

We will use the following results: 

\begin{lemma}
\emph{(\cite[Lemma~4.5]{PTtop})}\label{lem:sod:Ktheory}
Let $\mathscr{D}=\langle \mathcal{C}_1, \mathcal{C}_2\rangle$
be a $\mathrm{Perf}(M)$-linear semiorthogonal decomposition. 
Then we have 
\begin{align}\label{splitting}
    \mathcal{K}_M^{\rm{top}}(\mathscr{D}) =\mathcal{K}_M^{\rm{top}}(\mathcal{C}_1) \oplus \mathcal{K}_M^{\rm{top}}(\mathcal{C}_2).
\end{align}  
\end{lemma}

\begin{prop}
\emph{(\cite[Proposition~4.6]{PTtop})}\label{prop:gsection}
Let $h \colon M \to B$ be a proper morphism. 
	Let $\mathscr{D}$ be either $\mathrm{Perf}(M)$, $D^b(M)$ 
	or $\mathrm{MF}(M, f)$ for a non-zero function $f$ on $M$, 
	where we assume that $M$ is smooth in the last case. 
	Let $\mathscr{D}=\langle \mathcal{C}_1, \mathcal{C}_2 \rangle$ be a 
	$\mathrm{Perf}(B)$-linear 
 semiorthogonal decomposition. 
	Then the natural maps 
	\begin{align*}
		\eta_i \colon K^{\rm{top}}(\mathcal{C}_i) \to \phi_{\ast}\mathcal{K}_B^{\rm{top}}(\mathcal{C}_i)
	\end{align*}
	in Lemma~\ref{lem:gsection} are equivalences. 
\end{prop}

\subsection{Pull-back of topological K-theory}
Let $\pi \colon \mathcal{M}^{\dag} \to M$ be a proper morphism of quasi-projective 
schemes, and let $g \colon M \to \Delta$ be a flat morphism 
such that the composition $g^{\dag} \colon \mathcal{M}^{\dag} \to M \to \Delta$ is also flat. 
For a closed point $0 \in \Delta$, let 
\[M_0:=g^{-1}(0), \mathcal{M}_0^{\dag}:=(g^{\dag})^{-1}(0),\] so that 
we have the Cartesian square 
\begin{align*}
    \xymatrix{
\mathcal{M}_0^{\dag} \inclusion^-{i^{\dag}} \ar[d]_-{\pi_0} & 
\mathcal{M}^{\dag} \ar[d]^-{\pi} \\
M_0 \inclusion^-{i} & M.     
    }
\end{align*}
Let $D^b(\mathcal{M}^{\dag})=\langle \mathcal{C}', \mathcal{C} \rangle$ be a 
$\mathrm{Perf}(M)$-linear semiorthogonal decomposition, which is strong 
in the sense that $\mathcal{C}$ is both left and right admissible in $D^b(\mathcal{M}^{\dag})$. 
Then there is an induced semiorthogonal decomposition, see~\cite[Theorem~5.6]{MR2801403}:
\begin{align}\label{sod:M0dag}
    D^b(\mathcal{M}_0^{\dag})=\langle \mathcal{C}_0', \mathcal{C}_0 \rangle.
\end{align}

In the next lemma, $i^{-1}$ is the left adjoint to the functor $i_* \colon D(\mathrm{Sh}_{\mathbb{Q}}(M_0)) \to D(\mathrm{Sh}_{\mathbb{Q}}(M))$.

\begin{lemma}\label{lem:Ktop:pullback}
In the above setting, there is a natural isomorphism
\begin{align*}
    i^{-1}\mathcal{K}_M^{\rm{top}}(\mathcal{C})_{\mathbb{Q}}
    \stackrel{\cong}{\to} \mathcal{K}_{M_0}^{\rm{top}}(\mathcal{C}_0)_{\mathbb{Q}}. 
\end{align*}
\end{lemma}
\begin{proof}
The morphism $i^{\dag}$ is quasi-smooth, so 
the pull-back $(i^{\dag})^{\ast}$ gives a functor 
$D^b(\mathcal{M}^{\dag}) \to D^b(\mathcal{M}_0^{\dag})$. 
The above functor 
restricts to a  
$\mathrm{Perf}(M)$-linear functor 
$i^{\dag\ast} \colon 
    \mathcal{C} \to \mathcal{C}_0$, see~\cite[Theorem~5.6]{MR2801403}. 
   Therefore 
    there is an induced morphism 
    \begin{align*}
        \mathcal{K}_M^{\rm{top}}(\mathcal{C})_{\mathbb{Q}} \to 
        \mathcal{K}_M^{\rm{top}}(\mathcal{C}_0)_{\mathbb{Q}} \simeq i_{\ast}\mathcal{K}_{M_0}^{\rm{top}}(\mathcal{C}_0)_{\mathbb{Q}}. 
    \end{align*}
    Here, the last equivalence follows from Theorem~\ref{thm:phiproper}. 
    By adjunction, we obtain a morphism 
    \begin{align*}
    \eta_{\mathcal{C}} \colon 
        i^{-1}\mathcal{K}_M^{\rm{top}}(\mathcal{C})_{\mathbb{Q}}
    \to \mathcal{K}_{M_0}^{\rm{top}}(\mathcal{C}_0)_{\mathbb{Q}}. 
    \end{align*}
    The above construction applied for $\mathcal{C}=D^b(\mathcal{M}^{\dag})$
    gives a morphism 
    \begin{align}\label{mor:perf}
        i^{-1}\mathcal{K}_M^{\rm{top}}(D^b(\mathcal{M}^{\dag}))_{\mathbb{Q}} \to 
        \mathcal{K}_{M_0}^{\rm{top}}(D^b(\mathcal{M}_0^{\dag}))_{\mathbb{Q}}. 
    \end{align}
    The above morphism is an equivalence. 
       Indeed, we have 
    \begin{align*}
        \mathcal{K}_M^{\rm{top}}(D^b(\mathcal{M}^{\dag}))_{\mathbb{Q}}=\pi_{\ast}\omega_{\mathcal{M}^{\dag}}[\beta^{\pm 1}], \ 
        \mathcal{K}_{M_0}^{\rm{top}}(D^b(\mathcal{M}_0^{\dag}))_{\mathbb{Q}}=\pi_{0\ast}\omega_{\mathcal{M}_0^{\dag}}[\beta^{\pm 1}]
    \end{align*}
    by Theorem~\ref{thm:phiproper} and Lemma~\ref{lem:KBM}, and then 
    the equivalence of (\ref{mor:perf}) 
    follows from the proper base change theorem. 
    From the $\mathrm{Perf}(M)$-linear semiorthogonal decompositions (\ref{sod:M0dag}), 
    the morphism (\ref{mor:perf}) is identified with 
    \begin{align*}
        \eta_{\mathcal{C}} \oplus 
        \eta_{\mathcal{C}'} \colon 
         i^{-1}\mathcal{K}_M^{\rm{top}}(\mathcal{C})_{\mathbb{Q}}\oplus i^{-1}\mathcal{K}_M^{\rm{top}}(\mathcal{C}')_{\mathbb{Q}}
    \to \mathcal{K}_{M_0}^{\rm{top}}(\mathcal{C}_0)_{\mathbb{Q}} \oplus \mathcal{K}_{M_0}^{\rm{top}}(\mathcal{C}_0')_{\mathbb{Q}}.
    \end{align*}
    Therefore $\eta_{\mathcal{C}}$ is an equivalence, as desired. 
\end{proof}

\section{Topological K-theory of BPS categories: the case G=GL and $l>2g-2$}\label{sec4}
In this section, we prove part (1) of Theorem~\ref{thm:intro} after rationalization. The main ingredient is Proposition \ref{prop:topK}, where we compute the (relative) topological K-theory of BPS categories in terms of the intersection complex of the good moduli space. To show part (1) of Theorem~\ref{thm:intro}, we use an extension or Arinkin's sheaf over the full Hitchin base. We also discuss computations of topological K-theory of quasi-BPS categories  beyond the BPS condition, and of \v{S}penko--Van den Bergh noncommutative resolutions for the moduli of stable vector bundles on a curve.

\subsection{The comparison with intersection complexes}\label{subsec41}
Let $\mathcal{M}^L=\mathcal{M}^L(r, \chi)$ and 
$M^L=M^L(r, \chi)$ as in (\ref{mor:phi}). 
Let 
\begin{align*}
    \mathbb{T}^L=\mathbb{T}^L(r, \chi)_w \subset D^b(\mathcal{M}^L)
\end{align*}
be a quasi-BPS category. 
The above subcategory is closed under the action of $\mathrm{Perf}(M^L)$, and 
the semiorthogonal decomposition (\ref{sod:main}) is $\mathrm{Perf}(M^L)$-linear. 
In this subsection, we compute the topological K-theory of $\mathbb{T}^L$ in terms of the BPS cohomology of $M^L$. Note that the BPS cohomology for $M^L$ is the intersection cohomology $\mathrm{IH}^\star(M^L)$ as the category of semistable $L$-twisted Higgs bundles on $C$ has homological dimension one, see~\cite{Mein}.

\begin{prop}\label{prop:topK}
Suppose that $l>2g-2$ and that the tuple $(r, \chi, w)$ satisfies the BPS condition. 
There is an isomorphism in $D(\mathrm{Sh}_{\mathbb{Q}}(M^L))$:
\begin{align*}
    \mathcal{K}_{M^L}^{\rm{top}}(\mathbb{T}^L)_{\mathbb{Q}} \cong \mathrm{IC}_{M^L}[-\dim M^L][\beta^{\pm 1}]. 
\end{align*}    
\end{prop}
\begin{proof}
    Let $\phi$ be the morphism (\ref{mor:phi}). 
    Since $\phi$ is proper, there is the following commutative diagram, 
    see~\cite[Theorem~2.12]{GS}: 
    \begin{align*}
\xymatrix{
\mathrm{Cat}^{\mathrm{perf}}(\mathcal{M}^{L\dag})
\ar[r]^-{\mathcal{K}_{\mathcal{M}^{L\dag}}^{\rm{top}}} \ar[d]_-{\phi_{\ast}} & \mathrm{Sh}_{\mathrm{Sp}}(\mathcal{M}^{L\dag})
\ar[r]^-{\mathrm{Rat}} \ar[d]_-{\phi_{\ast}} & D(\mathrm{Sh}_{\mathbb{Q}}(\mathcal{M}^{L\dag}))
\ar[d]_-{\phi_{\ast}} \\
\mathrm{Cat}^{\mathrm{perf}}(\mathcal{M}^L)
\ar[r]^-{\mathcal{K}_{M^L}^{\rm{top}}} & \mathrm{Sh}_{\mathrm{Sp}}(\mathcal{M}^L)
\ar[r]^-{\mathrm{Rat}}  & D(\mathrm{Sh}_{\mathbb{Q}}(\mathcal{M}^L))
}
    \end{align*}
    Then we have 
    \begin{align*}
\phi_{\ast}\mathcal{K}^{\rm{top}}_{\mathcal{M}^{L\dag}}(D^b(\mathcal{M}^{L\dag}))_{\mathbb{Q}} \cong 
\mathcal{K}_{M^L}^{\rm{top}}(\phi_{\ast}D^b(\mathcal{M}^{L\dag}))_{\mathbb{Q}}.
    \end{align*}
By (\ref{isom:KU}), (\ref{isom:ktopalpha}), we have 
\begin{align*}
    \mathcal{K}^{\rm{top}}_{\mathcal{M}^{L\dag}}(D^b(\mathcal{M}^{L\dag}))_{\mathbb{Q}}
    \cong 
    (\underline{KU}_{\mathcal{M}^{L\dag}})_{\mathbb{Q}} 
    \cong \mathbb{Q}_{\mathcal{M}^{L\dag}}[\beta^{\pm 1}]. 
\end{align*}
It follows that we have 
\begin{align*}
    \phi_{\ast}\mathbb{Q}_{\mathcal{M}^{L\dag}}[\beta^{\pm 1}]\cong \mathcal{K}_{M^L}^{\rm{top}}(\phi_{\ast}D^b(\mathcal{M}^{L\dag}))_{\mathbb{Q}}. 
\end{align*}

We may assume that $0\leq w<d$. 
By Theorem~\ref{thm:sod:JS}, 
the subcategory $\mathbb{T}^L \subset D^b(\mathcal{M}^{L\dag})$
fits into a $\mathrm{Perf}(M^L)$-linear 
semiorthogonal decomposition. 
Therefore, by Lemma~\ref{lem:sod:Ktheory}, 
$\mathcal{K}_{M^L}^{\rm{top}}(\mathbb{T}^L)_{\mathbb{Q}}$ 
is a direct summand of 
$\mathcal{K}_{M^L}^{\rm{top}}(\phi_{\ast}D^b(\mathcal{M}^{L\dag}))_{\mathbb{Q}}$, 
hence a direct summand of $\phi_{\ast}\mathbb{Q}_{\mathcal{M}^{L\dag}}[\beta^{\pm 1}]$. 
By the BBDG decomposition theorem~\cite{BBD}, we have that:
\begin{align}\label{BBD}
    \phi_{\ast}\mathbb{Q}_{\mathcal{M}^{L\dag}}[\dim \mathcal{M}^{L\dag}] \cong \bigoplus_{p\in \mathbb{Z}} A_p[-p],
\end{align}
where $A_p \in \mathrm{Perv}(M^L)$ is a semisimple 
perverse sheaf. 
The morphism $\phi$ restricted to 
the stable locus $(M^L)^{\rm{st}} \subset M^L$
is a $\mathbb{P}^{dN-1}$-bundle, 
and the decomposition (\ref{BBD}) over $(M^L)^{\rm{st}}$ 
is 
\begin{align*}
    \phi_{\ast}\mathbb{Q}_{\mathcal{M}^{L\dag}}|_{(M^L)^{\rm{st}}}
    \cong \mathbb{Q}_{(M^L)^{\rm{st}}} \oplus 
    \mathbb{Q}_{(M^L)^{\rm{st}}}[-2] \oplus \cdots \oplus 
    \mathbb{Q}_{(M^L)^{\rm{st}}}[-2dN+2]. 
\end{align*}
It follows that 
\begin{align}\label{isom:Qbeta}
    \phi_{\ast}\mathbb{Q}_{\mathcal{M}^{L\dag}}[\beta^{\pm 1}]|_{(M^L)^{\rm{st}}}
    \cong \mathbb{Q}_{(M^L)^{\rm{st}}}[\beta^{\pm 1}]
    \oplus \cdots \oplus \mathbb{Q}_{(M^L)^{\rm{st}}}[\beta^{\pm 1}],
\end{align}
where there are $dN$-direct sums in the right hand side. 
On the other hand, the semiorthogonal decomposition 
in Theorem~\ref{thm:sod:JS} restricted to $(M^L)^{\rm{st}}$
is 
\begin{align*}
    D^b\Big(\phi^{-1}((M^L)^{\rm{st}})\Big)=
    \Big\langle \mathbb{T}^L(r, \chi)_0|_{(M^L)^{\rm{st}}}, 
    \ldots, 
    \mathbb{T}^L(r, \chi)_{dN-1}|_{(M^L)^{\rm{st}}} \Big\rangle. 
\end{align*}
Therefore, we have 
\begin{align*}\mathcal{K}_{M^L}^{\rm{top}}(D^b(\mathcal{M}^{L\dag}))_{\mathbb{Q}}|_{(M^L)^{\rm{st}}}
\cong \bigoplus_{i=0}^{dN-1}
\mathcal{K}_{M^L}^{\rm{top}}(\mathbb{T}^L(r, \chi)_i)_{\mathbb{Q}}|_{(M^L)^{\rm{st}}}.
\end{align*}
Note that each $\mathbb{T}^L(r, \chi)_i|_{(M^L)^{\rm{st}}}$ is 
étale locally on $(M^L)^{\rm{st}}$ independent of $i$ up to equivalence, since 
a choice of $i$ corresponds to a power of a Brauer class of $(M^L)^{\rm{st}}$. 
By comparing with (\ref{isom:Qbeta}), we conclude that 
\begin{align*}
    \mathcal{K}_{M^L}^{\rm{top}}(\mathbb{T}^L)_{\mathbb{Q}}|_{(M^L)^{\rm{st}}}\cong \mathbb{Q}_{(M^L)^{\rm{st}}}[\beta^{\pm 1}]. 
\end{align*}
As $\mathcal{K}_{M^L}^{\rm{top}}(\mathbb{T}^L)_{\mathbb{Q}}$
is a direct summand of (\ref{BBD}), we can write 
\begin{align*}
     \mathcal{K}_{M^L}^{\rm{top}}(\mathbb{T}^L)_{\mathbb{Q}}
     \cong \mathrm{IC}_{M^L}[-\dim M^L][\beta^{\pm 1}]
     \oplus P_0[\beta^{\pm 1}] \oplus 
     P_1[1][\beta^{\pm 1}],
\end{align*}
where $P_i \in \mathrm{Perv}(M^L)$ for $i=0, 1$
are semisimple perverse sheaves. 
It is enough to show that $P_0=P_1=0$. 

We take $y \in M^L$ and let $Q_y$ be the corresponding 
Ext-quiver with the associated dimension vector $\bm{d}$, 
the moduli stack $\mathcal{X}_{Q_y}(\bm{d})$
and its good moduli space $X_{Q_y}(\bm{d})$ as in (\ref{etslice}).  
Let 
\begin{align*}\mathbb{T}_{Q_y}(\bm{d})_w \subset D^b(\mathcal{X}_y(\bm{d}))
\end{align*}
be the quasi-BPS category for the quiver $Q_y$, 
see Remark~\ref{rmk:qbps:loc}. 
By~\cite[Theorem~8.26]{PTtop}, 
there is an isomorphism 
\begin{align*}
    \mathcal{K}_{X_{Q_y}(\bm{d})}(\mathbb{T}_{Q_y}(\bm{d})_w)_{\mathbb{Q}} \cong 
    \mathrm{IC}_{X_{Q_y}(\bm{d})}[-\dim X_{Q_y}(\bm{d})][\beta^{\pm 1}]. 
\end{align*}
As $\mathbb{T}_{Q_y}(\bm{d})_w$ gives an \'{e}tale local model 
for $\mathbb{T}^L$, 
we have $P_i=0$ for $i=0, 1$. 
\end{proof}

We remark that the method used in \cite[Proof of Theorem 8.14 assuming Theorem 8.15]{PTtop}
also applies to compute the topological K-theory of the quasi-BPS category $\mathbb{T}^L$ for all tuples $(r,\chi,w)$. We do not present full arguments as we will not use the following result in this paper. Further, we do not compute the topological K-theory of quasi-BPS categories for general tuples $(r,\chi,w)$ in the case $L=\Omega_C$, as the proof of Theorem \ref{prop:BPS} uses the support lemma~\cite[Theorem 6.6]{PTK3}, which holds for tuples satisfying the BPS condition, see the proof of Lemma~\ref{lem:support}.  

To state the result, we first need to introduce some notation. Write $(r,\chi)=d(r_0,\chi_0)$ for $(r_0,\chi_0)$ coprime. 
We recall the set $S^d_w$ of partitions $(d_i)_{i=1}^k$ of $d$ considered in~\cite[Section 8.1]{PTtop}, and computed in \cite[Proposition 8.5 and Lemma 8.6]{PTtop}. The set $S^d_w$ labels the summands of the topological K-theory of quasi-BPS categories for symmetric quivers, in this case for the quiver with one vertex and $e=1+lr_0^2$ loops. Note that 
this is the Ext-quiver for points in the deepest stratum on $M^L$, see Subsection \ref{subsec:loc}.

\begin{defn}\label{defsdw}
    Let $Q$ be the quiver with one vertex and $e$ loops. Let $d>0$ and $w\in \mathbb{Z}$. 
The set $S^d_w$ consists of all partitions $(d_i)_{i=1}^k$ of $d$ such that \[w_i:=\frac{e-1}{2}d_i\left(\sum_{j<i}d_j-\sum_{j>i}d_j\right)+\frac{wd_i}{d}\in\mathbb{Z}\] for all $1\leq i\leq k$. 
\end{defn}

Note that, if $e$ is odd, the condition above is that each $d_i$ is divisible by $d/\gcd(d,w)$. Thus $S^d_w$ is in a natural bijection with the set of partitions of $\gcd(d,w)$.

In the rest of this subsection, we assume $e=1+lr_0^2$. We will refer to elements of $S^d_w$ either as partitions $(d_i)_{i=1}^k$ or as tuples $(d_i,w_i)_{i=1}^k$.

We define the set $S(r,\chi,w)$ as follows. Let $m:=\gcd(r,\chi,w+1-g^{\mathrm{sp}})$ and let $(\widetilde{r}, \widetilde{\chi}, \widetilde{w}):=\frac{1}{m}(r,\chi,w+1-g^{\mathrm{sp}})$. 
Then $S(r,\chi,w)$ is the set of partitions $(r_i,\chi_i,w_i)_{i=1}^k$ of $(r,\chi, w+1-g^{\mathrm{sp}})$ such that $(r_i,\chi_i,w_i)=m_i(\widetilde{r}, \widetilde{\chi}, \widetilde{w})$ for some $m_i\in \mathbb{Z}_{>0}$. Thus there is a natural bijection between $S(r,\chi,w)$ and the set of partitions of $g$.
We note the following:
\begin{prop}\label{setofpartitions}
    There is a natural bijection of sets
\[S^d_w\xrightarrow{\cong} S(r,\chi,w),\,\, (d_i,w_i)_{i=1}^k \mapsto (d_ir_0, d_i\chi_0, \widetilde{w}_i)_{i=1}^k\] such that $r_0|w_i-\widetilde{w}_i$ for all $1\leq i\leq k$. In particular, if the vector $(r,\chi, w+1-g^{\mathrm{sp}})$ is primitive, then $S^d_w=\{(d)\}$. 
\end{prop}

\begin{proof}
By definition, the set $S^d_w$ contains partitions $(d_i)_{i=1}^k$ such that 
\[w_i:=\frac{lr_0^2}{2}d_i(d-d_i)+\frac{wd_i}{d}\in \mathbb{Z}.\]
Alternatively, we need to have that
\[\frac{lr_0^2(d-1)d_i}{2}+\frac{wd_i}{d}\in\mathbb{Z}, \text{ or }w^\circ_i:=\frac{(lr_0^2(d-1)d+2w)d_i}{2d}\in\mathbb{Z}.\] 
Note that $r_0|w_i-w^\circ_i$.
Further, we have that 
$2d| lr_0^2(d-1)d-lr_0d(r_0d-1)$. Thus, $S^d_w$ consists of partitions $(d_i)_{i=1}^k$ such that 
\[\widetilde{w}_i:=\frac{lr_0d(r_0d-1)+2w}{2d}d_i\in\mathbb{Z}.\] Note that $r_0|w_i-\widetilde{w}_i$ for all $1\leq i\leq k$.
By a direct computation, we have that \[m=\gcd(d,w+1-g^{\mathrm{sp}})=\gcd\left(d,w+\frac{lr_0d(r_0d-1)}{2}\right).\]
Then $S^d_w$ is in bijection with partitions of $\left(d,w+\frac{lr_0d(r_0d-1)}{2}\right)$ with all terms divisible by $\frac{1}{m}\left(d,w+\frac{lr_0d(r_0d-1)}{2}\right)$. This last set of partitions is in natural bijection with $S(r,\chi,w)$.
\end{proof}


We define some direct sums of IC (alternatively, BPS) sheaves 
associated with partitions of $d$. 
For a partition $A=(d_i)_{i=1}^k$ of $d$, its length is defined to be $\ell(A):=k$. Assume the set $\{d_1, \ldots, d_k\}=\{e_1,\ldots, e_s\}$ has cardinality $s$ and that, for each $1\leq i\leq s$, there are $m_i$ elements in $\{d_1, \ldots, d_k\}$ equal to $e_i$. 
We define the following maps, 
given by the direct sums of polystable Higgs bundles:
\begin{align*}
\oplus_i&\colon M^L(e_ir_0, e_i\chi_0)^{\times m_i}\to M^L(m_ie_ir_0, m_ie_i\chi_0), \\
\oplus'&\colon \times_{i=1}^s M^L(m_ie_ir_0,m_ie_i\chi_0)\to M^L.
\end{align*}
The above maps are finite maps.
We define the following perverse sheaves: 
\begin{align}\label{BPSAsheaf}
    \notag\mathrm{Sym}^{m_i}\big(\mathrm{IC}_{M^L(e_ir_0,e_i\chi_0)}\big)&:=\oplus_{i, \ast}\left(\mathrm{IC}_{M^L(e_ir_0,e_i\chi_0)}^{\boxtimes m_i}\right)^{\mathfrak{S}_{m_i}},\\
    \mathrm{IC}_{A}&:=
    \oplus'_{\ast}\left(\boxtimes_{i=1}^s \mathrm{Sym}^{m_i}
    \big(\mathrm{IC}_{M^L(e_ir_0,e_i\chi_0)}\big)\right).
\end{align}

\begin{defn}\label{def:BPS:sum}
For a tuple $(r,\chi,w)$ with $(r,\chi)=d(r_0,\chi_0)$ such that $(r_0,\chi_0)$ are coprime, define the following direct sum of symmetric 
products of BPS sheaves:
\begin{align}\label{defBPSddelta}
\mathrm{IC}^w_{M^L}&:=\bigoplus_{A\in S^d_w}\mathrm{IC}_A[-\ell(A)]\in D(\mathrm{Sh}_{\mathbb{Q}}(M^L)). 
\end{align}
\end{defn}

Then, as in \cite[Proof of Theorem 8.14 assuming Theorem 8.15]{PTtop}, to which we refer the reader for full details, one shows that:

\begin{prop}\label{prop:topKquasi}
Suppose that $l>2g-2$. 
There is an isomorphism in $D(\mathrm{Sh}_{\mathbb{Q}}(M^L))$:
\begin{align}\label{mapquasi}
   \mathrm{IC}^w_{M^L}[-\dim M^L][\beta^{\pm 1}]\cong  \mathcal{K}_{M^L}^{\rm{top}}(\mathbb{T}^L)_{\mathbb{Q}}. 
\end{align}    
\end{prop}

\begin{remark}\label{rem45}
We explain how one obtains, for $A\in S^d_w$, a map \begin{align}\label{mapA}
    \mathrm{IC}_A[-\ell(A)-\dim M^L][\beta^{\pm 1}]\oplus B=:\mathrm{IC}^\dagger_A\to \mathcal{K}_{M^L}^{\rm{top}}(\mathbb{T}^L)_{\mathbb{Q}},\end{align}
which induces one of summands of \eqref{mapquasi}. Here, $B$ is a direct sum of shifted perverse sheaves of support strictly contained in the support of $\mathrm{IC}_A$.

Let $A=(d_i,w_i)_{i=1}^k\in S^d_w$, and consider the corresponding tuple $(r_i,\chi_i,\widetilde{w}_i)_{i=1}^k\in S(r,\chi,w)$ from Proposition \ref{setofpartitions}. 
There are natural equivalences 
\begin{align}\label{naturalequiv}
    \mathbb{T}^L(r,\chi)_w\cong \mathbb{T}^L(r,\chi')_{w'}
\end{align} if $r|\chi'-\chi$ and $r|w'-w$.
The map \eqref{mapA} is induced by applying topological K-theory to the following functor, which is the composition of the Hall product with equivalences \eqref{naturalequiv}, see \cite[Proposition 8.2]{PTtop} and \cite[Proof of Theorem 8.14 assuming Theorem 8.15]{PTtop}):
\begin{align}\label{funA}
\bigotimes_{i=1}^k \mathbb{T}^L(r_i,\chi_i)_{\widetilde{w}_i}\cong \bigotimes_{i=1}^k \mathbb{T}^L(r_i,\chi_i)_{w_i}\to \mathbb{T}^L(r,\chi)_w.
\end{align}
\end{remark}

\subsection{Equivalences of rational topological K-theories for $l>2g-2$}\label{subsec42}
For simplicity, we write 
\begin{align*}\mathcal{M}^L=\mathcal{M}^L(r, \chi), \ 
\mathcal{M}^{L'}=\mathcal{M}^L(r, w+1-g^{\rm{sp}}).
\end{align*}
Let $(B^L)^{\rm{ell}} \subset B^L$ be the elliptic locus, 
i.e. the locus corresponding to reduced and irreducible spectral curves. 
We set 
\begin{align*}
    (\mathcal{M}^L)^{\rm{ell}}=\mathcal{M}^L\times_{B^L} (B^L)^{\rm{ell}}, 
    \ (\mathcal{M}^{L'})^{\rm{ell}}=\mathcal{M}^{L'}\times_{B^L} (B^L)^{\rm{ell}}. 
\end{align*}
Let $\mathcal{P}^{\rm{sm}}$ be the Poincare line 
bundle on $(\mathcal{M}^{L'})^{\rm{sm}} \times_{(B^L)^{\rm{sm}}}
(\mathcal{M}^L)^{\rm{sm}}$ as in (\ref{line:P22}).
By~\cite{Ardual}, it uniquely extends to a maximal 
Cohen-Macaulay sheaf 
\begin{align*}\mathcal{P}^{\rm{ell}}
\in \Coh((\mathcal{M}^{L'})^{\rm{ell}} \times_{(B^L)^{\rm{ell}}}
(\mathcal{M}^L)^{\rm{ell}})
\end{align*}
which induces an equivalence 
\begin{align}\label{FM:ell}
\Phi_{\mathcal{P}^{\rm{ell}}} \colon 
D^b((\mathcal{M}^{L'})^{\rm{ell}})_{-\chi-g^{\rm{sp}}+1} \stackrel{\sim}{\to} 
D^b((\mathcal{M}^L)^{\rm{ell}})_w. 
\end{align}
\begin{lemma}\label{lem:lift}
There is an object
\begin{align*}
    \mathcal{P} \in \mathbb{T}^L(r, w+1-g^{\rm{sp}})_{\chi+g^{\rm{sp}}-1} \boxtimes_{B^L}
    \mathbb{T}^L(r, \chi)_w \subset D^b(\mathcal{M}^{L'}\times_{B^L} \mathcal{M}^L)
\end{align*}
such that 
$\mathcal{P}|_{(B^L)^{\rm{ell}}} \cong \mathcal{P}^{\rm{ell}}$. 
\end{lemma}
\begin{proof}
  The restriction functor 
  \begin{align}\label{funct:rest}
      D^b(\mathcal{M}^{L'}\times_{B^L} \mathcal{M}^L) \to 
      D^b((\mathcal{M}^{L'})^{\rm{ell}} \times_{(B^L)^{\rm{ell}}}
      (\mathcal{M}^L)^{\rm{ell}})
  \end{align}
  is essentially surjective. 
  Let $\mathcal{P}'\in D^b(\mathcal{M}^{L'}\times_{B^L}\mathcal{M}^L)$ be a lift of $\mathcal{P}^{\rm{ell}}$. 
  From the semiorthogonal decomposition in Theorem~\ref{thm:sod}, 
  the subcategory 
  \begin{align}\label{Tboxtimes}
  \mathbb{T}^L(r, w+1-g^{\rm{sp}})_{\chi+g^{\rm{sp}}-1} \boxtimes_{B^L}
    \mathbb{T}^L(r, \chi)_w \subset D^b(\mathcal{M}^{L'}\times_{B^L} \mathcal{M}^L)
    \end{align}
    is a part of a semiorthogonal decomposition by~\cite[Theorem~5.8]{MR2801403}. 
    Its semiorthogonal complements are generated by 
    categorical Hall products, so they are sent to 
    zero under the functor (\ref{funct:rest}). 
    Therefore by taking the projection of $\mathcal{P}'$ 
    to the subcategory (\ref{Tboxtimes}), we obtain 
    a desired $\mathcal{P}$. 
\end{proof}

For an object $\mathcal{P}$ as in Lemma~\ref{lem:lift}, 
there is an induced functor 
\begin{align}\label{induce:P}
    \Phi_{\mathcal{P}} \colon 
    \mathbb{T}^L(r, w+1-g^{\rm{sp}})_{-\chi+1-g^{\rm{sp}}}
    \to \mathbb{T}^L(r, \chi)_w. 
\end{align}
In general we cannot expect the above Fourier-Mukai functor to 
be an equivalence, 
since there is an ambiguity in the choice of $\mathcal{P}$. However, following Groechenig--Shen~\cite{GS}, the functor induces an isomorphism in topological K-theory:
\begin{prop}\label{prop:induceK}
Suppose that $l>2g-2$. 
    The functor $\Phi_{\mathcal{P}}$ in (\ref{induce:P})
    induces a (rational) equivalence: 
    \begin{align}\label{equiv:topK}
        K^{\rm{top}}(\mathbb{T}^L(r, w+1-g^{\rm{sp}})_{-\chi-g^{\rm{sp}}+1})_{\mathbb{Q}}
        \stackrel{\sim}{\to} 
        K^{\rm{top}}(\mathbb{T}^L(r, \chi)_w)_{\mathbb{Q}}. 
    \end{align}
\end{prop}
\begin{proof}
    Note that we have the following diagram 
    \begin{align*}
        \xymatrix{
\mathcal{M}^{L'} \ar[r]^-{\pi'}  & M^{L'}\ar[rd]_-{h'} & & M^L 
\ar[ld]^-{h} & \mathcal{M}^L \ar[l]_-{\pi}  \\
& & B^L & &        
        }
    \end{align*}
    Here $\pi$, $\pi'$ are good moduli space 
    morphisms. 
We write $\mathbb{T}^L=\mathbb{T}^L(r, \chi)_w$
and $\mathbb{T}^{L'}=\mathbb{T}^L(r, w+1-g^{\rm{sp}})_{-\chi-g^{\rm{sp}}+1}$. 
The functor $\Phi_{\mathcal{P}}$ is linear over $\mathrm{Perf}(B^L)$, 
so it induces a morphism 
in $D(\mathrm{Sh}_{\mathbb{Q}}(B^L))$:
\begin{align}\label{induce:KB}
\Phi_{\mathcal{P}\mathbb{Q}}^K \colon 
    \mathcal{K}_{B^L}^{\rm{top}}(h'_{\ast}\mathbb{T}^{L'})_{\mathbb{Q}} \to 
    \mathcal{K}_{B^L}^{\rm{top}}(h_{\ast}\mathbb{T}^L)_{\mathbb{Q}}. 
\end{align}
It is enough to show that (\ref{induce:KB}) is an 
equivalence, as (\ref{equiv:topK}) 
is given by taking the global section of (\ref{induce:KB}), 
see Proposition~\ref{prop:gsection}. 

By Proposition~\ref{prop:topK},
we have 
\begin{align}\notag
\mathcal{K}_{B^L}^{\rm{top}}(h_{\ast}\mathbb{T}^L)_{\mathbb{Q}}
\cong h_{\ast}\mathcal{K}_{M^L}^{\rm{top}}(\mathbb{T}^L)_{\mathbb{Q}}
\cong h_{\ast}\mathrm{IC}_{M^L}[-\dim M^L][\beta^{\pm 1}]. 
\end{align}
Similarly, we have 
\begin{align*}
    \mathcal{K}_{B^L}^{\rm{top}}(h'_{\ast}\mathbb{T}^{L'})_{\mathbb{Q}}
\cong h'_{\ast}\mathrm{IC}_{M^{L'}}[-\dim M^{L'}][\beta^{\pm 1}]. 
\end{align*}
    On the other hand, as $h$ is proper, 
    we may apply the BBDG decomposition theorem~\cite{BBD} and write
    \begin{align*}        h_{\ast}\mathrm{IC}_{M^L}=\bigoplus_{i}\mathrm{IC}_{Z_i}(L_i)[-k_i],
    \end{align*}
    where $Z_i \subset B^L$ is an irreducible closed 
    subset and $L_i$ is a local system on a dense 
    open subset of $Z_i$. 
    By~\cite[Theorem~0.4]{DMJS}, each generic point of $Z_i$ 
    is contained in $(B^L)^{\rm{ell}}$. 
We have the same support property for $h'_{\ast}\mathrm{IC}_{M^{L'}}$. 
Therefore it is enough to check that (\ref{induce:KB}) 
is an equivalence on $(B^L)^{\rm{ell}} \subset B^L$. 

The restriction of (\ref{induce:KB}) to $(B^L)^{\rm{ell}}$
is 
\begin{align*}
\Phi_{\mathcal{P}^{\rm{ell}}}^K \colon 
 \mathcal{K}_{(B^L)^{\rm{ell}}}^{\rm{top}}(h_{\ast}\mathbb{T}^{L'}|_{(B^L)^{\rm{ell}}})_{\mathbb{Q}}
 \to \mathcal{K}_{(B^L)^{\rm{ell}}}^{\rm{top}}(h_{\ast}\mathbb{T}^L|_{(B^L)^{\rm{ell}}})_{\mathbb{Q}}. 
\end{align*}
The above map is an equivalence since 
\begin{align*}
\mathbb{T}^L|_{(B^L)^{\rm{ell}}}=
D^b((\mathcal{M}^L)^{\rm{ell}})_{w}, \ \mathbb{T}^{L'}|_{(B^L)^{\rm{ell}}}=D^b((\mathcal{M}^{L'})^{\rm{ell}})_{-\chi-g^{\rm{sp}}+1}
\end{align*}
and $\Phi_{\mathcal{P}^{\rm{ell}}}$
induces the equivalence (\ref{FM:ell}).     
\end{proof}

    A similar statement also holds for more general quasi-BPS categories. Namely, for any tuple $(r,\chi,w)$, there exists an object
    \[\mathcal{P}'\in \mathbb{T}^L(r, w+1-g^{\rm{sp}})_{\chi+g^{\rm{sp}}-1} \boxtimes_{B^L}
    \mathbb{T}^L(r, \chi)_w \subset D^b(\mathcal{M}^{L'}\times_{B^L} \mathcal{M}^L)\] which induces an equivalence of rational topological K-theory:
\begin{align}\label{isoquasi}
K^{\rm{top}}(\mathbb{T}^L(r, w+1-g^{\rm{sp}})_{-\chi-g^{\rm{sp}}+1})_{\mathbb{Q}}
        \stackrel{\sim}{\to} 
        K^{\rm{top}}(\mathbb{T}^L(r, \chi)_w)_{\mathbb{Q}}.
        \end{align}
We explain the construction of such an object. The proof that it induces the isomorphism \eqref{isoquasi} is analogous to the proof of Proposition \ref{prop:induceK} using the explicit form of the supports in the BBDG decomposition theorem.  

        First, using Proposition \ref{setofpartitions} one shows that there is a bijection of sets 
        \[S(r,\chi,w)\cong S(r,w+1-g^{\mathrm{sp}}, -\chi+1-g^{\mathrm{sp}}),\, A=(r_i,\chi_i,w_i)_{i=1}^k\mapsto A'=(r_i,w'_i,-\chi'_i)_{i=1}^k\] such that $r_i|\chi'_i-\chi_i$ and $r_i|w'_i-w_i$ for every $1\leq i\leq k$.
        Indeed, there is a bijection 
        \[S(r,\chi,w)\cong S(r,w+1-g^{\mathrm{sp}}, -\chi-1+g^{\mathrm{sp}}),\, (r_i,\chi_i,w_i)_{i=1}^k\mapsto (r_i,w_i,-\chi_i)_{i=1}^k.\] 
        We have $r|2(1-g^{\mathrm{sp}})$, so there is a bijection
        \[S(r,w+1-g^{\mathrm{sp}}, -\chi-1+g^{\mathrm{sp}})\cong S(r,w+1-g^{\mathrm{sp}}, -\chi+1-g^{\mathrm{sp}})\] which sends $(r_i,w_i,-\chi_i)_{i=1}^k\mapsto (r_i,w_i,-\chi'_i)_{i=1}^k$ such that $r_i|\chi'_i-\chi_i$ for all $1\leq i\leq k$.
        The functors \eqref{funA} for $A$ and $A'$ induce the summands $\mathrm{IC}^\dagger_A$ and $\mathrm{IC}^\dagger_{A'}$ (see \eqref{mapA}) for $\mathcal{K}^{\mathrm{top}}_{M^L}(\mathbb{T})_\mathbb{Q}$ 
        and $\mathcal{K}^{\mathrm{top}}_{M^{L'}}(\mathbb{T}')_\mathbb{Q}$, 
        respectively. 
      For $1\leq i\leq k$, consider the Poincaré line bundle (tensored with equivalences \ref{naturalequiv}):
      \[\mathcal{P}_i^{\mathrm{ell}}\in D^b\left(\mathcal{M}^L(r_i,\chi)^{\mathrm{ell}}\right)_{w_i}\otimes D^b\left(\mathcal{M}^L(r_i,w'_i)^{\mathrm{ell}}\right)_{-\chi'_i}.\]
      Then $\mathcal{P}_A$ induces a map:
      \[\Phi_{\mathcal{P}_A}\colon h_*\mathrm{IC}_A^\dagger\to h_*\mathrm{IC}_{A'}^\dagger\] which is an isomorphism onto the summands of largest support, see the argument in the proof of Proposition \ref{prop:induceK}: \[h_*\mathrm{IC}_A[\beta^{\pm 1}]\xrightarrow{\sim} h_*\mathrm{IC}_{A'}[\beta^{\pm 1}].\]
      Using the above observation and the $\chi$-independence phenomenon from \cite{DMJS}, we may choose $e_A\in \mathbb{N}$ inductively on $\ell(A)$ such that the kernel
        \[\mathcal{P}':=\bigoplus_{A\in S^d_w} \mathcal{P}_A^{\oplus e_A}\] induces the isomorphism \eqref{isoquasi}.

\subsection{Topological K-theory and the moduli of semistable vector bundles on a  curve}\label{subsec:bun}
Let $C$ be a smooth projective curve of genus $g\geq 1$. Consider the moduli stack of slope semistable vector bundles $\mathcal{B}(r,\chi):=\mathcal{B}\mathrm{un}(r,\chi)^{\mathrm{ss}}$ of rank $r$ and degree $\chi$ on the curve $C$ with good moduli space:
\[\mathcal{B}(r,\chi)\to \mathrm{B}(r,\chi).\]
Recall from \cite[Subsection 3.4]{PThiggs} the categories
\[\mathbb{B}(r,\chi)_w\subset D^b(\mathcal{B}(r,\chi))_w\]
which are defined similarly to Definition \ref{def:qbps}, see also \cite{SVdB, P3}. These categories are twisted non-commutative resolutions of singularities of $\mathrm{B}(r,\chi)$. There is a semiorthogonal decomposition of $D^b\left(\mathcal{B}(r,\chi)\right)_w$ in terms of Hall products of such categories analogous to the decomposition from Theorem \ref{sod:main}, see \cite[Theorem 3.17]{PThiggs}.

In this subsection, we mention a computation of the (rational) topological K-theory of $\mathbb{B}(r,\chi)_w$. This computation is not used later in the paper, but it may be of independent interest, and it complements the discussion in \cite[Subsection 3.4]{PThiggs}.
Recall from loc.cit. that the good moduli space $\mathrm{B}(r,\chi)$ has a stratification as in Subsection \ref{subsec:loc}. Write $(r,\chi)=d(r_0,\chi_0)$ for $(r_0,\chi_0)$ coprime.
The deepest stratum corresponds to vector bundles $V\otimes E_0$, where $E_0$ is a vector bundle of rank $r_0$ and Euler characteristic $\chi_0$ and $V$ is a vector space of dimension $d$. The Ext-quiver corresponding to such a point has one vertex and $r_0^2(g-1)+1$ loops. 
Consider the set of partitions $S^d_w$ for the quiver with one vertex and $e=r_0^2(g-1)+1$ loops, see Definition \ref{defsdw}. For $A\in S^d_w$, define $\mathrm{IC}_A$ as in \eqref{BPSAsheaf}, and then define $\mathrm{IC}^w_{\mathrm{Bun}(r,\chi)^{\mathrm{ss}}}$ as in \eqref{defBPSddelta}.
The following is proved as \cite[Proof of Theorem 8.14 assuming Theorem 8.15]{PTtop}:

\begin{prop}
    There is an isomorphism in $D(\mathrm{Sh}_{\mathbb{Q}}(\mathrm{B}(r,\chi)))$:
\[  \mathrm{IC}^w_{\mathrm{B}(r,\chi)}[-\dim \mathrm{B}(r,\chi)][\beta^{\pm 1}]\cong  \mathcal{K}_{\mathrm{B}(r,\chi)}^{\rm{top}}(\mathbb{B}(r,\chi)_w)_{\mathbb{Q}}. 
\]  
\end{prop}

\section{Torsion freeness of topological K-theories of quasi-BPS categories for $l>2g-2$}

We continue the discussion from Section \ref{sec4}, in particular we continue to assume that $G=\mathrm{GL}(r)$ and $l>2g-2$.
In this section, we prove, following Groechenig--Shen~\cite{GS}, that there is an equivalence (\ref{equiv:topK}) without rationalization.
This claim follows from the torsion 
freeness of $K^{\rm{top}}(\mathbb{T}^L(r, \chi)_w)$. 
Following the argument in~\cite{GS}, 
we use localization with respect to the $\mathbb{C}^{\ast}$-action on Higgs bundles which scales the Higgs field, where the fixed 
part corresponds to moduli stacks of chains on the curve. 

\subsection{The abelian category of chains}\label{subsec51}
The moduli stack $\mathcal{M}^L(r, \chi)$ 
admits a $\mathbb{C}^{\ast}$-action given by 
$t \cdot (F, \theta)=(F, t\theta)$ for $t\in \mathbb{C}^{\ast}$. 
A $\mathbb{C}^{\ast}$-fixed $L$-twisted Higgs bundle 
corresponds to 
a chain, see~\cite[Lemma~9.2]{HauTha}:
\begin{align}\label{chain}
    \mathcal{E}_0 \stackrel{\phi_1}{\to} \mathcal{E}_1 \to \cdots \stackrel{\phi_k}{\to}
    \mathcal{E}_k. 
\end{align}
Here, each $\mathcal{E}_i$ is a vector bundle on 
$C$ and $\phi_i\colon \mathcal{E}_{i-1} \to \mathcal{E}_i$ 
is a morphism of coherent sheaves, such that 
$(r_{\bullet}, \chi_{\bullet})=\{(r_i, \chi_i)\}_{0\leq i\leq k}$ for $(r_i, \chi_i)=(\mathrm{rank}(\mathcal{E}_i), \chi(\mathcal{E}_i))$
satisfies 
\begin{align}\label{rchi:sum}
    (r_0, \chi_0)+\cdots+(r_k, \chi_k)=(r, \chi+lk(k+1)/2). 
\end{align}
Given a chain (\ref{chain}), there is a corresponding 
$L$-twisted Higgs bundle: 
\begin{align}\label{Higgs:L}
    \mathcal{E}=\bigoplus_{i=0}^k \mathcal{E}_i \otimes L^{-i}
\end{align}
with Higgs field $\phi \colon \mathcal{E} \to \mathcal{E} \otimes L$ 
naturally induced by $\phi_i$.

We denote by 
$\mathcal{A}_k$ the abelian category of chains (\ref{chain}) such that 
each $\mathcal{E}_i$ is a coherent sheaf on $C$. 
 \begin{lemma}\label{lem:formchi}
 For $\mathcal{E}_{\bullet}, \mathcal{E}_{\bullet}'\in \mathcal{A}_k$, we have 
 $\Ext^{\geq 3}(\mathcal{E}_{\bullet}, \mathcal{E}_{\bullet}')=0$. 
 Moreover, by setting $(\mathrm{rank}(\mathcal{E}_i), \chi(\mathcal{E}_i))=(r_i, \chi_i)$, 
    $(\mathrm{rank}(\mathcal{E}_i'), \chi(\mathcal{E}_i'))=(r_i', \chi_i')$,
    and  
    \begin{align*}
        \chi((r_{\bullet}, \chi_{\bullet}), (r_{\bullet}', \chi_{\bullet}')):=
        \sum_{i=0}^{k-1} \left(r_i(\chi_i'-\chi_{i+1}')-\chi_i(r_i'-r_{i+1}')\right),
    \end{align*}
    we have
    \begin{align}\label{dim:formula0}
    \sum_{i\geq 0} (-1)^i \dim \Ext^i(\mathcal{E}_{\bullet}, \mathcal{E}_{\bullet}')
  =\chi((r_{\bullet}, \chi_{\bullet}), (r_{\bullet}', \chi_{\bullet}')).
    \end{align}
    \end{lemma}
    \begin{proof}
         The lemma follows from the fact that 
    $R\Hom(\mathcal{E}_{\bullet}, \mathcal{E}_{\bullet}')$ is computed 
    by the hypercohomology of the following complex, see~\cite[Proposition~4.4]{PHS}:
    \begin{align*}
        \bigoplus_{i}\mathcal{H}om(\mathcal{E}_i, \mathcal{E}_i') \to 
        \bigoplus_{i}\mathcal{H}om(\mathcal{E}_i, \mathcal{E}_{i+1}'). 
    \end{align*}
    \end{proof}

For $\alpha_{\bullet}=(\alpha_i)_{0\leq i\leq k}$ with $\alpha_i \in \mathbb{R}$, 
define the slope of a chain (\ref{chain}) to be
\begin{align*}
    \mu_{\alpha_{\bullet}}(\mathcal{E}_{\bullet}):=\frac{\sum_{i=0}^k (\chi_i+\alpha_i r_i)}{\sum_{i=0}^k r_i}.
\end{align*}
There is a notion of $\mu_{\alpha_{\bullet}}$-stability on 
$\mathcal{A}_k$: a chain (\ref{chain}) is $\mu_{\alpha_{\bullet}}$-(semi)stable
if we have  
\begin{align*}
    \mu_{\alpha_{\bullet}}(\mathcal{E}_{\bullet}') <(\leq)
    \mu_{\alpha_{\bullet}}(\mathcal{E}_{\bullet})
\end{align*}
for any non-zero subobject $\mathcal{E}_{\bullet}' \subsetneq \mathcal{E}_{\bullet}$ 
in $\mathcal{A}_k$. 
The $\mu_{\alpha_{\bullet}}$-stability corresponds to the stability of 
the $L$-twisted Higgs bundle (\ref{Higgs:L}) when $\alpha_i=-il$, as in this case 
$\mu_{\alpha}(\mathcal{E}_{\bullet})$ is the usual slope
$\chi/r$ of (\ref{Higgs:L}).

\begin{lemma}\label{lem:Ext:chain}
    Suppose that $\alpha_i-\alpha_{i+1}>2g-2$. 
    Then, for any $\mu_{\alpha_{\bullet}}$-semistable 
    $\mathcal{E}_{\bullet}$, $\mathcal{E}_{\bullet}'$
    with $\mu_{\alpha_{\bullet}}(\mathcal{E}_{\bullet})\geq\mu_{\alpha_{\bullet}}(\mathcal{E}_{\bullet}')$, we have 
    $\Ext^{\geq 2}(\mathcal{E}_{\bullet}, \mathcal{E}_{\bullet}')=0$.  
\end{lemma}
\begin{proof}
    The lemma follows from the 
    argument of~\cite[Lemma~4.6]{PHS}.    
\end{proof}

\subsection{Moduli stacks of chains}
Let 
$\mathcal{C}_{(r_{\bullet}, \chi_{\bullet})}^{\alpha_{\bullet}}$
be the moduli stack of $\mu_{\alpha_{\bullet}}$-semistable chains (\ref{chain})
such that $(\mathrm{rank}(\mathcal{E}_i), \chi(\mathcal{E}_i))=(r_i, \chi_i)$. 
Denote by 
\begin{align}\label{open:circ}
\mathcal{C}_{(r_{\bullet}, \chi_{\bullet})}^{\alpha_{\bullet} \circ}
\subset \mathcal{C}_{(r_{\bullet}, \chi_{\bullet})}^{\alpha_{\bullet}}
\end{align}
the open substack such that $\phi_i \neq 0$ when $r_{i-1} r_i \neq 0$. 
By~\cite[Lemma~9.2]{HauTha}, the $\mathbb{C}^{\ast}$-fixed stack of 
$\mathcal{M}^L(r, \chi)$ is given by 
\begin{align}\label{fixed:higgs}
\coprod_{(r_0, \chi_0)+\cdots+(r_k, \chi_k)=(r, \chi+lk(k+1)/2)}
\mathcal{C}_{(r_{\bullet}, \chi_{\bullet})}^{\alpha_{\bullet}\circ}, \text{ for } 
\alpha_i=-il, 0\leq i\leq k. 
    \end{align}

For $\mathcal{F} \in \Coh(C)$, we 
use the same symbol $\mathcal{F} \in \mathcal{A}_k$ to denote 
the constant chain 
\begin{align*}
    \mathcal{F} \stackrel{\id}{\to} \mathcal{F} \stackrel{\id}{\to} \cdots\stackrel{\id}{\to}
    \mathcal{F}. 
\end{align*}
\begin{defn}\label{def:JSpair:chain}
For $m \gg 0$, a pair 
\begin{align}\label{JS:chain}
    (\mathcal{E}_{\bullet}, s), \ \mathcal{E}_{\bullet} \in \mathcal{A}_k, \ 
    s \colon \mathcal{O}_C(-m) \to \mathcal{E}_{\bullet}
\end{align}
is called \textit {a JS (Joyce--Song) $\alpha_{\bullet}$-stable pair} 
if $\mathcal{E}_{\bullet}$ is 
$\mu_{\alpha_{\bullet}}$-semistable 
and, for any surjection 
$j \colon \mathcal{E}_{\bullet} \twoheadrightarrow \mathcal{E}_{\bullet}'$ in $\mathcal{A}_k$
with $\mu_{\alpha_{\bullet}}(\mathcal{E}_{\bullet}')=\mu_{\alpha_{\bullet}}(\mathcal{E}_{\bullet})$, we have 
$j \circ s \neq 0$. 
\end{defn}

Let $\mathcal{A}_k^{\dag}$ be the abelian category of pairs 
\begin{align*}
    W \otimes \mathcal{O}_C(-m) \to \mathcal{E}_{\bullet},
\end{align*}
where $W$ is a finite dimensional vector space and 
$\mathcal{E}_{\bullet} \in \mathcal{A}_k$. 
Note that $\mathcal{A}_k \subset \mathcal{A}_k^{\dag}$ is an abelian 
subcategory by regarding $\mathcal{E}_{\bullet}$ as a pair $(0 \to \mathcal{E}_{\bullet})$. 
We denote by $\mathcal{C}_{(r_{\bullet}, \chi_{\bullet})}^{\alpha_{\bullet}\dag}$
the moduli space of JS $\alpha_{\bullet}$-stable pairs 
(\ref{JS:chain})
such that $(\mathrm{rank}(\mathcal{E}_{i}), \chi(\mathcal{E}_i))=(r_i, \chi_i)$. 
The natural projection 
\begin{align*}
    \mathcal{C}_{(r_{\bullet}, \chi_{\bullet})}^{\alpha_{\bullet}\dag}
    \to \mathcal{C}_{(r_{\bullet}, \chi_{\bullet})}^{\alpha_{\bullet}}
\end{align*}
is smooth with image contained in the open substack (\ref{open:circ})
by the JS stability.

    \begin{lemma}\label{lem:ExtJS}
    Suppose that $\alpha_i-\alpha_{i+1} \geq 2g-2$. 
    Then, for any $\mu_{\alpha_{\bullet}}$-semistable $\mathcal{E}_{\bullet} \in \mathcal{A}_k$
    and any
    JS $\alpha_{\bullet}$-stable pair $I=(\mathcal{O}_C(-m) \to \mathcal{E}_{\bullet}')$ 
    in $\mathcal{A}_k^{\dag}$ 
       with $\mu_{\alpha_{\bullet}}(\mathcal{E}_{\bullet}) =\mu_{\alpha_{\bullet}}(\mathcal{E}_{\bullet}')$, we have 
    \[\Ext^{\geq 2}(\mathcal{E}_{\bullet}, I)=0.\] 
        Moreover, 
    we have
    \begin{align}\label{dim:formula}
    \dim \Hom(\mathcal{E}_{\bullet}, I)
    -\dim \Ext^1(\mathcal{E}_{\bullet}, I)=\chi((r_{\bullet}, \chi_{\bullet}), (r_{\bullet}', \chi_{\bullet}'))
    \end{align}
    where $(\mathrm{rank}(\mathcal{E}_i), \chi(\mathcal{E}_i))=(r_i, \chi_i)$, 
    $(\mathrm{rank}(\mathcal{E}_i'), \chi(\mathcal{E}_i'))=(r_i', \chi_i')$. 
\end{lemma}
\begin{proof}
    The complex $R\Hom(\mathcal{E}_{\bullet}, I)$ is given 
    by the cone of the map 
    \begin{align*}
        R\Hom(\mathcal{E}_{\bullet}, \mathcal{O}_C(-m)) \to 
        R\Hom(\mathcal{E}_{\bullet}, \mathcal{E}_{\bullet}'). 
    \end{align*}
    Therefore, by Lemma~\ref{lem:formchi}, we get 
    $\Ext^{\geq 3}(\mathcal{E}_{\bullet}, I)=0$ and 
    we obtain the exact sequence 
    \begin{align}\label{exact:JSI}
        \Ext^2(\mathcal{E}_{\bullet}, \mathcal{O}_C(-m)) \to 
        \Ext^2(\mathcal{E}_{\bullet}, \mathcal{E}_{\bullet}') \to 
        \Ext^2(\mathcal{E}_{\bullet}, I)\to 0. 
    \end{align}
    The vanishing $\Ext^2(\mathcal{E}_{\bullet}, I)=0$
    follows when $\alpha_i-\alpha_{i+1}>2g-2$ by Lemma~\ref{lem:Ext:chain}. 
    Suppose that $\alpha_i-\alpha_{i+1}=2g-2$. 
    The dual of the first map in (\ref{exact:JSI}) is given by (see~\cite[Lemma~4.5]{PHS}):
    \begin{align}\label{map:dual}
        \Hom_{\mathcal{A}_{k+1}}(\mathcal{E}_{\bullet}', \mathcal{E}_{\bullet+1} \otimes \Omega_C) \to \Hom_{\mathcal{A}_{k+1}}(\mathcal{O}_C(-m), 
        \mathcal{E}_{\bullet+1} \otimes \Omega_C). 
    \end{align}
    Here $\mathcal{E}_{\bullet}'$ and $\mathcal{E}_{\bullet+1}$ are 
    regarded as objects in $\mathcal{A}_{k+1}$ by 
\begin{align*}
    \mathcal{E}_{\bullet}'=(\mathcal{E}_0' \to \cdots \to \mathcal{E}_{k}' \to 0), \ 
    \mathcal{E}_{\bullet+1}=(0 \to \mathcal{E}_0 \to \cdots \to \mathcal{E}_k). 
\end{align*}
    The object $\mathcal{E}_{\bullet+1} \otimes \Omega_C$ is 
    $\mu_{\alpha_{\bullet}}$-semistable whose slope 
    is the same as $\mu_{\alpha_{\bullet}}(\mathcal{E}_{\bullet}')$. 
    Therefore the map (\ref{map:dual}) is injective by the definition of JS stability 
    of $I$, hence $\Ext^2(\mathcal{E}_{\bullet}, I)=0$. 
    The formula (\ref{dim:formula}) follows from the above vanishing 
    of $\Ext^{\geq 2}(\mathcal{E}_{\bullet}, I)$ together with the equality (\ref{dim:formula0}). 
\end{proof}

Let $\mathcal{M}^{L\dag}$ be the moduli space of JS stable Higgs bundles as in (\ref{JS:moduli}). 
The $\mathbb{C}^{\ast}$-action on $L$-twisted Higgs bundles naturally lifts 
to the action on $\mathcal{M}^{L\dag}$ by $t \cdot (F, \theta, s)=(F, t\theta, s)$ for $t\in \mathbb{C}^{\ast}$. 
Similarly to (\ref{fixed:higgs}), 
the $\mathbb{C}^{\ast}$-fixed locus of 
$\mathcal{M}^{L\dag}$ is 
\begin{align}\label{fixed:higgs2}
(\mathcal{M}^{L\dag})^{\mathbb{C}^{\ast}}=
\coprod_{(r_0, \chi_0)+\cdots+(r_k, \chi_k)=(r, \chi+lk(k+1)/2)}
\mathcal{C}_{(r_{\bullet}, \chi_{\bullet})}^{\alpha_{\bullet}\dag}, \ 
\alpha_i=-il. 
    \end{align}
    Note that, for $l>2g-2$,
    each $\mathcal{C}_{(r_{\bullet}, \chi_{\bullet})}^{\alpha_{\bullet}\dag}$
    is a smooth projective variety since 
    it is a component of a $\mathbb{C}^{\ast}$-fixed locus of 
    the smooth quasi-projective variety $\mathcal{M}^{L\dag}$ and it is 
    supported on the fiber at $0 \in B$.

\subsection{The Grothendieck ring of stacks}
We will study the class of $\mathcal{C}_{(r_{\bullet}, \chi_{\bullet})}^{\alpha_{\bullet}, \dag}$
in the Grothendieck ring of varieties to show 
the torsion freeness of topological K-theory. In this subsection and in the next subsection, we discuss some terminology 
and lemmas about Grothendieck rings of varieties, stacks, and motivic Hall algebras. These are basic tools in 
the wall-crossing arguments of Donaldson-Thomas theory, 
and we refer the reader to~\cite{BrI} for an introduction. 

For an Artin stack $\mathcal{S}$ over $\mathbb{C}$, we denote by 
$K(\mathrm{St}/\mathcal{S})$ the Grothendieck 
ring of stacks over $\mathcal{S}$. Its underlying $\mathbb{Q}$-vector 
space is generated by symbols
\begin{align*}
    [\rho \colon \mathcal{X} \to \mathcal{S}]
\end{align*}
where $\mathcal{X}$ is an Artin stack of finite type over $\mathbb{C}$ with affine 
geometric stabilizers, and these symbols satisfy certain motivic relations, 
see~\cite[Definition~3.10]{BrI} for its precise definition. 
We write $K(\mathrm{St}):=K(\mathrm{St}/\Spec \mathbb{C})$. 

Let $K(\mathrm{Var}) \subset K(\mathrm{St})$ be the subspace 
spanned by the class of varieties. 
We also denote by $\mathbb{L} \in K(\mathrm{Var})$ the class of 
the affine line $\mathbb{A}^1$. It is proved in~\cite[Lemma~3.9]{BrI}
that we have the identity 
\begin{align}\label{K:stack}
    K(\mathrm{St})=K(\mathrm{Var})\left[\frac{1}{\mathbb{L}}, \frac{1}{\mathbb{L}-1}, 
    \frac{1}{\mathbb{L}^n+\cdots+\mathbb{L}+1} \,\Big|\, n\geq 1\right]. 
\end{align}
Let $\widehat{K}(\mathrm{Var})$ be the dimensional completion of 
$K(\mathrm{Var})[\mathbb{L}^{-1}]$. 
By expanding the denominators in the right hand side of (\ref{K:stack}) in terms of 
$\mathbb{L}^{-1}$, we obtain 
a map 
\begin{align*}
    K(\mathrm{St}) \to \widehat{K}(\mathrm{Var}). 
\end{align*}

Let $K$ be a field. For a smooth projective variety $X$ over $\mathbb{C}$, set
\begin{align*}
    P_K(X, t):=\sum_{i} \dim_K (H^i(X, K))t^i \in \mathbb{Z}[t]. 
\end{align*}
There exists an extension of $P_K(X, t)$ for any complex algebraic variety $Y$
which satisfies the relation 
\begin{align*}
    P_K(Y, t)=P_K(Y\setminus Z, t)+P_K(Z, t)
\end{align*}
for any closed subvariety $Z\subset Y$. 
The correspondence $Y \mapsto P_K(Y, t)$ induces the map, see~\cite[Section~6]{GS}: 
\begin{align*}
    P_K \colon \widehat{K}(\mathrm{Var}) \to \mathbb{Q}[t][[t^{-1}]]. 
\end{align*}

\subsection{The motivic Hall algebra}
Recall the abelian category $\mathcal{A}_k$ from Subsection \ref{subsec51}.
Let $\mathcal{O}bj(\mathcal{A}_k)$ be the moduli 
stack of objects in $\mathcal{A}_k$. 
We set 
\begin{align*}
H(\mathcal{A}_k):=
    K(\mathrm{St}/\mathcal{O}bj(\mathcal{A}_k)). 
\end{align*}
There is an associative algebra structure on $H(\mathcal{A}_k)$, called 
\textit{the motivic Hall algebra}, defined as follows. 
Let $\mathcal{E}x(\mathcal{A}_k)$ be the 
moduli stack of short exact sequences in $\mathcal{A}_k$
\begin{align}\label{ses:chain}
    0 \to \mathcal{E}_{1\bullet} \to \mathcal{E}_{3\bullet} \to 
    \mathcal{E}_{2\bullet} \to 0. 
\end{align}
There are evaluation morphisms 
\begin{align*}
  p_i \colon  \mathcal{E}x(\mathcal{O}bj(\mathcal{A}_k)) \to \mathcal{O}bj(\mathcal{A}_k)
\end{align*}
sending (\ref{ses:chain}) to 
$\mathcal{E}_{i\bullet}$. 
The $\ast$-product on $H(\mathcal{A}_k)$ is given by 
\begin{align*}
[\mathcal{X}_1 \stackrel{\rho_1}{\to} \mathcal{O}bj(\mathcal{A}_k)] \ast
[\mathcal{X}_2 \stackrel{\rho_2}{\to} \mathcal{O}bj(\mathcal{A}_k)]=[\mathcal{X}_3 
\stackrel{\rho_3}{\to} \mathcal{O}bj(\mathcal{A}_k)]
\end{align*}
where $(\mathcal{X}_3, \rho_3=p_3 \circ \eta)$ is given by the
following diagram 
\begin{align*}
    \xymatrix{
\mathcal{X}_3 \ar[r]^-{\eta} \ar[d] 
& \mathcal{E}x(\mathcal{A}_k) \ar[r]^-{p_3} \ar[d]^-{(p_1, p_2)} &
\mathcal{O}bj(\mathcal{A}_k) \\
\mathcal{X}_1 \times \mathcal{X}_2 \ar[r]^{(\rho_1, \rho_2)} &
\mathcal{O}bj(\mathcal{A}_k)^{\times 2}, &
        }
\end{align*}
where the left square is Cartesian. 

Let $\mathcal{O}bj(\mathcal{A}_k^{\dag})$ be the moduli stack 
of objects in $\mathcal{A}_k^{\dag}$. 
Similarly to above, 
the $\mathbb{Q}$-vector space 
\begin{align*}    H(\mathcal{A}_k^{\dag}):=H(\mathrm{St}/\mathcal{O}bj(\mathcal{A}_k^{\dag}))
\end{align*}
admits a $\ast$-algebra structure 
given by the stack of short exact sequences in $\mathcal{A}_k^{\dag}$. 
Since $\mathcal{A}_k$ is an abelian subcategory 
of $\mathcal{A}_k^{\dag}$
by $\mathcal{E}_{\bullet} \mapsto (0 \to \mathcal{E}_{\bullet})$, there is 
an injective algebra homomorphism $H(\mathcal{A}_k) \to H(\mathcal{A}_k^{\dag})$. 
In particular, there are right and left actions of 
$H(\mathcal{A}_k)$ on $H(\mathcal{A}_k^{\dag})$. 

Recall the moduli stack $\mathcal{C}_{(r_{\bullet}, \chi_{\bullet})}^{\alpha_{\bullet}}$
of semistable chains and the moduli space $\mathcal{C}_{(r_{\bullet}, \chi_{\bullet})}^{\alpha_{\bullet}\dag}$ of JS stable chains. 
We set
\begin{align*}
   &\delta_{(r_{\bullet}, \chi_{\bullet})}^{\alpha_{\bullet}}:=
    [\mathcal{C}_{(r_{\bullet}, \chi_{\bullet})}^{\alpha_{\bullet}} \to \mathcal{O}bj(\mathcal{A}_k)]
    \in H(\mathcal{A}_k), \\
    &\delta_{(r_{\bullet}, \chi_{\bullet})}^{\alpha_{\bullet}\dag}:=
    [\mathcal{C}_{(r_{\bullet}, \chi_{\bullet})}^{\alpha_{\bullet}\dag} \to \mathcal{O}bj(\mathcal{A}_k^{\dag})]
    \in H(\mathcal{A}_k^{\dag}).  
\end{align*}
We denote by $\mathcal{C}_{(r_{\bullet}, \chi_{\bullet})}^{\alpha_{\bullet}\sharp}$
the moduli stack of pairs $(\mathcal{O}_C(-m) \to \mathcal{E}_{\bullet})$ such that 
$\mathcal{E}_{\bullet}$ is 
$\mu_{\alpha_{\bullet}}$-semistable 
with $(\mathrm{rank}(\mathcal{E}_i), \chi(\mathcal{E}_i))=(r_i, \chi_i)$ without imposing the JS stability. 
We also set 
\begin{align*}
\delta_{(r_{\bullet}, \chi_{\bullet})}^{\alpha_{\bullet}\sharp}:=
    [\mathcal{C}_{(r_{\bullet}, \chi_{\bullet})}^{\alpha_{\bullet}\sharp} \to \mathcal{O}bj(\mathcal{A}_k^{\dag})]
    \in H(\mathcal{A}_k^{\dag}).
    \end{align*}

\begin{lemma}\label{lem:id:hall}
The following identity holds in $H(\mathcal{A}_k^{\dag})$:
\begin{align}\label{id:delta}
\delta_{(r_{\bullet}, \chi_{\bullet})}^{\alpha_{\bullet}\sharp}=\sum_{\begin{subarray}{c}(r_{\bullet}^1, \chi_{\bullet}^1)+(r_{\bullet}^2, \chi_{\bullet}^2)=(r_{\bullet}, \chi_{\bullet})\\
\mu_{\alpha_{\bullet}}(r_{\bullet}^i, \chi_{\bullet}^i)=\mu_{\alpha_{\bullet}}(r_{\bullet}, \chi_{\bullet})
\end{subarray}}    
\delta_{(r_{\bullet}^1, \chi_{\bullet}^1)}^{\alpha_{\bullet}\dag}
\ast 
\delta_{(r_{\bullet}^2, \chi_{\bullet}^2)}^{\alpha_{\bullet}}.
\end{align}
\end{lemma}
\begin{proof}
For any pair $(\mathcal{O}_C(-m) \to \mathcal{E}_{\bullet})$, there is 
an exact sequence in $\mathcal{A}_k^{\dag}$, unique up to isomorphism 
\begin{align*}
    0 \to (\mathcal{O}_C(-m) \to \mathcal{E}_{1\bullet}) \to 
    (\mathcal{O}_C(-m) \to \mathcal{E}_{\bullet}) \to 
    (0 \to \mathcal{E}_{2\bullet}) \to 0. 
\end{align*}
Here $(\mathcal{O}_C(-m) \to \mathcal{E}_{1\bullet})$ is a JS 
stable pair and $\mathcal{E}_{2\bullet}$ is $\mu_{\alpha_{\bullet}}$-semistable 
such that $\mu_{\alpha_{\bullet}}(\mathcal{E}_{i\bullet})=\mu_{\alpha_{\bullet}}(\mathcal{E}_{\bullet})$. 
The above exact sequence is nothing but the Harder-Narasimhan filtration 
with respect to the JS stability. 
    Then the lemma follows by describing 
    the above Harder-Narasimhan filtration in terms 
    of motivic Hall algebras, see~\cite[Formula (3.11)]{JS}.
\end{proof}

\subsection{Proof of torsion freeness}
In this subsection, we prove the torsion freeness of 
the topological K-theory of quasi-BPS categories for $L$-twisted Higgs bundles 
using the technique of wall-crossing in Donaldson-Thomas theory~\cite{JS, K-S}. 
We use it to show that the equivalence in Proposition~\ref{prop:induceK} holds 
integrally. 

We denote by $\Lambda \subset K(\mathrm{Var})$
the $\mathbb{Q}$-subspace spanned by 
the classes of the products of $\mathrm{Sym}^i(C)$ for $i\in \mathbb{Z}$
and $\mathbb{L} \in K(\mathrm{Var})$. 
Let $\widehat{\Lambda}$ be the dimensional completion of 
$\Lambda[\mathbb{L}^{-1}]$, and we use the same symbol 
$\widehat{\Lambda}$ to denote its image in $\widehat{K}(\mathrm{Var})$. 

\begin{lemma}\label{lem:torfree}
For a smooth projective variety $Y$, suppose that its class $[Y] \in \widehat{K}(\mathrm{Var})$ lies in $\widehat{\Lambda}$. 
Then $H^{\ast}(Y, \mathbb{Z})$ is torsion free.
\end{lemma}
\begin{proof}
The argument is the same as in~\cite[Theorem~6.1]{GS}. 
    It is enough to show \[P_{\mathbb{Q}}(Y, t)=P_{\mathbb{F}_p}(Y, t)\] 
    for any prime $p$. 
    As $[Y] \in \widehat{\Lambda}$, it is enough to check this 
    for $Y=\mathrm{Sym}^{i}(C)$, where it is known that 
    $H^{\ast}(\mathrm{Sym}^i(C), \mathbb{Z})$ is torsion-free, see~\cite[Equation (12.3)]{Macurve}. 
\end{proof}

\begin{prop}\label{prop:lambda}
Suppose that $\alpha_i-\alpha_{i+1} \geq 2g-2$. 
Then the class of the variety $\mathcal{C}_{(r_{\bullet}, \chi_{\bullet})}^{\alpha_{\bullet}\dag}$
lies in $\widehat{\Lambda}$. 
\end{prop}
\begin{proof}
    We denote by $\Pi$ the composition 
    \begin{align*}
        \Pi \colon H(\mathcal{A}_k^{\dag}) \to K(\mathrm{St}) \to \widehat{K}(\mathrm{Var}),
    \end{align*}
    where the first map forgets the map to $\mathcal{O}bj(\mathcal{A}_k^{\dag})$. 
    The stack $\mathcal{C}_{(r_{\bullet}, \chi_{\bullet})}^{\alpha_{\bullet}\sharp}$
    is a vector bundle over $\mathcal{C}_{(r_{\bullet}, \chi_{\bullet})}^{\alpha_{\bullet}}$
    with fiber $\mathbb{A}^{\Hom(\mathcal{O}_C(-m), \mathcal{E}_{\bullet})}$. 
    For $m\gg 0$, we have 
    \begin{align*}\dim \Hom(\mathcal{O}_C(-m), \mathcal{E}_{\bullet})
    &=\dim \Hom(\mathcal{O}_C(-m), \mathcal{E}_0) \\
    &=mr_0+\chi_0+(1-g)r_0. 
    \end{align*}
    Therefore 
    the image of $\Pi$ of the left hand side of (\ref{id:delta}) is 
    $\mathbb{L}^{mr_0+\chi_0+(1-g)r_0}[\mathcal{C}_{(r_{\bullet}, \chi_{\bullet})}^{\alpha_{\bullet}}]
    $. 

    On the other hand, by Lemma~\ref{lem:ExtJS},
    the stack representing $\delta_{(r_{\bullet}^1, \chi_{\bullet}^1)}^{\alpha_{\bullet}\dag}
\ast 
\delta_{(r_{\bullet}^2, \chi_{\bullet}^2)}^{\alpha_{\bullet}}$
in the right hand side of (\ref{id:delta}) is the vector bundle 
stack over 
$\mathcal{C}_{(r^1_{\bullet}, \chi^1_{\bullet})}^{\alpha_{\bullet}\dag} \times 
\mathcal{C}_{(r^2_{\bullet}, \chi^2_{\bullet})}^{\alpha_{\bullet}}$
with fiber of the form $\mathbb{A}^{a}/\mathbb{A}^b$
such that $a-b=-\chi((r^2_{\bullet}, \chi^2_{\bullet}), (r^1_{\bullet}, \chi^1_{\bullet}))$.
Therefore its class is 
$\mathbb{L}^{-\chi((r^2_{\bullet}, \chi^2_{\bullet}), (r^1_{\bullet}, \chi^1_{\bullet}))}
[\mathcal{C}_{(r^1_{\bullet}, \chi^1_{\bullet})}^{\alpha_{\bullet}\dag} \times 
\mathcal{C}_{(r^2_{\bullet}, \chi^2_{\bullet})}^{\alpha_{\bullet}}]$. 
By applying $\Pi$ to (\ref{id:delta}), we obtain the following identity
in $\widehat{K}(\mathrm{Var})$:
\begin{align}\label{id:delta2}
    &\mathbb{L}^{mr_0+\chi_0+(1-g)r_0}[\mathcal{C}_{(r_{\bullet}, \chi_{\bullet})}^{\alpha_{\bullet}}]
    = \\
    &\notag \sum_{\begin{subarray}{c}(r_{\bullet}^1, \chi_{\bullet}^1)+(r_{\bullet}^2, \chi_{\bullet}^2)=(r_{\bullet}, \chi_{\bullet})\\
\mu_{\alpha_{\bullet}}(r_{\bullet}^i, \chi_{\bullet}^i)=\mu_{\alpha_{\bullet}}(r_{\bullet}, \chi_{\bullet})
\end{subarray}}    
\mathbb{L}^{-\chi((r^2_{\bullet}, \chi^2_{\bullet}), (r^1_{\bullet}, \chi^1_{\bullet}))}
[\mathcal{C}_{(r^1_{\bullet}, \chi^1_{\bullet})}^{\alpha_{\bullet}\dag} \times 
\mathcal{C}_{(r^2_{\bullet}, \chi^2_{\bullet})}^{\alpha_{\bullet}}].
\end{align}
The class $[\mathcal{C}_{(r_{\bullet}, \chi_{\bullet})}^{\alpha_{\bullet}}]$
lies in $\widehat{\Lambda}$, see~\cite[Theorem~B]{GPH}.
Therefore from (\ref{id:delta2}) and using induction on 
$r_0+\cdots+r_k$, we conclude that 
$[\mathcal{C}_{(r_{\bullet}, \chi_{\bullet})}^{\alpha_{\bullet}\dag}]$
also lies in $\widehat{\Lambda}$. 
\end{proof}

\begin{cor}\label{cor:torsionfree}
For $l>2g-2$, the singular cohomology $H^{\ast}((\mathcal{M}^{L\dag})^{\mathbb{C}^{\ast}}, \mathbb{Z})$ is 
torsion free.     
\end{cor}
\begin{proof}
    The corollary follows from the decomposition (\ref{fixed:higgs2}), Lemma~\ref{lem:torfree}, and Proposition~\ref{prop:lambda}. 
\end{proof}

Consider a tuple $(r,\chi,w)$ and let $\mathbb{T}^L=\mathbb{T}^L(r, \chi)_w$ be the corresponding quasi-BPS category. 
We have the following torsion freeness of its topological K-theory:
\begin{prop}\label{prop:Ktop:free}
Suppose that $l>2g-2$. Then 
the topological K-group $K_{\ast}^{\rm{top}}(\mathbb{T}^L)$ is torsion free. 
\end{prop}
\begin{proof}
    By~\cite[Theorem~A.4]{HosLe},
    the variety $\mathcal{M}^{L\dag}$ decomposes into the direct sum 
    of the components of its $\mathbb{C}^{\ast}$-fixed loci up to 
    Tate twist in Voevodsky's triangulated category of 
    mixed motives with integer coefficient. 
    Therefore, from Corollary~\ref{cor:torsionfree}
    and applying the Betti realization~\cite{Florence}, 
    the singular cohomology $H^{\ast}(\mathcal{M}^{L\dag}, \mathbb{Z})$ is 
    also torsion free. 
    Then $K_{\ast}^{\rm{top}}(\mathcal{M}^{L\dag})$ is torsion free, see~\cite[Proposition~6.6]{GS}. 
    Since $\mathbb{T}^L$ is a semiorthogonal 
    summand of $D^b(\mathcal{M}^{L\dag})$, 
    we have that $K_{\ast}^{\rm{top}}(\mathbb{T}^L)$ is 
    a direct summand of $K_{\ast}^{\rm{top}}(\mathcal{M}^{L\dag})$, hence it is torsion-free. 
\end{proof}

\begin{thm}\label{thm:induceK}
Suppose that $l>2g-2$. 
    The functor $\Phi_{\mathcal{P}}$ in (\ref{induce:P})
    induces an equivalence of topological K-theory spectra
    \begin{align}\label{equiv:topK2}
        K^{\rm{top}}(\mathbb{T}^L(r, w+1-g^{\rm{sp}})_{-\chi-g^{\rm{sp}}+1})
        \stackrel{\sim}{\to} 
        K^{\rm{top}}(\mathbb{T}^L(r, \chi)_w). 
    \end{align}
\end{thm}
\begin{proof}
By Proposition~\ref{prop:induceK} and Proposition~\ref{prop:Ktop:free}, 
the result follows from the argument of~\cite[Theorem~3.10]{GS}, which we explain below. 
We use the same notation in the proof of Proposition~\ref{prop:induceK}. 
Let $\mathcal{P}'$ be an object 
\begin{align*}
    \mathcal{P}' \in \mathbb{T}^L(r, \chi)_{-w} \boxtimes_{B^L}
    \mathbb{T}^L(r, w+1-g^{\rm{sp}})_{-\chi-g^{\rm{sp}}+1} 
\end{align*}
which restricts to the kernel object of the inverse of 
the equivalence (\ref{FM:ell}). Note that the above $\mathcal{P}'$ 
exists by the argument of Lemma~\ref{lem:lift}. 
Let 
\begin{align*}
    \Phi_{\mathcal{P}'} \colon \mathbb{T}^L \to \mathbb{T}^{L'}
\end{align*}
be the induced functor. 
We have the following commutative diagram 
\begin{align*}
    \xymatrix{
K_{\ast}^{\rm{top}}(\mathbb{T}^{L'}) \ar[r]^-{\Phi_{\mathcal{P}}^K} \ar[d]
& K_{\ast}^{\rm{top}}(\mathbb{T}^L) \ar[d] \ar[r]^-{\Phi_{\mathcal{P}'}^K} & 
    K_{\ast}^{\rm{top}}(\mathbb{T}^{L'}) \ar[d] \\
  K_{\ast}^{\rm{top}}(\mathbb{T}^{L'})_{\mathbb{Q}} \ar[r]^-{\Phi_{\mathcal{P}\mathbb{Q}}^K} 
& K_{\ast}^{\rm{top}}(\mathbb{T}^L)_{\mathbb{Q}}  \ar[r]^-{\Phi_{\mathcal{P}'\mathbb{Q}}^K} & 
    K_{\ast}^{\rm{top}}(\mathbb{T}^{L'})_{\mathbb{Q}}  
    }
\end{align*}
The vertical arrows are injective by Proposition~\ref{prop:Ktop:free}.
The bottom arrows are given by taking the global sections 
of direct sum of shifts of perverse sheaves on $B$, see the morphism (\ref{induce:KB}). 
Therefore, from~\cite[Lemma~3.3]{GS}, the composition of bottom arrows 
of the above diagram is a unipotent endomorphism. 
Then the composition of top arrows $\Phi_{\mathcal{P}'}^K \circ \Phi_{\mathcal{P}}^K$ 
is a unipotent map of 
free abelian groups, which implies that $\Phi_{\mathcal{P}}^K$
is injective. Applying the same argument to $\Phi_{\mathcal{P}}^K \circ \Phi_{\mathcal{P}'}^K$, 
we obtain the surjectivity of $\Phi_{\mathcal{P}}^K$. 
Therefore $\Phi_{\mathcal{P}}^K$ is an isomorphism. 
\end{proof}

\section{Topological K-theory of BPS categories: the case G=GL and $L=\Omega_C$}
In this section, we prove part (1) of Theorem~\ref{thm:intro2}. 
We consider topological K-theory of (reduced) quasi-BPS 
categories for the usual Hitchin moduli spaces, i.e. for $L=\Omega_C$, 
and prove the expected symmetry for rational topological K-theories. 
\subsection{Quasi-BPS categories in the case of $L=\Omega_C$}\label{subsec:OmegaC}
As before, we 
write $\mathcal{M}^{\rm{red}}=\mathcal{M}(r, \chi)^{\rm{red}}$, $M=M(r, \chi)$, 
$\mathcal{M}^{\Omega_C(p)}=\mathcal{M}^{\Omega_C(p)}(r, \chi)$, 
etc. 
Recall the 
reduced quasi-BPS category 
\begin{align*}
    \mathbb{T}^{\rm{red}} :=\mathbb{T}(r, \chi)_w^{\rm{red}} \subset 
    D^b(\mathcal{M}^{\rm{red}}). 
\end{align*}
Note that, for a fixed $p\in C$,
we have the following diagram, see Subsection~\ref{subsec:qbps}:
\begin{align}\label{dia:crit}
\xymatrix{
\mathcal{M}^{\rm{cl}} \inclusion \ar[d]_-{\pi} & 
\mathcal{M}^{\rm{red}} \inclusion  & \mathcal{M}^{\Omega_C(p)} \ar[d] &
\mathcal{V}_0 \ar[d]_-{\pi_{\mathcal{V}}} \ar[r]^-{f} \ar[l] & \mathbb{C} \\
M \iinclusion & & M^{\Omega_C(p)} & N_0. \ar[l] \ar[ur]_-{f_{N_0}} &
}    
\end{align}
Here, each vertical arrow is a good moduli space morphism, 
the function $f$ is given by 
\begin{align*}
    f(x, v)=\mathrm{tr}(s_0(x) \circ v), \ 
    x \in \mathcal{M}^{\Omega_C(p)}, v \in \mathcal{V}_0|_{x},
\end{align*}
where $s_0$ is the section of $\mathcal{V}_0$ as in 
(\ref{sec:s0}). 
The critical locus $\mathrm{Crit}(f)$ is isomorphic to the classical 
truncation of the $(-1)$-shifted cotangent of $\mathcal{M}^{\rm{red}}$, see~\cite[Chapter~2]{T}.  
Recall that $\mathcal{V}$ consists of 
\begin{align}\label{V0:const}
    (F, \theta, u), \ \theta \colon F \to F\otimes \Omega_C(p), \ 
    u \in \mathrm{End}(F|_{p}),
\end{align}
where $(F, \theta)$ is a semistable $\Omega_C(p)$-Higgs bundle, 
and the subbundle $\mathcal{V}_0 \subset \mathcal{V}$ corresponds to (\ref{V0:const}) such that
$\mathrm{tr}(u)=0$. 
Let $\mathbb{C}^{\ast}$ acts on fibers of 
$\mathcal{V}_0 \to \mathcal{M}^{\Omega_C(p)}$ and 
\begin{align*}
    \mathbb{T}_{\mathbb{C}^{\ast}}' \subset D^b_{\mathbb{C}^{\ast}}(\mathcal{V}_0), \ 
    \mathbb{T}' \subset D^b(\mathcal{V}_0)
\end{align*}
be the subcategory consisting of objects $\mathcal{P}$ such that, 
for all $\nu \colon B \mathbb{C}^{\ast} \to \mathcal{V}_0$,
the set of weights $\nu^{\ast}\mathcal{P}$ satisfies 
the weight condition as in (\ref{cond:nu}), i.e. 
\begin{align}\label{cond:nu2}
\mathrm{wt}(\nu^{\ast}\mathcal{P}) \subset 
\left[-\frac{1}{2}\mathrm{wt} \det (\nu^{\ast}\mathbb{L}_{\mathcal{V}})^{>0}, 
\frac{1}{2}\mathrm{wt} \det (\nu^{\ast}\mathbb{L}_{\mathcal{V}})^{>0}
\right]+\frac{w}{d}\mathrm{wt}(\nu^{\ast}\delta). 
\end{align}
By~\cite[Lemma~2.6, Corollary~3.15]{PTquiver}, there is a Koszul equivalence 
\begin{align}\label{equiv:K:T}
    \mathbb{T}^{\rm{red}}\stackrel{\sim}{\to} \mathrm{MF}^{\rm{gr}}(\mathbb{T}'_{\mathbb{C}^{\ast}}, f). 
\end{align}

We now define the JS stable pair version of 
the vector bundle $\mathcal{V}$. 
For $m\gg 0$, define $\mathcal{V}^{\dag}$ to be consisting of tuples
\begin{align}\label{JS:V2}
(F, \theta, u, s), \ 
\theta \colon F \to F \otimes\Omega_C(p), \ u \in \mathrm{End}(F|_p), \ 
s \colon \mathcal{O}_C(-m) \to F,
\end{align}
where $(F, \theta, u)$ is as in (\ref{V0:const}), and 
the tuple (\ref{JS:V2}) satisfies the JS stability: 
for any $(F', \theta', u')$ as in (\ref{V0:const}) with 
$\mu(F)=\mu(F')$ and a surjection $\eta \colon (F, \theta) \twoheadrightarrow (F', \theta')$
of $\Omega_C(p)$-Higgs bundles which fits into a commutative diagram
\begin{align*}
    \xymatrix{
F|_{p} \ar[r]^-{\eta} \ar[d]_-{u} & F|_{p}' \ar[d]_-{u'} \\
F|_{p} \ar[r]^-{\eta} & F|_{p}'
        }
\end{align*}
we have $\eta \circ s \neq 0$. 
We also define $\mathcal{V}_0^{\dag} \subset \mathcal{V}^{\dag}$
by the condition $\mathrm{tr}(u)=0$. 
\begin{lemma}\label{JS:V}
The stacks $\mathcal{V}^{\dag}$, $\mathcal{V}_0^{\dag}$ are smooth algebraic 
spaces such that the compositions
\begin{align}\label{mor:alpha}
\alpha \colon \mathcal{V}^{\dag}\stackrel{\pi_{\rm{JS}}}{\to}
 \mathcal{V} \stackrel{\pi_{\mathcal{V}}}{\to} N, \
\alpha \colon 
    \mathcal{V}_0^{\dag} \stackrel{\pi_{\rm{JS}}}{\to} \mathcal{V}_0 \stackrel{\pi_{\mathcal{V}}}{\to} N_0
\end{align}
are proper morphisms.  
Here, the morphism $\pi_{\mathrm{JS}}$ 
sends $(F, \theta, u, s)$
to $(F, \theta, u)$. 
\end{lemma}
\begin{proof}
We prove the lemma only for $\mathcal{V}^{\dag}$. 
It is enough to prove the claim \'{e}tale locally at 
any point in $M^{\Omega_C(p)}$. 
Let $y \in M^{\Omega_C(p)}$ be a closed point. 
By Lemma~\ref{lem:equiver}, we 
may assume that $y$ lies in the deepest stratum, 
corresponding to $V \otimes E_0$ where $\dim V=d$ and $E_0$ 
is a stable $\Omega_C(p)$-Higgs bundle with 
$(\mathrm{rank}(E_0), \chi(E_0))=(r_0, \chi_0)$. 
Using the \'{e}tale local description of 
$\mathcal{M}^{\Omega_C(p)} \to M^{\Omega_C(p)}$
as in Subsection~\ref{subsec:loc},
(also see 
the proof of~\cite[Proposition~3.23]{PThiggs}),  
the composition 
\begin{align*}
    \mathcal{V}^{\dag} \to \mathcal{V} \to N
\end{align*}
is \'{e}tale locally on $M^{\Omega_C(p)}$ at $y$
isomorphic to 
\begin{align}\label{Vdag:N}
    (\mathfrak{gl}(V)^{\oplus (1+2r_0^2 g)}\oplus V^{\otimes \chi(E_0(m))})^{\rm{ss}}/GL(V)
    \to \mathfrak{gl}(V)^{\oplus (1+2r_0^2 g)}\ssslash GL(V), 
\end{align}
where the semistable locus is with respect to the determinant character of $GL(V)$. 
The left hand side is the moduli space of stable representations 
of a quiver with vertices $\{0, 1\}$,  $(1+2r_0^2 g)$-loops at $1$, $\chi(E_0(m))$-arrows from $0$ to $1$, and 
with dimension vector $(1, d)$. 
The stable representations correspond to 
those 
generated by the images from the maps from $0$
to $1$, see~\cite[Lemma~6.1.9]{T}. 
The source of the map (\ref{Vdag:N}) is smooth because it consists of stable representations.
By~\cite[Theorem~4.1]{Hille}, the map (\ref{Vdag:N}) is also projective
since the map 
\begin{align*}
     (\mathfrak{gl}(V)^{\oplus (1+2r_0^2 g)}\oplus V^{\otimes \chi(E_0(m))})/GL(V)
    \to \mathfrak{gl}(V)^{\oplus (1+2r_0^2 g)}\ssslash GL(V)
\end{align*}
is the good moduli space morphism. We therefore obtain 
the desired conclusion. 
\end{proof}
\begin{lemma}\label{lem:TJS}
The compositions
\begin{align*}
    \mathbb{T}'_{\mathbb{C}^{\ast}} \subset
    D_{\mathbb{C}^{\ast}}^b(\mathcal{V}_0) \stackrel{\pi_{\rm{JS}}^{\ast}}{\to}     
    D_{\mathbb{C}^{\ast}}^b(\mathcal{V}_0^{\dag}), \ 
    \mathbb{T}' \subset D^b(\mathcal{V}_0) \stackrel{\pi_{\rm{JS}}^{\ast}}{\to}  D^b(\mathcal{V}_0^{\dag})
\end{align*}
are fully-faithful and admit right adjoints. 
    \end{lemma}
    \begin{proof}
    Using the \'{e}tale local description (\ref{Vdag:N}) in terms of the Ext-quiver, 
    an argument similar to Theorem~\ref{thm:sod:JS} applies. See also the proof 
    of~\cite[Proposition~3.7]{PTtop}.
    \end{proof}

    \subsection{Topological K-theory of $\mathbb{Z}/2$-graded BPS categories}
    Recall the equivalence (\ref{equiv:K:T}). 
    We introduce the 
    $\mathbb{Z}/2$-graded version 
    of quasi-BPS category by replacing the 
    right hand side in (\ref{equiv:K:T}) with the $\mathbb{Z}/2$-graded dg-category of matrix factorizations
\begin{align*}
    \mathbb{T}^{\rm{red}, \mathbb{Z}/2}:=\mathrm{MF}(\mathbb{T}', f)
    \subset \mathrm{MF}(\mathcal{V}_0, f). 
\end{align*}
Let $g$ be the function on $\mathcal{V}_0^{\dag}$ 
defined by the following commutative diagram 
\begin{align*}
    \xymatrix{
\mathcal{V}_0^{\dag} \ar[r]^-{\pi_{\rm{JS}}} \ar[rrd]_-{g} & \mathcal{V}_0 \ar[rd]^-{f} \ar[r] & N_0 \ar[d]^-{f_{N_0}} \\
& & \mathbb{C}.     
    }
\end{align*}
By Lemma~\ref{lem:TJS} and 
using~\cite[Proposition~2.5]{PT0}, 
the following composition functor is fully-faithful 
with right adjoint 
\begin{align}\label{ff:adj}
    \mathbb{T}^{\rm{red}, \mathbb{Z}/2}
    \subset \mathrm{MF}(\mathcal{V}_0, f) 
    \stackrel{\pi_{\rm{JS}}^{\ast}}{\to} 
    \mathrm{MF}(\mathcal{V}_0^{\dag}, g). 
\end{align}
\begin{lemma}\label{lem:Ktop:phi}
There is an equivalence 
\begin{align}\label{equiv:Kphi}
    \mathcal{K}_{\mathcal{V}_0^{\dag}}^{\rm{top}}(\mathrm{MF}(\mathcal{V}_0^{\dag}, g))_{\mathbb{Q}}
    \simeq \phi_{g}^{\rm{inv}}[\beta^{\pm 1}]. 
\end{align}
In the above, $\phi_g^{\rm{inv}}$ is the monodromy invariant vanishing cycle 
of $\phi_g:=\phi_g(\mathbb{Q}_{\mathcal{V}_0^{\dag}})$:
\begin{align*}
\phi_g^{\rm{inv}}:=\mathrm{Cone}(\phi_g \stackrel{1-T}{\to} \phi_g)=\phi_g \otimes \mathbb{Q}[\gamma],
\end{align*}
where $\mathbb{Q}[\gamma]=\mathbb{Q} \oplus \mathbb{Q}\gamma$ with $\deg \gamma=1$ 
and $T$ is the monodromy operator, which in this case vanishes $T=0$. 
\end{lemma}
\begin{proof}
The equivalence (\ref{equiv:Kphi}) is proved in~\cite[Lemma~6.5]{PTtop}. 
    The monodromy operator $T$ 
    vanishes because 
    the $\mathbb{C}^{\ast}$-action on the fibers of 
    $\mathcal{V}_0 \to \mathcal{M}^{\Omega_C(p)}$ lifts to 
    an action on $\mathcal{V}_0^{\dag}$, and $g$ is of weight one 
    with respect to the above $\mathbb{C}^{\ast}$-action. 
\end{proof}
\begin{remark}
    A reason of considering $\mathbb{Z}/2$-periodic version is that the graded 
    category
    $\mathrm{MF}^{\rm{gr}}(\mathcal{V}_0^{\dag}, g)$ is not 
    linear over $\mathrm{Perf}(\mathcal{V}_0^{\dag})$, 
    rather it is linear over $\mathrm{Perf}^{\rm{gr}}(\mathcal{V}_0^{\dag})$. 
    Because of a lack of reference of relative topological K-theory linear 
    over categories of graded perfect complexes, we use the $\mathbb{Z}/2$-version which 
    is linear over $\mathrm{Perf}(\mathcal{V}_0^{\dag})$. 
\end{remark}

For a variety $M$, we say an object $P \in D(\mathrm{Sh}_{\mathbb{Q}}(M))$ 
is \textit{perverse-split} if 
$P$ is isomorphic to a direct sum 
$\oplus_{i \in \mathbb{Z}}A_i[-i]$
with $A_i$ is a perverse sheaf on $M$. 
\begin{lemma}\label{lem:split}
The object
\begin{align}\label{KNtop:alpha}
    \mathcal{K}_{N_0}^{\rm{top}}(\alpha_{\ast}\mathbb{T}^{\rm{red}, \mathbb{Z}/2})_{\mathbb{Q}}
    \in D(\mathrm{Sh}_{\mathbb{Q}}(N_0))
\end{align}
is perverse-split. In the above, $\alpha$ is the morphism (\ref{mor:alpha}). 
Moreover, the object 
(\ref{KNtop:alpha}) contains $\phi_{f_{N_0}}(\mathrm{IC}_{N_0})[\beta^{\pm 1}, \gamma]$ as
a direct summand. 
\end{lemma}
\begin{proof}
    Since $\alpha$ is projective, 
    by Theorem~\ref{thm:phiproper} and Lemma~\ref{lem:Ktop:phi}
    we have 
    \begin{align}\label{isom:KNtop}
    \mathcal{K}_{N_0}^{\rm{top}}(\alpha_{\ast}\mathrm{MF}(\mathcal{V}_0^{\dag}, g))_{\mathbb{Q}}
    \cong 
        \alpha_{\ast}\mathcal{K}_{\mathcal{V}_0^{\dag}}^{\rm{top}}(\mathrm{MF}(\mathcal{V}_0^{\dag}, g))_{\mathbb{Q}}
        \cong \alpha_{\ast}\phi_g^{\rm{inv}}[\beta^{\pm 1}]. 
    \end{align}
    We have 
    \begin{align*}
        \alpha_{\ast}\phi_g^{\rm{inv}} \cong \phi_{f_{N_0}}(\alpha_{\ast}\mathbb{Q}_{\mathcal{V}_0^{\dag}})\otimes_{\mathbb{Q}}\mathbb{Q}[\gamma]
    \end{align*}
    which is perverse-split by the BBDG decomposition theorem and 
    the fact that the vanishing cycle functor $\phi_{f_{N_0}}$ preserves the perverse t-structure. 
    Therefore (\ref{isom:KNtop}) is perverse-split. 
    Since (\ref{ff:adj}) is a part of a semiorthogonal decomposition, 
    the object (\ref{KNtop:alpha}) is a direct summand of (\ref{isom:KNtop}). 
    Therefore (\ref{KNtop:alpha}) is perverse-split. 
    The second statement follows as in the proof of Proposition~\ref{prop:topK}, 
    noting that $\mathcal{V}_0^{\dag}\to N_0$ is a projective bundle over a 
    dense open smooth subset $N_0^{\rm{st}} \subset N_0$
    (e.g. we can take $N_0^{\rm{st}}$ to be the preimage 
    of the stable part in $M^{\Omega_C(p)}$ under the 
    morphism $N_0 \to M^{\Omega_C(p)}$ in (\ref{dia:crit})),
    and the category $\mathbb{T}^{\rm{red}, \mathbb{Z}/2}$ restricted to $N_0^{\rm{st}}$ is 
    equivalent to the category of matrix factorization of 
    $f_{N_0} \colon N^{\rm{st}}_0 \to \mathbb{C}$, 
    possibly twisted by some Brauer class. 
\end{proof}

Let $0 \colon \mathcal{M}^{\Omega_C(p)} \to \mathcal{V}_0$ be the zero section. 
It induces the morphism on good moduli spaces
\begin{align*}
    \overline{0} \colon M^{\Omega_C(p)} \to N_0
\end{align*}
which is a section of the projection 
$N_0 \to M^{\Omega_C(p)}$. In particular, 
$\overline{0}$ is a closed immersion. 

\begin{lemma}\label{lem:support}
Assume that the tuple $(r,\chi,w)$ is primitive.
Then there is a closed subscheme $Z \subset N_0$ whose 
support is the image of $\overline{0}$ such that 
the $\mathrm{Perf}(N_0)$-linear structure on $\mathbb{T}^{\rm{red}, \mathbb{Z}/2}$
descends to the $\mathrm{Perf}(Z)$-linear structure 
via the restriction functor $\mathrm{Perf}(N_0) \to \mathrm{Perf}(Z)$. 
\end{lemma}
\begin{proof}
As in~\cite[Theorem~6.4]{PTK3}
(the result in loc.cit. is stated for moduli of sheaves on K3 surfaces, but it applies ad litteram to the local Calabi-Yau surface $\mathrm{Tot}_C(\Omega_C)$), 
the graded version $\mathbb{T}^{\rm{red}}$ 
    admits a strong generator $\mathcal{E}$. 
    Indeed, the dg-category $\mathbb{T}^{\rm{red}}$ is smooth, 
    so $\mathcal{E}$ is constructed by taking 
    the direct sum of second factors of objects in $(\mathbb{T}^{\rm{red}})^{\rm{op}} \boxtimes 
    \mathbb{T}^{\rm{red}}$ which generates the diagonal
    $(\mathbb{T}^{\rm{red}})^{\rm{op}} \boxtimes \mathbb{T}^{\rm{red}}$-module. 
    There is a natural morphism in $D^b(N_0)$:
    \begin{align}\label{morphism:EN}
        \mathcal{O}_{N_0} \to \pi_{\mathcal{V}\ast}R\mathcal{H}om(\mathcal{E}, \mathcal{E}). 
    \end{align}    
    By categorical support lemma~\cite[Theorem~6.6]{PTK3} (stated for K3 surfaces, but the argument applies ad litteram to Higgs bundles as well), 
    \cite[Lemma~5.4]{PTquiver}, any object in $\mathbb{T}^{\rm{red}}$ is supported over 
    $\mathcal{V}_0\times_{N_0}\mathrm{Im}(\overline{0})$,
    where $\mathcal{V}_0 \to N_0$ is the good moduli space 
    morphism. Therefore the right hand side in (\ref{morphism:EN})
    is supported on the image of $\overline{0}$.     
    It follows that there is a closed subscheme $i \colon Z \hookrightarrow N_0$
    with support the image of $\overline{0}$ such that 
    the right hand side in (\ref{morphism:EN}) lies in the image 
    of $i_{\ast} \colon D^b(Z) \to D^b(N_0)$. 
    Then the morphism (\ref{morphism:EN}) 
    factors through $\mathcal{O}_{N_0} \twoheadrightarrow \mathcal{O}_Z$. 

    Let $\mathcal{E}^{\mathbb{Z}/2} \in \mathbb{T}^{\rm{red}, \mathbb{Z}/2}$
    be the object by forgetting the grading of $\mathcal{E}$. 
   By the above argument, there is a morphism: 
    \begin{align}\label{mor:BN}
        \mathcal{O}_Z \to \pi_{\mathcal{V}\ast}R\mathcal{H}om(\mathcal{E}^{\mathbb{Z}/2}, 
        \mathcal{E}^{\mathbb{Z}/2}) =:\mathscr{B}_{N_0},
    \end{align}
which is a morphism of sheaves of $\mathbb{Z}/2$-graded dg-algebras over $N_0$.  
     By the above construction of $\mathcal{E}$, the object 
    $\mathcal{E}^{\mathbb{Z}/2}$ is also 
    a strong generator of $\mathbb{T}^{\rm{red}, \mathbb{Z}/2}$. 
    Therefore, we have the fully-faithful functor 
    \begin{align}\label{ff:Z2}
        \mathbb{T}^{\rm{red}, \mathbb{Z}/2} \hookrightarrow 
        D(\mathscr{B}_{N_0}), (-) \mapsto \pi_{\mathcal{V}\ast}R\mathcal{H}om(\mathcal{E}^{\mathbb{Z}/2}, -),
    \end{align}
    where the right hand side is the derived category of $\mathbb{Z}/2$-graded dg-modules 
    over $\mathscr{B}_{N_0}$. 
    The above functor is $\mathrm{Perf}(N_0)$-linear, and 
    the $\mathrm{Perf}(N_0)$-linear structure on $D(\mathscr{B}_{N_0})$ descends 
    to the $\mathrm{Perf}(Z)$-linear structure on it 
    where the action of $\mathrm{Perf}(Z)$ is given by 
    the tensor product over $\mathcal{O}_Z$ through the morphism (\ref{mor:BN}). 
    Since the image of the pull-back functor $\mathrm{Perf}(N_0) \to \mathrm{Perf}(Z)$ generates 
    $\mathrm{Perf}(Z)$, the above $\mathrm{Perf}(Z)$-linear 
    structure on $D(\mathscr{B}_{N_0})$ restricts to 
    the $\mathrm{Perf}(Z)$-linear structure on $\mathbb{T}^{\rm{red}, \mathbb{Z}/2}$ under 
    the embedding (\ref{ff:Z2}), i.e. 
    $\mathbb{T}^{\rm{red}, \mathbb{Z}/2} \otimes \mathrm{Perf}(Z) \subset \mathbb{T}^{\rm{red}, \mathbb{Z}/2}$. 
    \end{proof}

 \subsection{BPS sheaves}
Let $X$ be the non-compact Calabi-Yau 3-fold:
\begin{align*}
X=\mathrm{Tot}_C(\mathcal{O}_C \oplus \Omega_C)
\stackrel{p}{\to} C,
\end{align*}
where $p$ is the natural projection. 
Let $\mathcal{M}_X$ be the moduli stack of 
compactly supported coherent sheaves $E$ on 
$X$ with $\mathrm{rank}(p_{\ast}E)=r$ and 
$\chi(p_{\ast}E)=\chi$. 
Note that $\mathcal{M}_X$ consists of tuples
\begin{align}\label{F:theta}
    (F, \theta, \iota), 
\end{align}
where $(F, \theta)$ is a Higgs bundle on $C$
and $\iota$
is an endomorphism of $(F, \theta)$. 
Consider the closed subscheme 
\begin{align*}
\mathcal{M}_X^{\rm{red}} \subset \mathcal{M}_X    
\end{align*}
consisting of tuples (\ref{F:theta}) such that the trace of $\iota$ 
is zero. Then there are isomorphisms
\begin{align*}
    \Omega_{\mathcal{M}}[-1]^{\rm{cl}} \cong 
    \mathcal{M}_X, \ 
    \Omega_{\mathcal{M}^{\rm{red}}}[-1]^{\rm{cl}} \cong 
    \mathcal{M}_X^{\rm{red}}, 
\end{align*}
where $\Omega_{\mathcal{M}}[-1]$ is the $(-1)$-shifted 
cotangent of $\mathcal{M}$, see~\cite{T} for the above 
isomorphisms. 
Consider the following diagram 
\begin{align*}
    \xymatrix{
\mathcal{M}^{\rm{cl}} \ar[d]_-{\pi} & \mathcal{M}_X^{\rm{red}} \ar[d]_-{\pi_X} \ar[l]^-{p} \inclusion & \mathcal{V}_0 \ar[d] \ar[r]^-{f} & \mathbb{C} \\
M & M_X^{\rm{red}} \inclusion \ar[l]^-{\overline{p}} & N_0 
\ar[ru]_-{f_{N_0}}&
    }
\end{align*}
such that $\mathcal{M}_X^{\rm{red}}= \mathrm{Crit}(f)$, where recall that $\mathcal{M}=\mathcal{M}(r, \chi)$
and $M=M(r, \chi)$. 
Let
\begin{align*}
    \phi_{\mathcal{M}} \in \mathrm{Perv}(\mathcal{M}_X^{\rm{red}})
\end{align*}
be the DT perverse sheaf~\cite{BBBJ} on $\mathcal{M}_X^{\rm{red}}$
with respect to the orientation data
canonically defined as a $(-1)$-shifted cotangent, 
see~\cite[Section~3.3.3]{T}. 
The orientation data determined by the embedding $\mathcal{M}_X^{\rm{red}} \hookrightarrow 
\mathcal{V}_0$ matches 
with the above one, so $\phi_{\mathcal{M}}$ is isomorphic 
to $\phi_f(\mathrm{IC}_{\mathcal{V}_0})$, 
see~\cite[Proposition~2.4]{KinjoKoseki}. 

The \textit{BPS sheaf} on $M_X^{\rm{red}}$ is defined by 
\begin{align*}    \mathcal{BPS}_{M_X^{\rm{red}}}:={}^p\mathcal{H}^1(\pi_{X\ast}
    \phi_{\mathcal{M}}) \in \mathrm{Perv}(M_X^{\rm{red}}),    
\end{align*}
where ${}^p\mathcal{H}^1(-)$ is the first 
cohomology with respect to the perverse t-structure. 
By the support lemma~\cite[Proposition~5.1]{KinjoKoseki}, 
the sheaf $\mathcal{BPS}_{M_X^{\rm{red}}}$ is supported on 
the image of the map 
$\overline{0}_M \colon M \to M_X^{\rm{red}}$ induced 
by the zero section $0 \colon \mathcal{M}^{\rm{red}} \to \mathcal{M}_X^{\rm{red}}$. 
The BPS sheaf on $M$ is given by, see~\cite[Proposition~5.12]{KinjoKoseki}:
\begin{align*}
    \mathcal{BPS}_M :=\overline{p}_{\ast}
    \mathcal{BPS}_{M_X^{\rm{red}}} \in \mathrm{Perv}(M). 
\end{align*}
\begin{remark}
In~\cite[Proposition~5.11]{KinjoKoseki}, the support 
lemma is given for the unreduced moduli space $M_X$. 
The reduced version removes the extra $\mathbb{A}^1$-factor 
in loc. cit. 
    \end{remark}
The following lemma is proved similarly to~\cite[Proposition~3.10]{KinjoKoseki}. 
    \begin{lemma}\label{lem:BPS:IC}
    There is an isomorphism 
    \begin{align*}
    \mathcal{BPS}_{M_X^{\rm{red}}} \cong \phi_{f_{N_0}}(\mathrm{IC}_{N_0}). 
        \end{align*}
    \end{lemma}

Let $(r, \chi, w)$ satisfy the BPS condition and $\mathbb{T}^{\rm{red}}=\mathbb{T}(r, \chi)_w^{\rm{red}}$ be the 
reduced BPS category. 
We have the following relation 
of its topological K-theory with BPS sheaves: 

    \begin{thm}\label{prop:BPS}
There is an isomorphism 
\begin{align}\label{KMtop}
    \mathcal{K}_M^{\rm{top}}(\mathbb{T}^{\rm{red}})_{\mathbb{Q}} \cong 
    \mathcal{BPS}_M[\beta^{\pm 1}]. 
\end{align}
    \end{thm}
    \begin{proof}
        Let $i \colon Z\hookrightarrow N_0$ be a closed subscheme 
        as in Lemma~\ref{lem:support}. 
        By Theorem~\ref{thm:phiproper}, we have 
\begin{align}\label{KN0T}
    \mathcal{K}_{N_0}^{\rm{top}}(i_{\ast}\mathbb{T}^{\rm{red}, \mathbb{Z}/2})_{\mathbb{Q}}
    \cong i_{\ast}\mathcal{K}_Z^{\rm{top}}(\mathbb{T}^{\rm{red}, \mathbb{Z}/2})_{\mathbb{Q}}
\end{align}
which is perverse-split by Lemma~\ref{lem:split}. 
Since the composition $Z \hookrightarrow N_0 \to M$ is proper (indeed, a homeomorphism), we conclude 
that $\mathcal{K}_M^{\rm{top}}(\mathbb{T}^{\rm{red}, \mathbb{Z}/2})_{\mathbb{Q}}$ 
is also perverse-split. 
We have the following relation 
\begin{align}\label{KMtop:Z2}
    \mathcal{K}_M^{\rm{top}}(\mathbb{T}^{\rm{red}, \mathbb{Z}/2})_{\mathbb{Q}}
    \cong \mathcal{K}_M^{\rm{top}}(\mathbb{T}^{\rm{red}})_{\mathbb{Q}} \otimes \mathbb{Q}[\gamma]
\end{align}
see~\cite[Proposition~7.5]{PTtop} (the statement in loc. cit. is for the absolute version, but the 
argument for the relative version is the same). 
It follows that the left hand side of (\ref{KMtop}) is perverse-split. 
Moreover by the second statement of Lemma~\ref{lem:split}
and Lemma~\ref{lem:BPS:IC}, 
it contains $\mathcal{BPS}_M[\beta^{\pm 1}]$ 
as a direct summand. 

It follows that we can write 
\begin{align*}
     \mathcal{K}_M^{\rm{top}}(\mathbb{T}^{\rm{red}})_{\mathbb{Q}} =\mathcal{BPS}_M[\beta^{\pm 1}]
     \oplus P_0[\beta^{\pm 1}] \oplus P_1[1][\beta^{\pm 1}]
\end{align*}
for some $P_i \in \mathrm{Perv}(M)$. It is enough to show that $P_0=P_1=0$. 

The argument of the vanishing $P_i=0$ is the same as in the last part of 
the proof of Proposition~\ref{prop:topK}. 
Namely for a closed point $y\in M$, let $\mathscr{P}_{Q_y}(\bm{d})$ and $P_{Q_y}(\bm{d})$ be as in (\ref{loc:P(d)}). 
Then from~\cite[Theorem~8.26, Theorem~10.6]{PTtop},    
we see that 
\begin{align*}
    \mathcal{K}_{P_{Q_y}(\bm{d})}(\mathbb{T}_{Q_y}(\bm{d})_{w}^{\rm{red}}) \cong \mathcal{BPS}_{P_{Q_y}(\bm{d})}[\beta^{\pm 1}]
\end{align*}
where $\mathbb{T}_{Q_y}(\bm{d})_w^{\rm{red}} \subset 
D^b(\mathscr{P}(\bm{d})^{\rm{red}})_w$
is the preprojective reduced quasi-BPS category, see~\cite[Definition~2.14]{PTquiver}. Therefore $P_i=0$ \'{e}tale locally 
at each $y\in M$, hence $P_i=0$. 
    \end{proof}

\begin{lemma}\label{lem:Ktop:global}
The natural map 
\begin{align}\label{mor:eta}
\eta \colon 
    K^{\rm{top}}(\mathbb{T}^{\rm{red}})_{\mathbb{Q}} 
    \to    \Gamma(\mathcal{K}^{\rm{top}}_B(\mathbb{T}^{\rm{red}})_{\mathbb{Q}})
\end{align}
given in Lemma~\ref{lem:gsection} is an isomorphism. 
\end{lemma}
\begin{proof}
By Lemma~\ref{lem:support} and
by Theorem~\ref{thm:phiproper}
for proper morphisms 
\begin{align*}
    N_0 \hookleftarrow Z \to M \to B,
\end{align*}
we have that 
\begin{align*}
    \Gamma(\mathcal{K}_{N_0}^{\rm{top}}(\mathbb{T}^{\rm{red}, \mathbb{Z}/2})_{\mathbb{Q}})=\Gamma(\mathcal{K}_B^{\rm{top}}(\mathbb{T}^{\rm{red}, \mathbb{Z}/2})_{\mathbb{Q}}). 
\end{align*}
By applying Proposition~\ref{prop:gsection}
for the proper 
morphism $\mathcal{V}_0^{\dag} \to N_0$
and using the fact that $\mathbb{T}^{\rm{red}, \mathbb{Z}/2}$
is a semiorthogonal summand of $\mathrm{MF}(\mathcal{V}_0^{\dag}, g)$, see (\ref{ff:adj}), 
the natural map 
\begin{align*}
K^{\rm{top}}(\mathbb{T}^{\rm{red}, \mathbb{Z}/2})_{\mathbb{Q}}
\to \Gamma(\mathcal{K}^{\rm{top}}_{N_0}(\mathbb{T}^{\rm{red}, \mathbb{Z}/2})_{\mathbb{Q}})
\end{align*}
is an isomorphism. 
Noting (\ref{KMtop:Z2})
and also its absolute version in~\cite[Proposition~7.5]{PTtop}, 
it follows that 
\begin{align*}
    \eta \otimes 1_{\mathbb{Q}[\gamma]} \colon 
    K^{\rm{top}}(\mathbb{T}^{\rm{red}})_{\mathbb{Q}}
    \otimes \mathbb{Q}[\gamma] \to 
    \Gamma(K_B^{\rm{top}}(\mathbb{T}^{\rm{red}})_{\mathbb{Q}})
    \otimes \mathbb{Q}[\gamma]
\end{align*}
is an isomorphism, where $\eta$ is 
the morphism (\ref{mor:eta}). 
Therefore $\eta$ is an isomorphism. 
\end{proof}

\subsection{Duality of rational topological K-theories}
In this subsection, we prove part (1) of Theorem~\ref{thm:intro2}. 
\begin{thm}\label{thm:induceK2}
There is an equivalence 
\begin{align}\label{equiv:topK3}
    K^{\rm{top}}(\mathbb{T}(r, w)^{\rm{red}}_{-\chi})_{\mathbb{Q}} \stackrel{\sim}{\to}
    K^{\rm{top}}(\mathbb{T}(r, \chi)_w^{\rm{red}})_{\mathbb{Q}}. 
\end{align}
    \end{thm}
\begin{proof}
Let $L$ be a line bundle on $C$ such that $\deg L \gg 0$ is even
and admits a surjection 
\begin{align}\label{surj:OL}
\mathcal{O}_C \oplus \Omega_C \twoheadrightarrow L.
\end{align}
We write $\mathcal{M}^L=\mathcal{M}^L(r, \chi)$, 
$\mathcal{M}'^{L}=\mathcal{M}^L(r, w+1-g^{\rm{sp}})$ 
etc. 
By~\cite[Theorem~5.6, Proposition~5.7]{KinjoMasuda}, 
\cite[Section~3.3]{KinjoKoseki}, there is a 
commutative diagram 
\begin{align}\label{diagram:MMB}
    \xymatrix{
\mathcal{M} \ar[d]_-{\pi} & \mathcal{M}_X^{\rm{red}} 
\ar[l]_-{p} \inclusion \ar[d]_-{\pi_X} & \mathcal{M}^L
\ar[d]_-{\pi^L} \ar[rdd]^-{w} \\
M \ar[d]_-{h} & M_X^{\rm{red}} \ar[l]_-{\overline{p}} 
\ar[d]_-{h_X} \inclusion & M^L \ar[d]_-{h^L}
\ar[rd]_-{\overline{w}} \\
B & B_X \ar[l]^-{p_B} \ar[r] & B^L \ar[r]_-{w_B} & \mathbb{C}
    }
\end{align}
    such that $\mathcal{M}_X^{\rm{red}}=\mathrm{Crit}(\overline{w})$, 
    and the embedding $\mathcal{M}_X^{\rm{red}} \hookrightarrow \mathcal{M}^L$ is induced by (\ref{surj:OL}). 
 Here $h_X \colon \mathcal{M}_X^{\rm{red}} \to B_X$ is a 
 Hitchin-type map, where $B_X$ is given by
 \begin{align*}
     B_X=\bigoplus_{i=1}^r H^0(C, \mathrm{Sym}^i(\mathcal{O}_C \oplus \Omega_C)).
 \end{align*}
 The map $h_X$ sends a compactly supported coherent sheaf on 
 $X$ to its support, see~\cite[Section~2.4]{KinjoKoseki}. 
  By~\cite[Proposition~3.10]{KinjoKoseki}, there is 
  an isomorphism
\begin{align*}
    \mathcal{BPS}_M \cong \overline{p}_{\ast}
    \phi_{\overline{w}}(\mathrm{IC}_{M^L}). 
\end{align*}
Therefore using Proposition~\ref{prop:topK}, Theorem~\ref{prop:BPS}, and Theorem~\ref{thm:phiproper}, 
there are isomorphisms (cf.~Remark~\ref{rmk:phipB}):
\begin{align}\label{compute:pB}
    p_{B\ast}\phi_{w_B}(\mathcal{K}_{B^L}^{\rm{top}}(h^L_{\ast}\mathbb{T}^L))_{\mathbb{Q}} 
    &\cong p_{B\ast}\phi_{w_B}(h^L_{\ast}\mathrm{IC}_{M^L}[-\dim M^L][\beta^{\pm 1}]) \\
   \notag &\cong p_{B\ast}h_{X\ast}\phi_{\overline{w}}(\mathrm{IC}_{M^L})[\beta^{\pm 1}][-\dim M^L] \\
  \notag  & \cong h_{\ast}\overline{p}_{\ast}\phi_{\overline{w}}(\mathrm{IC}_{M^L})[\beta^{\pm 1}][-\dim M^L] \\
   \notag &\cong h_{\ast}\mathcal{BPS}_M[\beta^{\pm 1}][-\dim M^L] \\
  \notag  & \cong h_{\ast}\mathcal{K}_M^{\rm{top}}(\mathbb{T}^{\rm{red}})_{\mathbb{Q}}[-\dim M^L] \\
  \notag  &\cong \mathcal{K}_B^{\rm{top}}(\mathbb{T}^{\rm{red}})_{\mathbb{Q}}[-\dim M^L]. 
\end{align}
Applying $p_{B\ast}\phi_{w_B}$ to the 
isomorphism (\ref{induce:KB}), we obtain the isomorphism
\begin{align*}
    \mathcal{K}_B^{\rm{top}}(\mathbb{T}(r, w+1-g^{\rm{sp}})^{\rm{red}}_{-\chi-g^{\rm{sp}}+1})_{\mathbb{Q}} \stackrel{\sim}{\to}
    \mathcal{K}_B^{\rm{top}}(\mathbb{T}(r, \chi)_w^{\rm{red}})_{\mathbb{Q}}.
\end{align*}
The isomorphism 
(\ref{equiv:topK3}) follows by taking the global section of the 
above isomorphism (see Lemma~\ref{lem:Ktop:global}) and the following equivalence 
from~\cite[Lemma~3.8]{PThiggs}
\begin{align*}
    \mathbb{T}(r, w+1-g^{\rm{sp}})_{-\chi-g^{\rm{sp}}+1}^{\rm{red}}\simeq \mathbb{T}(r, w)_{-\chi}^{\rm{red}}. 
\end{align*}
\end{proof}  
\begin{remark}\label{rmk:phipB}
By~\cite[Remark~2.5]{KinjoKoseki},
the map $B_X \to B^L$ restricted to $\mathrm{Im}(h_X)$ is injective. 
The notation $p_{B\ast}\phi_{w_B}(-)$ means that 
$\phi_{w_B}(-)$ lies in the image from $\mathrm{Im}(h_X)$, 
then apply $p_{B\ast}$ by regarding $\phi_{w_B}(-)$ as a sheaf on 
$B_X$. The same notation will also appear in the later sections. 
\end{remark}

\section{Review of quasi-BPS categories for SL/PGL Higgs bundles}
In this section, we recall the SL/PGL Higgs moduli 
spaces and the definition of their quasi-BPS categories from \cite{PThiggs}. 
   \subsection{SL-Higgs moduli spaces}
    Let $C$ be a smooth projective curve of genus $g$, and let $L$ be a line bundle on $C$ 
      such that $l=\deg L>2g-2$ or $L=\Omega_C$. 
For each decomposition $r=r_1+\cdots+r_k$, we 
set 
\begin{align*}
    \mathrm{SL}(r_{\bullet}):=\Ker\left(\prod_{i=1}^k \GL(r_i) \stackrel{\det}{\to} \mathbb{C}^{\ast}\right), \ 
    (g_i)_{1\leq i\leq k} \stackrel{\det}{\mapsto} \prod_{i=1}^k \det g_i. 
\end{align*}
For a tuple of integers $\chi_{\bullet}=(\chi_1, \ldots, \chi_k)$, the moduli stack of 
$(L, \chi_{\bullet})$-twisted $\mathrm{SL}(r_{\bullet})$-Higgs bundle is given as follows. 
Consider the closed substack 
\begin{align*}
\mathcal{M}^L(r_{\bullet}, \chi_{\bullet})^{\rm{tr}=0}
    \subset \prod_{i=1}^k \mathcal{M}^L(r_i, \chi_i)
\end{align*}
given by the derived zero locus of 
\begin{align*}
  \mathrm{tr} \colon   \prod_{i=1}^k \mathcal{M}^L(r_i, \chi_i) \to H^0(L)^{\rm{der}}, \ 
  \{(F_i, \theta_i)\}_{1\leq i\leq k}
\to \sum_{i=1}^k \mathrm{tr}(\theta_i).
\end{align*}
Here $H^0(L)^{\rm{der}}$ is the derived space of global 
sections: 
\begin{align*}
    H^0(L)^{\rm{der}}=\Spec \mathrm{Sym}(R\Gamma(L)^{\vee})
    =\begin{cases} H^0(L), & l>2g-2 \\
    H^0(\omega_C) \times \Spec \mathbb{C}[\epsilon], & L=\Omega_C, \end{cases}
\end{align*}
where $\deg \epsilon=-1$.
Consider the map 
\begin{align*}
    \det \colon \mathcal{M}^L(r_{\bullet}, \chi_{\bullet})^{\rm{tr}=0}
    \to \mathcal{P}ic(C), \ \{(F_i, \theta_i)\}_{1\leq i\leq k}
    \mapsto \otimes_{i=1}^k \det F_i. 
\end{align*}
Here $\mathcal{P}ic(C)$ is the Picard stack of 
line bundles on $C$, which is a trivial $\mathbb{C}^{\ast}$-gerbe
\begin{align*}
    \mathcal{P}ic(C)=    
    \mathrm{Pic}(C)/\mathbb{C}^{\ast}.
\end{align*}
We fix a line bundle $A$ on $C$ of degree $\chi+r(g-1)$. 
The $(L, \chi_{\bullet})$-twisted $\mathrm{SL}(r_{\bullet})$-Higgs moduli stack is 
given by the Cartesian square 
\begin{align*}
    \xymatrix{
\mathcal{M}_{\mathrm{SL}(r_{\bullet})}^L(\chi_{\bullet})  \ar[r] \ar[d] & \mathcal{M}^L(r_{\bullet}, 
\chi_{\bullet})^{\rm{tr}=0} 
\ar[d]^-{\rm{det}} \\
\Spec \mathbb{C} \ar[r] & \mathcal{P}ic(C),
     }
\end{align*}
where the bottom arrow corresponds to $A$. 
The stack
$\mathcal{M}_{\mathrm{SL}(r_{\bullet})}^L(\chi_{\bullet})$ is smooth for $l>2g-2$ and it is quasi-smooth for $L=\Omega_C$. 
In particular, for $k=1$, we obtain 
the $(L, \chi)$-twisted $\mathrm{SL}(r)$-Higgs moduli stack $\mathcal{M}^L_{\mathrm{SL}(r)}(\chi)$.

Note that the center of $\mathrm{SL}(r_{\bullet})$ is 
\begin{align*}
    Z(\mathrm{SL}(r_{\bullet}))=\Ker\left( (\mathbb{C}^{\ast})^k \to \mathbb{C}^{\ast} \right), 
    (t_i)_{1\leq i\leq k} \mapsto \prod_{i=1}^k t_i^{r_i}, 
\end{align*}
and we have 
\begin{align*}
    \Hom(Z(\mathrm{SL}(r_{\bullet})), \mathbb{C}^{\ast})
    =\mathbb{Z}^{\oplus k}/(r_1, \ldots, r_k)\mathbb{Z}. 
\end{align*}
There is a corresponding orthogonal decomposition
\begin{align*}
    D^b(\mathcal{M}_{\mathrm{SL}(r_{\bullet})}^L(\chi_{\bullet}))=\bigoplus_{w_{\bullet} \in \mathbb{Z}^{\oplus k}/(r_1, \ldots, r_k)\mathbb{Z}}
    D^b(\mathcal{M}_{\mathrm{SL}(r_{\bullet})}^L(\chi_{\bullet}))_{w_{\bullet}}
\end{align*}
where for $w=0$ each summand corresponds to the weight $(w_1, \ldots, w_k)$-component with respect to the
action of $Z(\mathrm{SL}(r_{\bullet}))$.

\subsection{PGL-Higgs moduli spaces}
For each decomposition $r=r_1+\cdots+r_k$, we set
\begin{align*}
    \mathrm{PGL}(r_{\bullet}):=\left(\prod_{i=1}^k \GL(r_i)\right)/\mathbb{C}^{\ast}. 
\end{align*}
There is a natural action 
of $\mathcal{P}ic(C)$ 
on the disjoint union of 
$\mathcal{M}^L(r_{\bullet}, \chi_{\bullet})^{\rm{tr}=0}$ for $\chi_{\bullet} \in \mathbb{Z}^k$
via the tensor product. 
The moduli stack of $L$-twisted $\mathrm{PGL}(r_{\bullet})$-Higgs bundles is given by
\begin{align*}
    \mathcal{M}^L_{\mathrm{PGL}(r_{\bullet})} &:=
    \left( \coprod_{\chi_{\bullet} \in \mathbb{Z}^k}\mathcal{M}^L(r_{\bullet}, \chi_{\bullet})^{\rm{tr}=0} \right)/\mathcal{P}ic(C) \\
    &=\coprod_{\chi_{\bullet} \in \mathbb{Z}^{\oplus k}/(r_1, \ldots, r_k)\mathbb{Z}}
    \mathcal{M}^L_{\mathrm{PGL}(r_{\bullet})}(\chi_{\bullet}),
\end{align*}
where each component is given by 
\begin{align}\label{quot:Pic}
   \mathcal{M}^L_{\mathrm{PGL}(r_{\bullet})}(\chi_{\bullet})
   =\mathcal{M}^L(r_{\bullet}, \chi_{\bullet})^{\rm{tr}=0}/\mathcal{P}ic^0(C). 
   \end{align}
In particular, for $k=1$, we obtain the $L$-twisted $\mathrm{PGL}(r)$-Higgs moduli 
stack $\mathcal{M}_{\mathrm{PGL}(r)}^L(\chi)$ for $\chi \in \mathbb{Z}/r\mathbb{Z}$. 

By taking the quotient by $\mathrm{Pic}^0(C)$ in (\ref{quot:Pic}) instead 
of the quotient by $\mathcal{P}ic^0(C)$, we 
obtain the $\mathbb{C}^{\ast}$-gerbe
\begin{align*}
     \widetilde{\mathcal{M}}^L_{\mathrm{PGL}(r_{\bullet})}(\chi_{\bullet})
     \to \mathcal{M}^L_{\mathrm{PGL}(r_{\bullet})}(\chi_{\bullet}). 
\end{align*}
There is an orthogonal decomposition of the derived category 
into subcategories of fixed $\mathbb{C}^{\ast}$-weight:
\begin{align}\label{decom:pgl}
  D^b(\widetilde{\mathcal{M}}^L_{\mathrm{PGL}(r_{\bullet})}(\chi_{\bullet}))=
  \bigoplus_{w\in \mathbb{Z}} D^b(\mathcal{M}^L_{\mathrm{PGL}(r_{\bullet})}(\chi_{\bullet}))_w
\end{align}
Note that the $w=0$ component is equivalent to $D^b(\mathcal{M}^L_{\mathrm{PGL}(r_{\bullet})}(\chi_{\bullet}))$.

The center of $\mathrm{PGL}(r_{\bullet})$ is 
\begin{align*}
    Z(\mathrm{PGL}(r_{\bullet}))=(\mathbb{C}^{\ast})^k/\mathbb{C}^{\ast}
\end{align*}
and we have 
\begin{align*}
    \Hom(Z(\mathrm{PGL}(r_{\bullet})), \mathbb{C}^{\ast})=
    \Ker\left(\mathbb{Z}^{\oplus k} \to \mathbb{Z}   \right), \ 
    (w_1, \ldots, w_k) \mapsto \sum_{i=1}^k w_i. 
\end{align*}
There is a corresponding orthogonal decomposition: 
\begin{align*}
    D^b(\mathcal{M}^L_{\mathrm{PGL}(r_{\bullet})}(\chi_{\bullet}))_w
    =\bigoplus_{w_1+\cdots+w_k=w} 
    D^b(\mathcal{M}^L_{\mathrm{PGL}(r_{\bullet})}(\chi_{\bullet}))_{w_{\bullet}} 
\end{align*}
where each summand corresponds to the weight $(w_1, \ldots, w_k)$-component with 
respect to the
action of $Z(\mathrm{PGL}(r_{\bullet}))$. 

\subsection{Semiorthogonal decompositions of SL/PGL-Higgs moduli stacks}
We recall the SL/PGL versions of the semiorthogonal decomposition from Theorem~\ref{thm:sod}: 
\begin{thm}\emph{(\cite[Theorem~7.2]{PThiggs})}\label{thm:SL}
For each $r_{\bullet} \in \mathbb{Z}^k$, $\chi_{\bullet} \in \mathbb{Z}^k$ and 
$w_{\bullet} \in \mathbb{Z}^k/(r_1, \ldots, r_k)\mathbb{Z}$, 
there is a subcategory
\begin{align*}
    \mathbb{T}^L_{\mathrm{SL}(r_{\bullet})}(\chi_{\bullet})_{w_{\bullet}} \subset D^b(\mathcal{M}^L_{\mathrm{SL}(r_{\bullet})}(\chi_{\bullet}))_{w_{\bullet}}
\end{align*}
such that there is a semiorthogonal decomposition 
\begin{align}\label{sod:SL}
D^b(\mathcal{M}^L_{\mathrm{SL}(r)}(\chi))_w=
\left\langle \mathbb{T}^L_{\mathrm{SL}(r_{\bullet})}(\chi_{\bullet})_{w_{\bullet}} \,\Big|\,
\frac{v_1}{r_1}<\cdots<\frac{v_k}{r_k} \right\rangle. 
\end{align}
The right hand side is after all partitions $r=r_1+\cdots+r_k$, 
$\chi=\chi_1+\cdots+\chi_k$
such that $\chi_i/r_i=\chi/r$, 
and 
$w_{\bullet} \in \mathbb{Z}^k/(r_1, \ldots, r_k)\mathbb{Z}$ with $w_1+\cdots+w_r=w$ in 
$\mathbb{Z}/r\mathbb{Z}$, and 
where $v_{\bullet} \in \mathbb{Q}^{\oplus k}/(r_1, \ldots, r_k)\mathbb{Q}$ 
is determined by
\begin{align*}
    w_i=v_i+\frac{l}{2}r_i \left(\sum_{i>j}r_j-\sum_{i<j} r_j  \right). 
\end{align*}
The fully-faithful functor 
\begin{align*}
       \mathbb{T}^L_{\mathrm{SL}(r_{\bullet})}(\chi_{\bullet})_{w_{\bullet}} 
       \to D^b(\mathcal{M}^L_{\mathrm{SL}(r)}(\chi))_w
\end{align*}
is induced by the categorical Hall product. 
\end{thm}

\begin{thm}\emph{(\cite[Theorem~7.3]{PThiggs})}\label{thm:PGL}
For each $r_{\bullet} \in \mathbb{Z}^k$, 
$\chi_{\bullet} \in \mathbb{Z}^k/(r_1, \ldots, r_k)\mathbb{Z}$ 
and $w_{\bullet} \in \mathbb{Z}^k$, 
there is a subcategory 
\begin{align*}
      \mathbb{T}^L_{\mathrm{PGL}(r_{\bullet})}(\chi_{\bullet})_{w_{\bullet}} \subset D^b(\mathcal{M}^L_{\mathrm{PGL}(r_{\bullet})}(\chi_{\bullet}))_{w_{\bullet}}
\end{align*}
such that there is a semiorthogonal decomposition 
\begin{align}\label{sod:PGL}
    D^b(\mathcal{M}^L_{\mathrm{PGL}(r)}(\chi))_w=
\left\langle \mathbb{T}^L_{\mathrm{PGL}(r_{\bullet})}(\chi_{\bullet})_{w_{\bullet}} \,\Big|\,
\frac{v_1}{r_1}<\cdots<\frac{v_k}{r_k} \right\rangle. 
\end{align}
The right hand side is after all partitions $r=r_1+\cdots+r_k$, 
$\chi=\chi_1+\cdots+\chi_k$ 
such that 
$\chi_1/r_1=\cdots=\chi_k/r_k$
and $w_{\bullet} \in \mathbb{Z}^{\oplus k}$ with $w_1+\cdots+w_k=w$, 
and $v_i\in \frac{1}{2}\mathbb{Z}$
is determined by
\begin{align*}
    w_i=v_i+\frac{l}{2}r_i \left(\sum_{i>j}r_j-\sum_{i<j} r_j  \right). 
\end{align*}
The fully-faithful functor 
\begin{align*}
       \mathbb{T}^L_{\mathrm{PGL}(r_{\bullet})}(\chi_{\bullet})_{w_{\bullet}} 
       \to D^b(\mathcal{M}^L_{\mathrm{PGL}(r)}(\chi))_w
\end{align*}
is induced by the categorical Hall product. 
\end{thm}

In particular by setting $k=1$ in Theorem~\ref{thm:SL} and 
Theorem~\ref{thm:PGL}, we obtain 
the subcategories 
\begin{align*}
    \mathbb{T}_{\mathrm{SL}(r)}^L(\chi)_w \subset 
    D^b(\mathcal{M}_{\mathrm{SL}(r)}^L(\chi))_w, \ 
    \mathbb{T}_{\mathrm{PGL}(r)}^L(\chi)_w \subset 
    D^b(\mathcal{M}_{\mathrm{PGL}(r)}^L(\chi))_w
\end{align*}
which are the quasi-BPS categories for SL/PGL 
Higgs bundles. Again we omit $L$ from the notation 
if $L=\Omega_C$. 

\subsection{The SL/PGL symmetry conjecture}
The Hitchin maps give the diagram 
\begin{align}\label{dia:SLPGL}
    \xymatrix{
 \mathcal{M}^L_{\mathrm{SL}(r)}(\chi) \ar[rd] & & \ar[ld] \mathcal{M}^L_{\mathrm{PGL}(r)}(\chi) \\
    & B_{\geq 2}^L, &
    }
\end{align}
where $B_{\geq 2}^L$ is defined by 
\begin{align*}
    B_{\geq 2}^L :=\bigoplus_{i=2}^r H^0(C, L^{\otimes i}). 
\end{align*}
The maps in (\ref{dia:SLPGL}) have relative dimension 
$g^{\rm{sp}}-g$. 
The following is the SL/PGL version of 
Conjecture~\ref{conj:T0}:
\begin{conj}\emph{(\cite[Conjecture~7.5]{PThiggs})}\label{conj:PS}
Suppose that the tuple $(r, \chi, w)$ satisfies the BPS condition.
Then there is an equivalence 
\begin{align*}
    \mathbb{T}^L_{\mathrm{PGL}(r)}(w+1-g^{\rm{sp}})_{-\chi+1-g^{\rm{sp}}} \stackrel{\sim}{\to} 
     \mathbb{T}^L_{\mathrm{SL}(r)}(\chi)_w.
\end{align*}
    \end{conj}

\section{The SL/PGL-duality of topological K-theories}
In this section, we provide evidence towards Conjecture \ref{conj:PS}, namely we prove part (2) of Theorem~\ref{thm:intro2} and part (2) of Theorem~\ref{thm:intro}. 
\subsection{Topological K-theory of SL-moduli spaces}
In this subsection, we compute the topological K-theory of the quasi-BPS categories for SL-moduli spaces using Joyce--Song pairs. The argument is similar to the one used to prove Proposition \ref{prop:topK} and Theorem \ref{prop:BPS}. 

Recall the moduli stack of SL-Higgs bundles together with the 
good moduli space 
\begin{align*}
\mathcal{M}^L_{\mathrm{SL}(r)}(\chi) \to M^L_{\mathrm{SL}(r)}(\chi)
\end{align*}
and the Hitchin map 
\begin{align}\label{hit:h0}
    h_0 \colon M^L_{\mathrm{SL}(r)}(\chi) \to B^L_{\geq 2}. 
\end{align}
Note that $M^L_{\mathrm{SL}(r)}(\chi)$
is the fiber of the smooth morphism 
\begin{align}\label{map:smooth}
    M^L(r, \chi) \to H^0(L) \times \mathrm{Pic}(C), \ 
    (F, \theta) \mapsto (\mathrm{tr}(\theta), \det F)
\end{align}
at $(0, A)$, where $A$ is a fixed line bundle of degree $\chi+r(g-1)$. 
We define the space of Joyce--Song pairs $\mathcal{M}^L_{\mathrm{SL}(r)}(\chi)^{\rm{JS}}$ 
to be the fiber of the composition 
\begin{align*}
    \mathcal{M}^L(r, \chi)^{\rm{JS}} \to M^L(r, \chi) \to H^0(L) \times \mathrm{Pic}(C) 
\end{align*}
at $(0, A)$. For simplicity, we write $M^L=M^L(r, \chi)$, $\mathcal{M}^{L\dag}=\mathcal{M}^L(r, \chi)^{\rm{JS}}$, 
$M_\mathrm{SL}^L=M^L_{\mathrm{SL}(r)}(\chi)$ and $\mathcal{M}_\mathrm{SL}^{L\dag}=\mathcal{M}^L_{\mathrm{SL}(r)}(\chi)^{\rm{JS}}$. We first discuss the case of $l>2g-2$.

\begin{prop}\label{prop:topK:SL}
Assume that $l>2g-2$ and $(r, \chi, w)$ satisfies the BPS condition. 
Then the complex of sheaves 
\begin{align*}
    \mathcal{K}_{B^L_{\geq 2}}^{\rm{top}}(\mathbb{T}^L_{\mathrm{SL}(r)}(\chi)_w)_{\mathbb{Q}} 
    \in D(\mathrm{Sh}_{\mathbb{Q}}(B^L_{\geq 2}))
\end{align*}
is of the form $(P\oplus Q[1])[\beta^{\pm 1}]$ 
for semisimple perverse sheaves $P, Q$ 
whose generic supports are contained in $B^L_{\geq 2} \cap (B^L)^{\rm{ell}}$. 
\end{prop}
\begin{proof}
We use the notation from Lemma~\ref{lem:Ktop:pullback}.
By setting $\mathcal{C}=\mathbb{T}^L(r, \chi)_w \subset D^b(\mathcal{M}^{L\dag})$, 
its base change $\mathcal{C}_\mathrm{SL} \subset D^b(\mathcal{M}_\mathrm{SL}^{L\dag})$ as in Lemma~\ref{lem:Ktop:pullback}
is equivalent to $\mathbb{T}_{\mathrm{SL}(r)}^L(\chi)_w$, which 
follows by taking the base-change of the semiorthogonal decomposition in Theorem~\ref{thm:sod:JS}. 
By Lemma~\ref{lem:Ktop:pullback}, 
we have the equivalence
\begin{align*}
    \mathcal{K}_{M_\mathrm{SL}^L}^{\rm{top}}(\mathbb{T}_{\mathrm{SL}(r)}^L(\chi)_w)_{\mathbb{Q}}
    \simeq i^{-1}\mathcal{K}_{M^L}^{\rm{top}}(\mathbb{T}^L(r, \chi)_w)_{\mathbb{Q}}. 
\end{align*}
Note that we have 
\begin{align*}
    \mathcal{K}_{M^L}^{\rm{top}}(\mathbb{T}^L(r, \chi)_w)_{\mathbb{Q}}\cong \mathrm{IC}_{M^L}[-\dim M^L][\beta^{\pm 1}]
\end{align*}
by Proposition~\ref{prop:topK}. Since the map (\ref{map:smooth}) is smooth, 
we have 
\begin{align*}
    i^{-1}\mathrm{IC}_{M^L}[-\dim M^L] \cong \mathrm{IC}_{M_\mathrm{SL}^L}[-\dim M_\mathrm{SL}^L]. 
\end{align*}
Therefore we have 
\begin{align*}
     \mathcal{K}_{M_\mathrm{SL}^L}^{\rm{top}}(\mathbb{T}_{\mathrm{SL}(r)}^L(\chi)_w)_{\mathbb{Q}}\cong 
     \mathrm{IC}_{M_\mathrm{SL}^L}[-\dim M_\mathrm{SL}^L][\beta^{\pm 1}]. 
\end{align*}
Using Theorem~\ref{thm:phiproper}, we conclude that there is an isomorphism 
\begin{align*}
   \mathcal{K}_{B^L_{\geq 2}}^{\rm{top}}(\mathbb{T}^L_{\mathrm{SL}(r)}(\chi)_w)_{\mathbb{Q}}
   \cong h_{0\ast} \mathrm{IC}_{M_\mathrm{SL}^L}[-\dim M_\mathrm{SL}^L][\beta^{\pm 1}]
\end{align*}
where $h_0$ is the Hitchin map (\ref{hit:h0}). 
It is proved in~\cite[Theorem~0.1]{MSint}
that $h_{0\ast}\mathrm{IC}_{M_\mathrm{SL}^L}$ is of the form 
$\oplus_i A_i[-i]$ for semisimple perverse sheaves $A_i$ 
with generic supports contained in $B^L_{\geq 2} \cap (B^L)^{\rm{ell}}$. 
Therefore, we obtain the desired conclusion. 
\end{proof}

\begin{lemma}\label{lem:sl:tfree}
Suppose that $l>2g-2$. 
Then the topological K-groups $K_{\ast}^{\rm{top}}(\mathbb{T}^L_{\mathrm{SL}(r)}(\chi)_w)$
are torsion free. 
\end{lemma}
\begin{proof}
      We use the notation in the proof of Proposition~\ref{prop:topK:SL}.
    Since $\mathbb{T}^L_{\mathrm{SL}(r)}(\chi)_w$ is a semiorthogonal 
    summand of $D^b(\mathcal{M}_\mathrm{SL}^{L\dag})$, 
    it is enough to prove that $K_{\ast}^{\rm{top}}(\mathcal{M}_\mathrm{SL}^{L\dag})$ is torsion 
    free. Note that the $\mathbb{C}^{\ast}$-action on $\mathcal{M}^{L\dag}$ 
    scaling the Higgs field restricts to the action of $\mathcal{M}_\mathrm{SL}^{L\dag}$. 
    As in the proof of Proposition~\ref{prop:Ktop:free}, 
    it is enough to prove that $H^{\ast}\big((\mathcal{M}_\mathrm{SL}^{L\dag})^{\mathbb{C}^{\ast}}, \mathbb{Z}\big)$
    is torsion free. Since $\mathcal{M}_\mathrm{SL}^{L\dag}$ is independent of a choice of $A \in \mathrm{Pic}(C)$ 
    of degree $\chi+r(g-1)$, we may assume that $A=\mathcal{O}_C((\chi+r(g-1))c_0)$ 
    for a fixed point $c_0 \in C$. 

    Recall that, by Proposition~\ref{prop:lambda}, the class of $(\mathcal{M}^{L\dag})^{\mathbb{C}^{\ast}}$
    in $\widehat{K}(\mathrm{Var})$ is a linear combination of classes of 
    products of $\mathbb{L}$ and $\mathrm{Sym}^i(C)$ for $i\in \mathbb{Z}$. 
    We have the following pull-back square: 
    \begin{align*}
        \xymatrix{
(\mathcal{M}_\mathrm{SL}^{L\dag})^{\mathbb{C}^{\ast}} \times \mathrm{Pic}^0(C) \ar[r] \ar[d] &
\mathrm{Pic}^0(C) \ar[d]^{[r]} \\
(\mathcal{M}^{L\dag})^{\mathbb{C}^{\ast}} \ar[r]_-{\det(-)\otimes A^{-1}} & \mathrm{Pic}^0(C).         
        }
    \end{align*}
    We then apply the argument of~\cite[Proposition~6.4]{GS}
    by replacing $\check{\mathcal{M}}$ in loc. cit. with $(\mathcal{M}_\mathrm{SL}^{L\dag})^{\mathbb{C}^{\ast}}$
    to conclude that 
    the class $(\mathcal{M}_\mathrm{SL}^{L\dag})^{\mathbb{C}^{\ast}} \times \mathrm{Pic}^0(C)$
    in $\widehat{K}(\mathrm{Var})$ is a linear combination of the classes 
    of the form $\widetilde{\prod_{i}\mathrm{Sym}^{l_i}(C)} \times \mathbb{L}^m$, where 
    $\widetilde{\prod_{i}\mathrm{Sym}^{l_i}(C)}$ is defined to be the pull-back diagram: 
    \begin{align*}
       \xymatrix{\widetilde{\prod_{i}\mathrm{Sym}^{l_i}(C)} \ar[r] \ar[d] &
\mathrm{Pic}^0(C) \ar[d]^{[r]} \\
\prod_{i}\mathrm{Sym}^{l_i}(C)\ar[r]_-{\otimes_i \mu(k_i)} & \mathrm{Pic}^0(C).         
        }  
    \end{align*}
    Here $\mu(k)$ is the map 
    \begin{align*}
    \mu(k) \colon \mathrm{Sym}^l(C) \to \mathrm{Pic}^0(C), \ 
    (c_1, \ldots, c_l) \mapsto \mathcal{O}_C(c_1+\cdots+c_l-l c_0)^{\otimes k}. 
    \end{align*}
    It is proved in~\cite[Lemma~6.8]{GS} that 
    $H^{\ast}\big(\widetilde{\prod_{i}\mathrm{Sym}^{l_i}(C)}, \mathbb{Z}\big)$ is torsion free. 
    Therefore as in the proof of Lemma~\ref{lem:torfree},
    we conclude that $H^{\ast}\big((\mathcal{M}_\mathrm{SL}^{L\dag})^{\mathbb{C}^{\ast}} \times \mathrm{Pic}^0(C), \mathbb{Z}\big)$ is torsion free, 
    hence $H^{\ast}\big((\mathcal{M}_\mathrm{SL}^{L\dag})^{\mathbb{C}^{\ast}}, \mathbb{Z}\big)$
    is also torsion free.  
\end{proof}

We next consider the case of $L=\Omega_C$. The computation is analogous to the one from Theorem \ref{prop:BPS}.
For simplicity, we write $M=M^{\Omega_C}(r, \chi)$, 
$M_{\mathrm{SL}}=M_{\mathrm{SL}}^{\Omega_C}(r, \chi)$, etc. 
Let $i \colon M_{\mathrm{SL}} \hookrightarrow M$ be the 
closed immersion, and set 
\begin{align*}
    \mathcal{BPS}_{M_{\mathrm{SL}}}:=i^{-1}\mathcal{BPS}_M[-2g]. 
\end{align*}

\begin{lemma}\label{lem:KtopSL:Omega}
There is an isomorphism 
\begin{align*}
    \mathcal{K}_{M_{\mathrm{SL}}}^{\rm{top}}(\mathbb{T}_{\mathrm{SL}(r)}(\chi)_w)_{\mathbb{Q}}
    \cong \mathcal{BPS}_{M_{\mathrm{SL}}}[\beta^{\pm 1}]. 
\end{align*}
\end{lemma}
\begin{proof}
    The lemma follows from the definition of the sheaf
    $\mathcal{BPS}_{M_{\mathrm{SL}}}$, and from Theorem~\ref{prop:BPS} and 
    Lemma~\ref{lem:Ktop:pullback}. 
\end{proof}

Let $\mathcal{O}_C \oplus \Omega_C \twoheadrightarrow L$ 
be a surjection as in the proof of Theorem~\ref{thm:induceK2}. 
In the notation of the diagram (\ref{diagram:MMB}), we have the commutative diagram 
\begin{align*}
    \xymatrix{
M \ar[d]_-{(\mathrm{tr}, \det)} & M_X^{\rm{red}} \inclusion \ar[l] & M^L \ar[d]_-{(\mathrm{tr}, \det)} \\
H^0(\Omega_C) \times \mathrm{Pic}(C) \iinclusion & & H^0(L) \times \mathrm{Pic}(C).     
    }
\end{align*}
Here, the bottom arrow is the inclusion 
induced by $\Omega_C \hookrightarrow \mathcal{O}_C \oplus \Omega_C \twoheadrightarrow L$. 
By taking the (classical) fiber at $(0, A) \in H^0(L) \times \mathrm{Pic}(C)$, 
we obtain the commutative diagram 
\begin{align}\label{comm:MXSL}
    \xymatrix{
M_{\mathrm{SL}} \ar@<-0.3ex>@{^{(}->}[d]_-{i} & \ar[l]_-{p_{\mathrm{SL}}} M_{X, \mathrm{SL}} \inclusion \ar@<-0.3ex>@{^{(}->}[d]_-{i} & M_{\mathrm{SL}}^L \ar@<-0.3ex>@{^{(}->}[d]_-{i} \ar[rd]^-{w_{\mathrm{SL}}} & \\
M & M_{X}^{\rm{red}} \ar[l]_-{\overline{p}} \inclusion & M^L \ar[r]^-{\overline{w}} 
& \mathbb{C},     
    }
\end{align}
where $M_{X, \mathrm{SL}}$ is the fiber of 
the morphism $M_X^{\rm{red}} \to H^0(L) \times \mathrm{Pic}(C)$ at $(0, A)$. 
The function $\overline{w}$ is given in the diagram (\ref{diagram:MMB}), which 
by~\cite[Proposition~5.7]{KinjoMasuda} can be described 
as, for $(F, \theta)$ a $L$-twisted Higgs bundle:  
    \begin{align}\label{funct:ow}
        \overline{w}(F, \theta)=\langle \alpha, \mathrm{tr}(\theta^2) \rangle
    \end{align}
    for some $\alpha \in H^0(L^{\otimes 2})^{\vee}$. The element $\alpha$ corresponds to the extension class
    \begin{align*}
        0 \to L' \to \mathcal{O}_C \oplus \Omega_C \to L \to 0
    \end{align*}
    under the isomorphism 
    \begin{align*}
        \Ext_C^1(L, L') =\Ext_C^1(L, L^{-1}\otimes \Omega_C) \cong H^0(L^{\otimes 2})^{\vee}. 
    \end{align*}

\begin{lemma}\label{lem:BPS:SLisom}
There is an isomorphism 
\begin{align*}
    \mathcal{BPS}_{M_{\mathrm{SL}}} \cong p_{\mathrm{SL}\ast}\phi_{\omega_{\mathrm{SL}}}(\mathrm{IC}_{M^L_{\mathrm{SL}}}). 
\end{align*}
\end{lemma}
\begin{proof}
    Recall that there is an isomorphism, see~\cite[Proposition~3.10]{KinjoKoseki}:
    \begin{align*}
        \mathcal{BPS}_{M} \cong \overline{p}_{\ast}\phi_{\overline{w}}(\mathrm{IC}_{M^L}). 
    \end{align*}
    It follows that 
    \begin{align*}
        \mathcal{BPS}_{M_{\mathrm{SL}}} \cong 
        i^{-1}\overline{p}_{\ast}\phi_{\overline{w}}(\mathrm{IC}_{M^L})[-2g] 
         \cong p_{\mathrm{SL}\ast}i^{-1}\phi_{\overline{w}}(\mathrm{IC}_{M^L})[-2g]. 
    \end{align*}
    It is enough to show that 
    \begin{align}\label{isom:iIC}
i^{-1}\phi_{\overline{w}}(\mathrm{IC}_{M^L}) \cong \phi_{w_{\mathrm{SL}}}(\mathrm{IC}_{M_{\mathrm{SL}}})[2g]. 
    \end{align}
    We have the following commutative diagram 
    \begin{align*}
        \xymatrix{
M_{\mathrm{SL}}^L \inclusion^-{j} \ar[rd]_-{i} \ar[rdd]_-{w_{\mathrm{SL}}} & 
M_{\mathrm{SL}}^L \times H^0(L) \times \mathrm{Pic}^0(C) \ar[r] \ar[d]_-{\gamma} &
H^0(L) \times \mathrm{Pic}^0(C) \ar[d]_{(\id, [r])} \\
& M^L \ar[d]_-{\overline{w}} \ar[r]_-{(\mathrm{tr}, \det \otimes A^{-1})} & H^0(L) \times \mathrm{Pic}^0(C) \\
& \mathbb{C}. &                }
    \end{align*}
    Here $j$ is given by $j(x)=(x, 0, A)$, the 
    right square is Cartesian, and the map $\gamma$ is given by 
    \begin{align*}
        \gamma((F, \theta), \eta, \mathcal{L})=\mathcal{L}\otimes (F, \theta+1_{F}\otimes \eta). 
    \end{align*}
    From the above diagram, we have 
    \begin{align*}
        i^{-1}\phi_{\overline{w}}(\mathrm{IC}_{M^L}) & \cong 
        j^{-1}\gamma^{-1}\phi_{\overline{w}}(\mathrm{IC}_{M^L}) \\
        &\cong j^{-1}\phi_{\gamma^{\ast}\overline{w}}(\gamma^{-1}\mathrm{IC}_{M^L}) \\
        &\cong j^{-1}\phi_{\gamma^{\ast}\overline{w}}(\mathrm{IC}_{M_{\mathrm{SL}}^L}\boxtimes 
        \mathbb{Q}[h^0(L)+g]). 
    \end{align*}
    By the above formula for $\gamma$
    and $\overline{w}$, 
    we have 
    \begin{align*}
\gamma^{\ast}\overline{w}((F, \theta), \eta, \mathcal{L})=\langle \alpha, \mathrm{tr}(\theta+1_F \otimes \eta)^2 \rangle =\omega_{\mathrm{SL}} \boxplus r\langle \alpha, \eta^2 \rangle. 
    \end{align*}
    Let $q(\eta)=\langle \alpha, \eta^2 \rangle$, 
    which is a quadratic function on $H^0(L)$. 
    From the identity $\mathcal{M}_X^{\rm{red}}=\mathrm{Crit}(w)$ in the rank one case, 
    we see that $\mathrm{Crit}(q)=H^0(\Omega_C) \subset H^0(L)$. 
    By using the Thom–Sebastiani theorem, we have
    \begin{align*}
      j^{-1}\phi_{\gamma^{\ast}\overline{w}}(\mathrm{IC}_{M_{\mathrm{SL}}^L}\boxtimes 
        \mathbb{Q}[h^0(L)+g]) &\cong j^{-1}(\phi_{w_{\mathrm{SL}}}(\mathrm{IC}_{\mathrm{M}_{\mathrm{SL}}^L}) \boxtimes \phi_q(\mathbb{Q}[h^0(L)+g])) \\
        & \cong j^{-1}(\phi_{w_{\mathrm{SL}}}(\mathrm{IC}_{M_{\mathrm{SL}}^L} \boxtimes 
        \mathbb{Q}[2g]) \\
        &\cong \phi_{w_{\mathrm{SL}}}(\mathrm{IC}_{M_{\mathrm{SL}}^L}[2g]). 
    \end{align*}
    Therefore the isomorphism (\ref{isom:iIC}) holds. 
\end{proof}

We have the following commutative diagram, where the vertical 
arrows are Hitchin maps and the top arrows are as in the 
diagram (\ref{comm:MXSL}):
\begin{align*}
    \xymatrix{
M_{\mathrm{SL}} \ar[d] & M_{X, \mathrm{SL}} \ar[d] \inclusion \ar[l] & M_{\mathrm{SL}}^L \ar[d] 
\ar[rd]^-{w_{\mathrm{SL}}} & \\
B_{\geq 2} & B_{X, \geq 2} \ar[l]^-{p_B} \ar[r] & B^L_{\geq 2} \ar[r]_-{w_B} & \mathbb{C}.     
    }
\end{align*}

   \begin{prop}\label{prop:Ktop:phiB}
There is an isomorphism 
\begin{align*}
    \mathcal{K}_{B_{\geq 2}}^{\rm{top}}(\mathbb{T}_{\mathrm{SL}(r)}(\chi)_w)_{\mathbb{Q}}
    \cong p_{B\ast}\phi_{w_B}(\mathcal{K}_{B^L_{\geq 2}}^{\rm{top}}(\mathbb{T}^L_{\mathrm{SL}(r)}(\chi)_w)_{\mathbb{Q}}). 
\end{align*}
    \end{prop}
    \begin{proof}
        The conclusion follows using the same argument as in  (\ref{compute:pB}) and Lemma~\ref{lem:KtopSL:Omega} and
        Lemma~\ref{lem:BPS:SLisom}. 
    \end{proof}

\subsection{Parabolic framing of PGL-moduli spaces}
In this and in the next subsections, we compute the topological K-theory of quasi-BPS categories for the PGL-moduli spaces. The computation is more difficult than in the SL case. A first issue is that we cannot use the straightforward generalization of the Joyce--Song stable pairs, as we explain below.

Recall the PGL-Higgs moduli space 
\begin{align*}
    \mathcal{M}_{\mathrm{PGL}(r)}^L(\chi)=\mathcal{M}^L(r, \chi)^{\rm{tr}=0}/\mathcal{P}ic^0(C). 
\end{align*}
In this subsection, we construct a framed version of the PGL-Higgs moduli space. 
A subtlety here is that there is no natural action of $\mathcal{P}ic^0(C)$ on 
the moduli of stable Joyce--Song pairs $\mathcal{M}^L(r, \chi)^{\rm{JS}}$, so that we cannot take its quotient. Instead, 
we use parabolic framing to rigidify automorphisms and to construct a space with an action of $\mathrm{Pic}^0(C)$. A similar idea also appeared in~\cite{Todpara, Todpara2}. 

We fix $p \in C$, and define $\mathcal{M}^L(r, \chi)^{\rm{par}}$ to be the moduli stack of 
tuples 
\begin{align}\label{pgl:tuple}
    (F, \theta, \xi), \theta \colon F \to F \otimes L, \ 
    \xi \in F|_{p} \setminus \{0\}
\end{align}
such that $(F, \theta)$ is a semistable $L$-twisted Higgs bundle, 
and that for any surjection 
$\eta \colon (F, \theta) \twoheadrightarrow 
(F', \theta')$ with $\mu(F)=\mu(F')$, we have 
$\eta|_{p}(\xi) \neq 0$. 
\begin{remark}\label{rmk:pframing}
Let $S=\mathrm{Tot}_C(L)$ and 
let $E \in \Coh(S)$ correspond to $(F, \theta)$ by the spectral construction. 
Then giving $\xi$ is equivalent to 
giving a morphism 
\begin{align*}
    \xi \colon \mathcal{O}_{F_p}[-1] \to E,
\end{align*}
where $F_p \subset S$ is the fiber of $S \to C$ at $p$. 
The pair $(E, \xi)$ is nothing but the parabolic 
stable pair considered in~\cite{Todpara}. 
\end{remark}
Similarly to~\cite[Theorem~2.10]{Todpara},
the moduli stack $\mathcal{M}^L(r, \chi)^{\rm{par}}$ is a 
quasi-projective 
scheme such that the composition 
\begin{align*}
    \mathcal{M}^L(r, \chi)^{\rm{par}} \to \mathcal{M}^L(r, \chi)
    \to M^L(r, \chi)
\end{align*}
is projective. Here, the first morphism is forgetting $\xi$, 
which is a smooth morphism. 
When $l>2g-2$, the moduli space $\mathcal{M}^L(r, \chi)^{\rm{par}}$ is also smooth since $\mathcal{M}^L(r, \chi)$ is a smooth stack. 

There is a natural action of $\mathrm{Pic}^0(C)$ on 
$\mathcal{M}^L(r, \chi)^{\rm{par}}$ as follows. 
We identify $\mathrm{Pic}^0(C)$ with
the moduli stack of 
pairs 
\begin{align*}
    (\mathcal{L}, \iota), \ \iota \colon 
    \mathcal{L}|_{p} \stackrel{\cong}{\to} \mathbb{C},
\end{align*}
where $\mathcal{L}$ is a line bundle on $C$ of degree zero. 
Note that the isomorphism $\iota$ rigidifies the automorphisms of the line bundle $\mathcal{L}$.
Then the action is given by 
\begin{align}\label{act:L}
(\mathcal{L}, \iota) \circ (F, \theta, \xi)=(F \otimes \mathcal{L}, 
\theta \otimes 1_{\mathcal{L}}, (1_{F}\otimes \iota)(\xi)). 
\end{align}

We denote by 
\begin{align}\label{closed:tr=0}
    \mathcal{M}^L(r, \chi)^{\rm{tr}=0, \rm{par}} \subset \mathcal{M}^L(r, \chi)^{\rm{par}}
\end{align}
the closed subscheme consisting of tuples (\ref{pgl:tuple}) satisfying $\mathrm{tr}(\theta)=0$. 
The above action (\ref{act:L}) of $\mathrm{Pic}^0(C)$ restricts to the action 
on the closed subscheme (\ref{closed:tr=0}).
The parabolic framed PGL-moduli space is defined to be 
the quotient stack
\begin{align*}
    \mathcal{M}^L_{\mathrm{PGL}(r)}(\chi)^{\rm{par}}:=  \mathcal{M}^L(r, \chi)^{\rm{tr}=0, \rm{par}}/\mathrm{Pic}^0(C). 
\end{align*}
The above stack is a Deligne-Mumford stack which is smooth 
for $l>2g-2$. 
Alternatively, let $\mathcal{M}_{\mathrm{SL}(r)}^L(\chi)^{\rm{par}}$
be the fiber of the morphism 
\begin{align*}
\mathcal{M}^L(r, \chi)^{\rm{tr=0}} \to \mathrm{Pic}(C), \ (F, \theta, \xi) \mapsto 
\det F
\end{align*}
    at a fixed $A \in \mathrm{Pic}(C)$ of degree $\chi+r(g-1)$. 
Then we have 
\begin{align}\label{alt:PGLpar}
\mathcal{M}_{\mathrm{PGL}(r)}^L(\chi)^{\rm{par}}=\mathcal{M}^L_{\mathrm{SL}(r)}(\chi)^{\rm{par}}/\Gamma[r],
\end{align}
where $\Gamma[r] \subset \mathrm{Pic}^0(C)$ is the subgroup of 
$r$-torsion elements. We note the following analogue of 
Theorem~\ref{thm:sod:JS}: 

\begin{prop}\label{prop:sodpar}
Assume that $l>2g-2$. 
There is a semiorthogonal decomposition 
\begin{align}\label{sod:JS}
    D^b(\mathcal{M}^L_{\mathrm{PGL}(r)}(\chi)^{\rm{par}})=
\left\langle \mathbb{T}^L_{\mathrm{PGL}(r_{\bullet})}(\chi_{\bullet})_{w_{\bullet}} \,\Big|\,
0\leq \frac{v_1}{r_1}<\cdots<\frac{v_k}{r_k}<1 \right\rangle. 
\end{align}
The right hand side is after all partitions $r=r_1+\cdots+r_k$, 
$\chi=\chi_1+\cdots+\chi_k$ 
for $r_{\bullet} \in \mathbb{Z}^{\oplus k}$, $\chi_{\bullet} \in \mathbb{Z}^{\oplus k}/(r_1, \ldots, r_k)\mathbb{Z}$
such that $(\chi_1/r_1, \ldots, \chi_k/r_k)=(\chi/r, \ldots, \chi/r)$
in $\mathbb{Q}^{\oplus k}/\mathbb{Q}$,
and $w_{\bullet} \in \mathbb{Z}^{\oplus k}$. 
Each $v_i\in \frac{1}{2}\mathbb{Z}$
is determined by
\begin{align*}
    v_i=w_i-\frac{l}{2}r_i \left(\sum_{i>j}r_j-\sum_{i<j} r_j  \right). 
\end{align*}
\end{prop}
\begin{proof}
Similarly to the proof of Theorem~\ref{thm:sod:JS}, 
there is a semiorthogonal decomposition 
\begin{align}\label{sod:parabolic}
    D^b(\mathcal{M}^L(r, \chi)^{\rm{par}})=\left\langle \boxtimes_{i=1}^k \mathbb{T}^L(r_i, \chi_i)_{w_i} \,\Big|\, 0\leq \frac{v_1}{r_1}<\cdots<\frac{v_k}{r_k} <1 \right\rangle. 
\end{align}
By restricting 
it to the trace free part, 
and taking the quotient by $\mathrm{Pic}^0(C)$ as in 
the proof of~\cite[Theorem~7.2]{PThiggs}, we obtain 
the semiorthogonal decomposition (\ref{sod:JS}). 
    \end{proof}

\begin{remark}\label{rmk:pbundle}
If $(r, \chi)$ are coprime, then 
$\mathcal{M}^L(r, \chi)^{\rm{par}} \to M^L(r, \chi)$ is a 
$\mathbb{P}^{r-1}$-bundle, 
and the semiorthogonal decomposition (\ref{sod:parabolic}) is 
\begin{align*}
    D^b(\mathcal{M}^L(r, \chi)^{\rm{par}}) =\left\langle 
    \mathbb{T}^L(r, \chi)_{w} \,\Big|\, 0\leq w \leq r-1 \right\rangle. 
\end{align*}
Then the semiorthogonal decomposition (\ref{prop:sodpar}) is 
\begin{align*}
     D^b(\mathcal{M}^L_{\mathrm{PGL}(r)}(\chi)^{\rm{par}})=\left\langle 
     \mathbb{T}^L_{\mathrm{PGL}(r)}(\chi)_{w} \,\Big|\, 0\leq w \leq r-1 \right\rangle. 
\end{align*}
The above semiorthogonal 
decomposition also holds for $L=\Omega_C$,
since $M(r, \chi)$ is smooth if $(r, \chi)$ are coprime and 
$M(r, \chi)^{\rm{par}} \to M(r, \chi)$ is a $\mathbb{P}^{r-1}$-bundle
as well. 
\end{remark}

\subsection{Fixed loci of $\Gamma[r]$-actions}
The computation of the topological K-theory of quasi-BPS categories for PGL has contributions from the $\Gamma[r]$-fixed loci, and
uses the topological K-theory for moduli of Higgs bundles on certain étale covers of $C$. This argument is standard when studying Hausel--Thaddeus mirror symmetry, see for example the arguments in \cite{GS, MSint, MSend}.

For $\gamma \in \Gamma[r]$, let 
\begin{align*}
    \mathcal{M}^L(r, \chi)_{\gamma}^{\rm{par}} \subset 
    \mathcal{M}^L(r, \chi)^{\rm{par}}
\end{align*}
be the $\gamma$-fixed subscheme. 
We describe the above fixed locus in terms of 
parabolic-framed moduli space on an \'{e}tale cover of $C$. 

Let $m$ be the order of $\gamma$, so 
we have $r=mr'$ for a positive integer $r'$. 
If $\gamma$ corresponds to $(\mathcal{L}_{\gamma}, \iota)$, 
then 
\begin{align}\label{isom:eta}
    \eta \colon \mathcal{L}_{\gamma}^{\otimes m} \stackrel{\cong}{\to}
    \mathcal{O}_C, \ 
    \eta|_{p}=\iota^{\otimes m}. 
\end{align}
The isomorphism $\eta$ in (\ref{isom:eta}) determines a
finite \'{e}tale cover of degree $m$, with Galois group 
$G_{\gamma}=\mathbb{Z}/m$:
\begin{align*}
\pi_{\gamma} \colon \widetilde{C} \to C. 
\end{align*}
The curve $\widetilde{C}$ is given by 
\begin{align*}
\widetilde{C}=\Spec (\mathcal{O}_C \oplus \mathcal{L}_{\gamma}^{-1} \oplus 
\cdots \oplus \mathcal{L}_{\gamma}^{-m+1}). 
\end{align*}
The isomorphism $\iota \colon \mathcal{L}_{\gamma}|_{p} \stackrel{\cong}{\to}\mathbb{C}$
determines a $\pi_{\gamma\ast}\mathcal{O}_{\widetilde{C}}$-module 
structure on $\mathcal{O}_p$, which determines a lift of $p$ 
to a point $\widetilde{p} \in \widetilde{C}$. 
Let $\widetilde{L}=\pi_{\gamma}^{\ast}L$ and let
$\mathcal{M}^{\widetilde{L}}(r', \chi)^{\rm{par}}$
be the parabolic-framed moduli stack for $(\widetilde{C}, \widetilde{L})$, 
with parabolic framing at $\widetilde{p}$. 
\begin{lemma}\label{lem:fixed}
There is an isomorphism 
\begin{align*}
    \mathcal{M}^{\widetilde{L}}(r', \chi)^{\rm{par}} \stackrel{\cong}{\to}
    \mathcal{M}^L(r, \chi)_{\gamma}^{\rm{par}}
\end{align*}
by sending $(F, \theta, \xi)$ to $(\pi_{\gamma\ast}F, \pi_{\gamma\ast}\theta, \pi_{\gamma\ast}\xi)$, where 
$\pi_{\gamma\ast}\xi$ is given by the composition 
\begin{align*}
\mathbb{C} \stackrel{\xi}{\hookrightarrow}
    F|_{p} \hookrightarrow \bigoplus_{g \in G_{\gamma}}
    F|_{g(p)}=(\pi_{\gamma\ast}F)|_{p}. 
\end{align*}
Here, the second arrow is the inclusion into the direct summand. 
\end{lemma}
\begin{proof}
    The argument of~\cite[Proposition~4.3]{Todpara2} applies. 
\end{proof}

Let 
\begin{align*}
    \mathcal{M}^L_{\mathrm{SL}(r)}(\chi)_{\gamma}^{\rm{par}}
    \subset  \mathcal{M}^L_{\mathrm{SL}(r)}(\chi)^{\rm{par}}
\end{align*}
be the $\gamma$-fixed subscheme. We also denote by 
\begin{align*}
    \mathcal{M}^{\widetilde{L}}_{\mathrm{SL}(r')/\pi_{\gamma}}(\chi)^{\rm{par}}
    \subset \mathcal{M}^{\widetilde{L}}(r', \chi)^{\rm{par}}
\end{align*}
be the closed subscheme determined by $\mathrm{tr}(\pi_{\gamma\ast}\theta)=0$
and $\det(\pi_{\gamma\ast}F)=A$. 
By Lemma~\ref{lem:fixed}, we obtain the following: 
\begin{lemma}\label{lem:fixed2}
There is an isomorphism 
\begin{align*}
     \mathcal{M}^{\widetilde{L}}_{\mathrm{SL}(r')/\pi_{\gamma}}(\chi)^{\rm{par}}
     \stackrel{\cong}{\to}   \mathcal{M}^L_{\mathrm{SL}(r)}(\chi)_{\gamma}^{\rm{par}}
\end{align*}
by sending $(F, \theta, \xi)$ to $(\pi_{\gamma\ast}F, \pi_{\gamma\ast}\theta, \pi_{\gamma\ast}\xi)$. 
\end{lemma}

\subsection{Topological K-theory of PGL-moduli spaces}
In this subsection, we prove Proposition \ref{prop:prime2} about the supports of the relative topological K-theory of quasi-BPS categories for PGL-moduli spaces when $l>2g-2$.
For simplicity, we set 
$\mathcal{M}^{L\sharp}=\mathcal{M}_{\mathrm{SL}(r)}^L(\chi)^{\rm{par}}$
and let $h^{\sharp} \colon \mathcal{M}^{L\sharp} \to B^L_{\geq 2}$ be the 
Hitchin map. 
\begin{lemma}\label{lem:Ktop:Gamma}
Suppose that $l>2g-2$. 
There is a natural isomorphism 
\begin{align}\label{Ktop:gamma}
\mathcal{K}_{B^L_{\geq 2}}^{\rm{top}}(D^b(\mathcal{M}^{L\sharp}/\Gamma[r]))_{\mathbb{C}}
\stackrel{\cong}{\to} \bigoplus_{\gamma \in \Gamma[r]}\mathcal{K}_{B^L_{\geq 2}}^{\rm{top}}(D^b(\mathcal{M}^{L\sharp}_{\gamma}))_{\mathbb{C}}^{\Gamma[r]}.     
\end{align}
\end{lemma}
\begin{proof}
    For a complex analytic space $M$ with an action of a finite group $G$
    with a quotient $\pi \colon M \to M\ssslash G$, 
    and a morphism $h \colon M\ssslash G \to B$, let 
    $KU_{(M, G)/B}$ be the sheaf of spectra which assigns 
    to each open subset $U \subset B$ the $G$-equivariant 
    topological K-theory spectra of $\pi^{-1}h^{-1}(U)$
    (cf.~\cite[Appendix]{GS}).
By~\cite[Proposition~2.25]{GS},
there is an equivalence 
\begin{align*}
    \mathcal{K}_{\mathcal{M}^{L\sharp}\ssslash \Gamma[r]}^{\rm{top}}(D^b(\mathcal{M}^{L\sharp}/\Gamma[r]))
    \stackrel{\sim}{\to} KU_{(\mathcal{M}^{L\sharp}, \Gamma[r])/(\mathcal{M}^{L\sharp}\ssslash \Gamma[r])}. 
\end{align*}
By pushing forward via $h^{\sharp}$, we obtain an equivalence 
\begin{align*}
    \mathcal{K}_{B^L_{\geq 2}}^{\rm{top}}(D^b(\mathcal{M}^{L\sharp}/\Gamma[r]))
    \stackrel{\sim}{\to} KU_{(\mathcal{M}^{L\sharp}, \Gamma[r])/B^L_{\geq 2}}. 
\end{align*}
    By a theorem of Atiyah-Segal~\cite{AtSeg}, there is an equivalence
    over $\mathbb{C}$:  
    \begin{align*}
        (KU_{(\mathcal{M}^{L\sharp}, \Gamma[r])/B^L_{\geq 2}})_{\mathbb{C}}
        \stackrel{\sim}{\to} \bigoplus_{\gamma \in \Gamma[r]}(KU_{\mathcal{M}_{\gamma}^{L\sharp}/B^L_{\geq 2}})_{\mathbb{C}}^{\Gamma[r]}. 
    \end{align*}
    Then the lemma follows from the equivalence
    \begin{align*}
        \mathcal{K}_{B^L_{\geq 2}}^{\rm{top}}(D^b(\mathcal{M}_{\gamma}^{L\sharp}))
        \stackrel{\sim}{\to} KU_{\mathcal{M}_{\gamma}^{L\sharp}/B^L_{\geq 2}}. 
    \end{align*}
\end{proof}

We have the following commutative diagram 
\begin{align}\label{dia:PGL}
\xymatrix{
    \mathcal{M}^{\widetilde{L}}_{\mathrm{SL}(r')/\pi_{\gamma}}(\chi)^{\rm{par}}
    \ar[r]^-{\cong} \ar[d] & 
\mathcal{M}^L_{\mathrm{SL}(r)}(\chi)_{\gamma}^{\rm{par}} \inclusion \ar[dd] & 
\mathcal{M}^L_{\mathrm{SL}(r)}(\chi)^{\rm{par}} \ar[d] \\
\mathcal{M}^{\widetilde{L}}_{\mathrm{SL}(r')/\pi_{\gamma}}(\chi) \ar[d] & &
\mathcal{M}_{\mathrm{SL}(r)}^L(\chi) \ar[d] \\
M_{\mathrm{SL}(r')/\pi_{\gamma}}^{\widetilde{L}}(\chi) \ar[r] \ar[d]_-{\widetilde{h}} \ar[rd]_-{h_{\gamma}}& 
M^L_{\mathrm{SL}(r)}(\chi)_{\gamma} \inclusion \ar[d] & M^L_{\mathrm{SL}(r)}(\chi) \ar[d]_-{h} \\
B^L(\pi_{\gamma}) \ar[r]_-{q_{\gamma}} & B^L_{\gamma} \inclusion^-{i_{\gamma}} & B^L_{\geq 2}. 
}
\end{align}
In the above, $\mathcal{M}^{\widetilde{L}}_{\mathrm{SL}(r')/\pi_{\gamma}}(\chi)$ is the closed substack 
of $\mathcal{M}^{\widetilde{L}}(r', \chi)$ given by 
\[(\mathrm{tr}(\pi_{\gamma\ast}\theta), \det(\pi_{\gamma\ast}F))=(0, A),\] 
$M_{\mathrm{SL}(r)}^L(\chi)_{\gamma}$ is the $\gamma$-fixed 
closed subscheme of $M_{\mathrm{SL}(r)}^L(\chi)$, and 
the middle vertical arrows are good moduli space morphisms. 
The closed subscheme $B^L_{\gamma} \subset B^L_{\geq 2}$ is the 
image of $h$ restricted to $M_{\mathrm{SL}(r)}^L(\chi)_{\gamma}$, and 
$B^L(\pi_{\gamma})$ is the Hitchin base for 
$M^{\widetilde{L}}_{\mathrm{SL}(r')/\pi_{\gamma}}(\chi)$. 
The map $q_{\gamma} \colon B^L(\pi_{\gamma}) \to B^L_{\gamma}$ is a quotient map 
with respect to the $G_{\pi_{\gamma}}$-action, see~\cite[Section~1.5]{MSend}. 
\begin{remark}\label{rmk:Bell}
By~\cite[Remark~1.6]{MSend}, there is a dense open subset  
$B^L(\pi_{\gamma})^{\ast} \subset B^L(\pi_{\gamma})$ 
upon which the $G_{\pi_{\gamma}}$-action is free. 
It follows that the map $q_{\gamma}$ sends 
the generic point of $B^L(\pi_{\gamma})$ to a point in the 
elliptic locus in $B$. 
\end{remark}
Using the commutative diagram (\ref{dia:PGL}), Lemma~\ref{lem:fixed2}, and Lemma~\ref{lem:Ktop:Gamma},
we obtain the following: 
\begin{cor}\label{cor:Ktop:gamma2}
    Suppose that $l>2g-2$. There is a natural isomorphism
    \begin{align}\label{Ktop:gamma2}
\mathcal{K}_{B^L_{\geq 2}}^{\rm{top}}(D^b(\mathcal{M}^L_{\mathrm{PGL}(r)}(\chi)^{\rm{par}}))_{\mathbb{C}}
\stackrel{\cong}{\to}\bigoplus_{\gamma \in \Gamma[r]}
i_{\gamma\ast}\mathcal{K}_{B^L_{\gamma}}^{\rm{top}}(D^b(\mathcal{M}^{\widetilde{L}}_{\mathrm{SL}(r')/\pi_{\gamma}}(\chi)^{\rm{par}}))_{\mathbb{C}}^{\Gamma[r]}. 
    \end{align}
\end{cor}

\begin{prop}\label{prop:coprime}
Assume that $l>2g-2$ and that $(r', \chi)$ are coprime. 
Then the object
\begin{align*}
    \mathcal{K}_{B^L_{\gamma}}^{\rm{top}}(D^b(\mathcal{M}^{\widetilde{L}}_{\mathrm{SL}(r')/\pi_{\gamma}}(\chi))^{\rm{par}})_{\mathbb{Q}}^{\Gamma[r]} \in D(\mathrm{Sh}_{\mathbb{Q}}(B^L_{\gamma}))
\end{align*}
is of the form $(P\oplus Q[1])[\beta^{\pm 1}]$ for 
semisimple perverse sheaves $P, Q$ with full support $B^L_{\gamma}$.     
\end{prop}
\begin{proof}
For simplicity, we write $\widetilde{M}^L= M^{\widetilde{L}}_{\mathrm{SL}(r')/\pi_{\gamma}}(\chi)$.
When $(r', \chi)$ are coprime, the morphism 
\begin{align*}
    \mathcal{M}^{\widetilde{L}}_{\mathrm{SL}(r')/\pi_{\gamma}}(\chi)^{\rm{par}} \to 
   \widetilde{M}^L
\end{align*}
is an \'{e}tale locally trivial $\mathbb{P}^{r'-1}$-bundle. 
It follows that there is a semiorthogonal decomposition, see Remark~\ref{rmk:pbundle}: 
\begin{align*}
    D^b(\mathcal{M}^{\widetilde{L}}_{\mathrm{SL}(r')/\pi_{\gamma}}(\chi)^{\rm{par}})
    =\left\langle  D^b(\widetilde{M}^L, \alpha^i) \,\Big|\, 
    0\leq i\leq r'-1\right\rangle
\end{align*}
for a Brauer class $\alpha$ on $\widetilde{M}^L$. 
It is enough to show that the object
\begin{align}\label{Ktop:alpha}
    \mathcal{K}_{B_{\gamma}^L}^{\rm{top}}(D^b(\widetilde{M}^L, \alpha^i))_{\mathbb{Q}}^{\Gamma[r]}\in D(\mathrm{Sh}_{\mathbb{Q}}(B_{\gamma}^L))
\end{align}
is of the form $(P\oplus Q[1])[\beta^{\pm 1}]$ for semisimple 
perverse sheaves $P, Q$ with full support $B_{\gamma}^L$. 

Since the class $\widehat{\alpha} \in H^3(M, \mathbb{Z})$ associated with $\alpha$ is a torsion class, it does not affect the topological 
K-theory after the rationalization. Therefore we have 
\begin{align*}
    \mathcal{K}_{\widetilde{M}^L}^{\rm{top}}(D^b(\widetilde{M}^L, \alpha^i))_{\mathbb{Q}} 
    \cong \mathbb{Q}_{\widetilde{M}^L}[\beta^{\pm 1}]. 
\end{align*}
Then we have 
\begin{align*}
    \mathcal{K}_{B^L_{\gamma}}^{\rm{top}}(D^b(\widetilde{M}^L, \alpha^i))_{\mathbb{Q}} 
    \cong h_{\gamma\ast}\mathrm{IC}_{\widetilde{M}^L}[-d][\beta^{\pm 1}],
\end{align*}
where $d=\dim \widetilde{M}^L$
and $h_{\gamma}$ is the morphism in the diagram (\ref{dia:PGL}). 
It is proved in~\cite[Theorem~2.3]{MSend} that the $\Gamma[r]$-fixed 
part of 
$\widetilde{h}_{\ast}\mathrm{IC}_{\widetilde{M}^L}$
is of the form $\oplus_i A_i[-i]$ where $A_i$ is a 
semisimple perverse sheaf on $B^L(\pi_{\gamma})$ 
with full support $B^L(\pi_{\gamma})$.
Since $q_{\pi}$ is a quotient map and $h_{\gamma}=q_{\gamma}\circ \widetilde{h}$, 
it follows that 
$h_{\gamma\ast}\mathrm{IC}_{\widetilde{M}^L}$
is of the form $\oplus_i A_i'[-i]$ where $A_i'$ is a 
semisimple perverse sheaf on $B^L_{\gamma}$ 
with full support $B^L_{\gamma}$.
Therefore we obtain the desired conclusion. 
\end{proof}

\begin{prop}\label{prop:prime2}
Suppose that $l>2g-2$ and that the tuple $(r, \chi, w)$ satisfies the BPS condition. 
Moreover, assume that either $(r, \chi)$ are coprime or $r$ is a prime number. 
Then the object 
\begin{align}\label{Ktop:qbps:pgl}
\mathcal{K}_{B^L_{\geq 2}}^{\rm{top}}(\mathbb{T}^L_{\mathrm{PGL}(r)}(\chi)_w)_{\mathbb{Q}}
\in D(\mathrm{Sh}_{\mathbb{Q}}(B^L_{\geq 2}))
\end{align}
is of the form $(P \oplus Q[1])[\beta^{\pm 1}]$ for semisimple 
perverse sheaves $P, Q$ on $B^L_{\geq 2}$ whose 
generic supports are contained in $B^L_{\geq 2} \cap (B^L)^{\rm{ell}}$. 
\end{prop}
\begin{proof}
    It is enough to prove the proposition after taking the tensor product 
    with $\mathbb{C}$. By Proposition~\ref{prop:sodpar} and Corollary~\ref{cor:Ktop:gamma2}, 
   the object (\ref{Ktop:qbps:pgl}) over $\mathbb{C}$ is a direct summand of 
   the following direct sum 
 \begin{align}\label{Ktop:gamma3}
\bigoplus_{\gamma \in \Gamma[\pi]} i_{\gamma\ast}\mathcal{K}_{B^L_{\gamma}}^{\rm{top}}(D^b(\mathcal{M}^{\widetilde{L}}_{\mathrm{SL}(r')/\pi_{\gamma}}(\chi)^{\rm{par}}))_{\mathbb{C}}^{\Gamma[r]}
    \end{align}
The above object is of the form 
$\oplus_{\gamma}(P_{\gamma} \oplus Q_{\gamma}[1])[\beta^{\pm 1}]$ for 
semisimple perverse sheaves $P_{\gamma}, Q_{\gamma}$
on $B_{\geq 2}^L$. 
It follows that (\ref{Ktop:qbps:pgl}) is 
of the form $(P\oplus Q[1])[\beta^{\pm 1}]$ for semisimple 
perverse sheaves $P, Q$. It is enough to show that their 
generic supports are contained in $(B^L)^{\rm{ell}}$. 

       By the assumption that $(r, \chi)$ are coprime or $r$ is prime, 
    a summand in (\ref{Ktop:gamma3}) corresponding 
    to $\gamma \neq 1$ satisfies that $\gcd(r', \chi)=1$. 
    Therefore, by Proposition~\ref{prop:coprime}, 
    the perverse sheaves $P_{\gamma}, Q_{\gamma}$ for $\gamma \neq 1$ have
    full support $B^L_{\gamma}$. In particular, their generic support 
    is contained in $(B^L)^{\rm{ell}}$, see Remark~\ref{rmk:Bell}.  
   As for the $\gamma=1$ summand, by Proposition~\ref{prop:topK:SL}, the object 
    \begin{align*}
        \mathcal{K}_{B^L_{\geq 2}}^{\rm{top}}(\mathbb{T}_{\mathrm{SL}(r)}^L(\chi)_w)_{\mathbb{Q}}
        \in D(\mathrm{Sh}_{\mathbb{Q}}(B^L_{\geq 2}))
    \end{align*}
    is of the form $(P' \oplus Q'[1])[\beta^{\pm 1}]$ for 
    semisimple perverse sheaves $P'$, $Q'$ whose generic supports are 
    contained in $(B^L)^{\rm{ell}}$. 
    The isomorphism (\ref{Ktop:gamma2})
    sends $\mathcal{K}_{B^L_{\geq 2}}^{\rm{top}}(\mathbb{T}^L_{\mathrm{PGL}(r)}(\chi)_w)_{\mathbb{C}}$
    to 
    \begin{align*}
        \mathcal{K}_{B^L_{\geq 2}}^{\rm{top}}(\mathbb{T}_{\mathrm{SL}(r)}^L(\chi)_w)_{\mathbb{C}}^{\Gamma[r]}
        \oplus
        \bigoplus_{\gamma \neq 1} i_{\gamma\ast}\mathcal{K}_{B^L_{\gamma}}^{\rm{top}}(D^b(\mathcal{M}^{\widetilde{L}}_{\mathrm{SL}(r')/\pi_{\gamma}}(\chi)^{\rm{par}}))_{\mathbb{C}}^{\Gamma[r]}.
    \end{align*}
    Therefore $P, Q$ also have generic supports contained in $(B^L)^{\rm{ell}}$.       
\end{proof}

\subsection{Torsion-freeness for PGL-moduli spaces}

In this subsection, we discuss the torsion-freeness of 
topological K-theories of quasi-BPS categories for 
$\mathrm{PGL}$-Higgs moduli spaces when $l>2g-2$.

\begin{prop}\label{prop:PGL:free}
For $l>2g-2$, the topological $K$-group 
$K_{\ast}^{\rm{top}}(\mathbb{T}_{\mathrm{PGL}(r)}^L(\chi)_w)$ is torsion free. 
    \end{prop}
\begin{proof}
    Since $\mathbb{T}_{\mathrm{PGL}(r)}^L(\chi)_w$ is 
    a semiorthogonal summand of $D^b(\mathcal{M}_{\mathrm{PGL}(r)}^L(\chi)^{\rm{par}})$
    by Proposition~\ref{prop:sodpar}, 
    it is enough to show that 
    \begin{align*}KU_{\ast}(\mathcal{M}_{\mathrm{PGL}(r)}^L(\chi)^{\rm{par}})
    =\pi_{\ast}(KU_{\Gamma[r]}(\mathcal{M}_{\mathrm{SL}(r)}^{L}(\chi)^{\rm{par}}))
    \end{align*}
    is torsion free, where the right hand side is the 
    $\Gamma[r]$-equivariant topological K-theory. 
    The proof of this claim is the same as in the unframed case~\cite[Theorem~6.11]{GS}. 
Below we give its outline. 

    Let $J=\mathrm{Pic}^0(C)$. 
    There is a $J$-torsor: 
    \begin{align}\label{Jtors}
        \mathcal{M}^L(r, \chi)^{\rm{tr}=0, \mathrm{par}} \to 
        \mathcal{M}^L_{\mathrm{PGL}(r)}(\chi)^{\rm{par}}. 
    \end{align}
    By~\cite[Lemma~6.12]{GS},
    there is a corresponding $\mathbb{C}^{\ast}$-gerbe $\beta$ on 
    $\mathcal{M}^L_{\mathrm{PGL}(r)}(\chi)^{\rm{par}} \times J$ 
    and a derived equivalence 
    \begin{align*}
        D^b(\mathcal{M}^L(r, \chi)^{\rm{tr}=0, \mathrm{par}})
        \stackrel{\sim}{\to} D^b(\mathcal{M}^L_{\mathrm{PGL}(r)}(\chi)^{\rm{par}}\times J, \beta). 
    \end{align*}
    The gerbe $\beta$ splits after pull-back 
    via $[r] \colon J \to J$. 
    Indeed, the gerbe $[r]^{\ast}\beta$ corresponds to the $J$-torsor 
    given by push-forward of the torsor (\ref{Jtors}) by
    $[r] \colon J \to J$, 
    which is a trivial $J$-torsor by the isomorphism 
    \begin{align*}
    (\pi, \det): 
\mathcal{M}^L(r, \chi)^{\rm{tr}=0, \mathrm{par}}/\Gamma[r] \stackrel{\cong}{\to} 
\mathcal{M}_{\mathrm{PGL}(r)}^L(\chi)^{\rm{par}} \times J,
    \end{align*}
    where $\pi$ is the natural projection. 

    For a prime number $p$, 
    write $r=p^a r'$ with where $\gcd(p, r')=1$. 
    Let $\beta'$ be the pull-back of $\beta$ by 
    $[r'] \colon J \to J$. 
    By~\cite[Lemma~6.12]{GS}, there is a derived equivalence 
\begin{align*}
        D^b(\mathcal{M}^L(r, \chi)^{\rm{tr}=0, \mathrm{par}}/\Gamma[r'])
        \stackrel{\sim}{\to} D^b(\mathcal{M}^L_{\mathrm{PGL}(r)}(\chi)^{\rm{par}}\times J, \beta')
    \end{align*}
    which induces an equivalence of spectra 
    \begin{align}\label{equiv:Gamma'}
        KU_{\Gamma[r']}(\mathcal{M}^L(r, \chi)^{\rm{tr}=0, \mathrm{par}})
        \stackrel{\sim}{\to} KU^{\hat{\beta}'}(\mathcal{M}^L_{\mathrm{PGL}(r)}(\chi)^{\rm{par}}\times J). 
    \end{align}

    Assume by contradiction that 
    the $p$-torsion part of 
    $KU_{\ast}(\mathcal{M}_{\mathrm{PGL}(r)}^L(\chi)^{\rm{par}})$ is non-zero. 
    Let $\mathrm{pr}_2$ be the projection 
    \begin{align*}
        \mathrm{pr}_2 \colon \mathcal{M}^L_{\mathrm{PGL}(r)}(\chi)^{\rm{par}} \times J \to J
    \end{align*}
    and consider the following locally constant sheaf of spectra on $J$
    \begin{align*}        \mathcal{F}=\mathrm{pr}_{2\ast}\underline{KU}^{\hat{\beta}'}(\mathcal{M}^L_{\mathrm{PGL}(r)}(\chi)^{\mathrm{par}} \times J)
    \end{align*}
    with fiber $KU(\mathcal{M}_{\mathrm{PGL}(r)}^L(\chi)^{\rm{par}})$.
    The pull-back of $\mathcal{F}$ 
by $[p^a] \colon J \to J$ is a constant sheaf, as $[p^a]^{\ast}\beta'=[r]^{\ast}\beta$ 
is a trivial gerbe. 
Therefore, by~\cite[Lemma~6.18]{GS}, the homotopy groups of 
$\Gamma(\mathcal{F})$ have non-zero $p$-torsions. 
It follows that the homotopy groups of (\ref{equiv:Gamma'}) have 
non-zero $p$-torsions. 
On the other hand, by the Atiyah-Segal completion theorem, 
there is an isomorphism 
\begin{align}\label{as:complete}
    \widehat{\pi_{\ast}}(KU_{\Gamma[r']}(\mathcal{M}^L(r, \chi)^{\rm{tr}=0, \mathrm{par}}))
    \cong \pi_{\ast}(KU(\mathcal{M}^L(r, \chi)^{\rm{tr}=0, \mathrm{par}})^{h\Gamma[r']}). 
\end{align}
Here, the left hand side is the the completion with respect to the augmentation 
ideal of $\mathbb{Z}[\Gamma[r']]$, 
and $h\Gamma'$ in the right hand side means the homotopy invariant with respect to the 
$\Gamma[r']$-action. 
By Lemma~\ref{lem:tfree:par2} below
and using that the order of $\Gamma[r']$ is coprime with $p$, 
we obtain that the right hand side of (\ref{as:complete}) does not have 
non-zero $p$-torsion, which is a contradiction. 
\end{proof}

We have used the following lemma: 
\begin{lemma}\label{lem:tfree:par2}
For $l>2g-2$, the topological $K$-group
$KU_{\ast}(\mathcal{M}^L(r, \chi)^{\rm{tr}=0, \mathrm{par}})$
is torsion free. 
\end{lemma}
\begin{proof}
It is enough to show that
$KU_{\ast}(\mathcal{M}^L(r, \chi)^{\mathrm{par}})$
is torsion free, 
because of the isomorphism 
\begin{align*}
   \mathcal{M}^L(r, \chi)^{\mathrm{tr}=0, \mathrm{par}} \times H^0(L)
   \stackrel{\cong}{\to} \mathcal{M}^L(r, \chi)^{\mathrm{par}}. 
\end{align*}
By the semiorthogonal decomposition (\ref{sod:parabolic}),
we have the direct sum decomposition of $KU_{\ast}(\mathcal{M}^L(r, \chi)^{\mathrm{par}})$
into the direct sum of topological K-groups of the products of quasi-BPS categories. 
These direct summands are part of the direct summands of 
$K^{\rm{top}}(\mathcal{M}^{L\dag})$
for $\mathcal{M}^{L\dag}=\mathcal{M}^L(r, \chi)^{\rm{JS}}$
by the 
semiorthogonal decomposition in Theorem~\ref{thm:sod:JS}. 
Therefore the desired torsion-freeness follows 
from the torsion-freeness of $KU_{\ast}(\mathcal{M}^{L\dag})$, 
which is discussed in the proof of Proposition~\ref{prop:Ktop:free}. 
\end{proof}

\subsection{Topological K-theory for PGL-moduli spaces and $L=\Omega_C$}

We next consider $\mathcal{M}_{\mathrm{PGL}(r)}(\chi):=\mathcal{M}_{\mathrm{PGL}(r)}^{\Omega_C}(\chi)$. 
Assume that $(r, \chi)$ are coprime. 
In this case, both of $\mathcal{M}_{\mathrm{SL}(r)}(\chi)$ 
and 
$\mathcal{M}_{\mathrm{PGL}(r)}(\chi)$ are smooth Deligne-Mumford stacks. 
In particular, we have 
\begin{align*}
    \Omega_{\mathcal{M}_{\mathrm{SL}(r)}(\chi)}[-1]=\mathcal{M}_{\mathrm{SL}(r)}(\chi). 
\end{align*}
Consider a surjection $\mathcal{O}_C \oplus \Omega_C \twoheadrightarrow L$ as in the 
proof of Theorem~\ref{thm:induceK2}, and for 
simplicity write 
$\mathcal{M}_{\mathrm{PGL}}^L=\mathcal{M}_{\mathrm{PGL}(r)}^L(\chi)$, 
$\mathcal{M}_{\mathrm{PGL}}^{L\sharp}=\mathcal{M}_{\mathrm{PGL}(r)}^L(\chi)^{\rm{par}}$, etc. 
Consider the following diagram 
\begin{align*}
    \xymatrix{
\mathcal{M}_{\mathrm{PGL}}^{\sharp} \ar[d] \ar@<-0.3ex>@{^{(}->}[rr] & & \mathcal{M}_{\mathrm{PGL}}^{L\sharp} \ar[d] \ar[rdd] \\
\mathcal{M}_{\mathrm{PGL}} \ar@<-0.3ex>@{^{(}->}[rr] \ar[d] 
& & \mathcal{M}_{\mathrm{PGL}}^L \ar[d] \ar[rd] \\
B_{\geq 2} & B_X^{\geq 2} \ar[r] \ar[l]^-{p_B} & B^L_{\geq 2}  \ar[r]_-{w_B}& \mathbb{C}.     
    }
\end{align*}
Here, the horizontal inclusions are induced by the embedding 
$\Omega_C \hookrightarrow L$, and the function $w_B$
is given as in the diagram (\ref{diagram:MMB}). 
\begin{prop}\label{prop:Ktop:vanish:pgl}
Suppose that $(r, \chi)$ are coprime. There is an isomorphism 
\begin{align}\label{isom:Ktop:vanish:pgl}
p_{B\ast}\phi_{w_B}(\mathcal{K}_{B^L_{\geq 2}}^{\rm{top}}(\mathbb{T}^L_{\mathrm{PGL}(r)}(\chi)_{w})_{\mathbb{Q}})
\cong \mathcal{K}_{B_{\geq 2}}^{\rm{top}}(\mathbb{T}_{\mathrm{PGL}(r)}(\chi)_w)_{\mathbb{Q}}. 
\end{align}
    \end{prop}
    \begin{proof}
        It is enough to prove the proposition after tensoring $\otimes_{\mathbb{Q}}\mathbb{C}$.
        Indeed, both sides over $\mathbb{C}$ are of the form 
        $(P_{\mathbb{C}} \oplus Q_{\mathbb{C}}[1])[\beta^{\pm 1}]$ 
        for semisimple perverse sheaves $P$, $Q$ with $\mathbb{C}$-coefficients. 
        Then both sides in (\ref{isom:Ktop:vanish:pgl}) are of the form 
        $(P\oplus Q[1])[\beta^{\pm 1}]$ for semisimple perverse 
        sheaves with $\mathbb{Q}$-coefficients. If two of such objects 
        are isomorphic over 
        $\mathbb{C}$, they are also isomorphic over $\mathbb{Q}$. 

        We first prove that 
        \begin{align}\label{isom:Ktop:pgl2}
p_{B\ast}\phi_{w_B}(\mathcal{K}_{B^L_{\geq 2}}^{\rm{top}}(D^b(\mathcal{M}_{\mathrm{PGL}}^{L\sharp})))_{\mathbb{C}} \cong 
\mathcal{K}_{B_{\geq 2}}^{\rm{top}}(D^b(\mathcal{M}_{\mathrm{PGL}}^{\sharp}))_{\mathbb{C}}. 
        \end{align}
        By the semiorthogonal decomposition (\ref{sod:parabolic}), Lemma~\ref{lem:Ktop:Gamma},
        and Lemma~\ref{lem:fixed2}, 
        there are decompositions
        \begin{align*}
            \mathcal{K}_{B^L_{\geq 2}}^{\rm{top}}(D^b(\mathcal{M}_{\mathrm{PGL}}^{L\sharp}))_{\mathbb{C}}
            &=\bigoplus_{\gamma \in \Gamma[r]}\left(\bigoplus_{0\leq w\leq r'-1}j_{\gamma\ast}^L
            \mathcal{K}_{B^L(\pi_{\gamma})}^{\rm{top}}(\mathbb{T}^{\widetilde{L}}_{\mathrm{SL}(r')/\pi_{\gamma}}(\chi)_w)_{\mathbb{C}}\right)^{\Gamma[r]}, \\
 \mathcal{K}_{B_{\geq 2}}^{\rm{top}}(D^b(\mathcal{M}_{\mathrm{PGL}}^{\sharp}))_{\mathbb{C}}
            &=\bigoplus_{\gamma \in \Gamma[r]}\left(\bigoplus_{0\leq w\leq r'-1}j_{\gamma\ast}
            \mathcal{K}_{B(\pi_{\gamma})}^{\rm{top}}(\mathbb{T}_{\mathrm{SL}(r')/\pi_{\gamma}}(\chi)_w)_{\mathbb{C}}\right)^{\Gamma[r]}.
                    \end{align*}
      Here, we have that $r'=r/m$, where $m$ is the order of $\gamma$, 
       $\pi_{\gamma} \colon \widetilde{C} \to C$ is the Galois cover 
    associated with $\gamma$, $\widetilde{L}=\pi_{\gamma}^{\ast}L$,
      and we have
    used the notation of the diagram 
    \begin{align*}
        \xymatrix{
\mathcal{M}_{\mathrm{SL}(r')/\pi_{\gamma}}^{\widetilde{L}}(\chi)^{\rm{par}} \ar[r]^-{\cong} \ar[d] & 
\mathcal{M}^{L}_{\mathrm{SL}(r)}(\chi)^{\rm{par}}_{\gamma} \inclusion \ar[d] & 
\mathcal{M}^L_{\mathrm{SL}(r)}(\chi)^{\rm{par}} \ar[d] \ar[rd] \\
B^L(\pi_{\gamma}) \ar[r]\ar@/_18pt/[rr]_-{j_{\gamma}^L} & B^L_{\gamma} \inclusion & B^L_{\geq 2} 
\ar[r]_-{w_B} & \mathbb{C}.         
        }
    \end{align*}
The subcategory 
\begin{align*}
    \mathbb{T}^{\widetilde{L}}_{\mathrm{SL}(r')/\pi_{\gamma}}(\chi)_w \subset D^b(\mathcal{M}^{\widetilde{L}}_{\mathrm{SL}(r')/\pi_{\gamma}}(\chi))_w
\end{align*}
is defined similarly to $\mathbb{T}_{\mathrm{SL}(r)}^L(\chi)_w$, which is 
indeed equivalent to the right hand side as $(r', \chi)$ 
are coprime. 
By the formula (\ref{funct:ow}), the function 
    \begin{align*}
    \mathcal{M}^{\widetilde{L}}_{\mathrm{SL}(r')}(\chi)^{\rm{par}} \to 
    B^L(\pi_{\gamma}) \stackrel{j_{\gamma}^L}{\to} B^L_{\geq 2} \stackrel{w_B}{\to} \mathbb{C}
    \end{align*}
    is given by the same formula as (\ref{funct:ow}), i.e. 
    \begin{align*}
        (F, \theta, \xi) \mapsto \langle \pi_{\gamma}^{\ast}\alpha, \mathrm{tr}(\theta^2) \rangle.
    \end{align*}
    Then the isomorphism (\ref{isom:Ktop:pgl2}) follows from isomorphisms
    \[p_{B\ast}\phi_{w}\left(j^L_{\gamma*}\mathcal{K}_{B^L(\pi_{\gamma})}^{\rm{top}}(\mathbb{T}^{\widetilde{L}}_{\mathrm{SL}(r')/\pi_{\gamma}}(\chi)_w)\right)_\mathbb{C}\cong j_{\gamma*}\mathcal{K}_{B(\pi_{\gamma})}^{\rm{top}}(\mathbb{T}_{\mathrm{SL}(r')/\pi_{\gamma}}(\chi)_w)_{\mathbb{C}},\]
    which are proved completely analogously to  Proposition~\ref{prop:Ktop:phiB}. 

    By Proposition~\ref{prop:sodpar} (see also Remark~\ref{rmk:pbundle}),
    the isomorphism (\ref{isom:Ktop:pgl2}) is 
    \begin{align*}
    \bigoplus_{0\leq w \leq r-1}
        p_{B\ast}\phi_{w_B}(\mathcal{K}_{B^L_{\geq 2}}^{\rm{top}}(\mathbb{T}^L_{\mathrm{PGL}(r)}(\chi))_{\mathbb{C}})
\cong 
\bigoplus_{0\leq w \leq r-1}
\mathcal{K}_{B_{\geq 2}}^{\rm{top}}(\mathbb{T}_{\mathrm{PGL}(r)}(\chi)_w)_{\mathbb{C}}. 
    \end{align*}
    It is straightforward to check that the above isomorphism 
    preserves the direct summands. 
    We thus obtain the isomorphism (\ref{isom:Ktop:vanish:pgl}). 
    \end{proof}

\subsection{The SL/PGL-duality of topological K-theories}\label{subsec:slpgl:topk}
For simplicity, we write 
\begin{align*}\mathcal{M}_{\mathrm{SL}}^L=\mathcal{M}_{\mathrm{SL}(r)}^L(\chi), \ \mathcal{M}^{L'}_{\mathrm{PGL}}=\mathcal{M}_{\mathrm{PGL}(r)}^L(w+1-g^{\rm{sp}}).
\end{align*} 
We denote
their pull-backs to the elliptic locus 
$B^L_{\geq 2} \cap (B^L)^{\rm{ell}}$ by $(\mathcal{M}_{\mathrm{SL}}^L)^{\rm{ell}}$, $(\mathcal{M}_{\mathrm{PGL}}^{L'})^{\rm{ell}}$ respectively. 
By~\cite[Section~4]{GS},
the Poincare sheaf $\mathcal{P}^{\rm{ell}}$
on $(\mathcal{M}^{L'})^{\rm{ell}}\times_{B^L_{\geq 2}}(\mathcal{M}^L)^{\rm{ell}}$ 
induces the maximal Cohen-Macaulay
sheaf
\begin{align*}
    \overline{\mathcal{P}}^{\rm{ell}} \in \Coh((\mathcal{M}_{\mathrm{PGL}}^{L'})^{\rm{ell}} \times_{B^L_{\geq 2}}
    (\mathcal{M}_{\mathrm{SL}}^L)^{\rm{ell}})
\end{align*}
and that there is a derived equivalence 
\begin{align}\label{equiv:slpgl}
    \Phi_{\overline{\mathcal{P}}^{\rm{ell}}} \colon 
    D^b((\mathcal{M}_{\mathrm{PGL}}^{L'})^{\rm{ell}})_{-\chi-g^{\rm{sp}}+1} \stackrel{\sim}{\to} 
    D^b((\mathcal{M}^L_{\mathrm{SL}})^{\rm{ell}})_{w}. 
\end{align}
As in Lemma~\ref{lem:lift}, 
there is an extension of $\overline{\mathcal{P}}^{\rm{ell}}$:
\begin{align*}
    \overline{P} \in \mathbb{T}^L_{\mathrm{PGL}(r)}(w+1-g^{\rm{sp}})_{\chi+g^{\rm{sp}}-1}
    \boxtimes_{B^{L}_{\geq 2}} \mathbb{T}^L_{\mathrm{SL}(r)}(\chi)_w
\end{align*}
with the induced Fourier-Mukai functor: 
\begin{align*}
    \Phi_{\overline{\mathcal{P}}} \colon 
    \mathbb{T}^L_{\mathrm{PGL}(r)}(w+1-g^{\rm{sp}})_{-\chi-g^{\rm{sp}}+1} \to 
    \mathbb{T}^L_{\mathrm{SL}(r)}(\chi)_w.
\end{align*}
The above functor induces a map of spectra
\begin{align}\label{equiv:topK:SLPGL0}
     \Phi_{\overline{\mathcal{P}}} \colon 
    K^{\rm{top}}(\mathbb{T}^L_{\mathrm{PGL}(r)}(w+1-g^{\rm{sp}})_{-\chi-g^{\rm{sp}}+1})
    \to 
    K^{\rm{top}}(\mathbb{T}^L_{\mathrm{SL}(r)}(\chi)_w).
\end{align}
\begin{thm}\label{thm:topK:SLPGL}
Suppose that $l>2g-2$ and that the tuple $(r, \chi, w)$ satisfies the BPS condition. 
Furthermore, assume that $(r, w+1-g^{\rm{sp}})$ are coprime or that $r$ is a prime number. 
Then the functor $\Phi_{\overline{\mathcal{P}}}$ induces an equivalence
of spectra
\begin{align}\label{equiv:topK:SLPGL}
    \Phi_{\overline{\mathcal{P}}} \colon 
    K^{\rm{top}}(\mathbb{T}^L_{\mathrm{PGL}(r)}(w+1-g^{\rm{sp}})_{-\chi-g^{\rm{sp}}+1})
    \stackrel{\sim}{\to} 
    K^{\rm{top}}(\mathbb{T}^L_{\mathrm{SL}(r)}(\chi)_w).
\end{align}
\end{thm}
\begin{proof}
As in the proof of Proposition~\ref{prop:induceK}, 
we first show that $\Phi_{\overline{\mathcal{P}}}$ induces an 
isomorphism
\begin{align}\label{equiv:Ktop:slpgl}
     \Phi_{\overline{\mathcal{P}}} \colon 
    \mathcal{K}_{B^L_{\geq 2}}^{\rm{top}}(\mathbb{T}^L_{\mathrm{PGL}(r)}(w+1-g^{\rm{sp}})_{-\chi-g^{\rm{sp}}+1})_{\mathbb{Q}}
    \stackrel{\sim}{\to} 
    \mathcal{K}_{B^L_{\geq 2}}^{\rm{top}}(\mathbb{T}^L_{\mathrm{SL}(r)}(\chi)_w)_{\mathbb{Q}}.
\end{align}
By Proposition~\ref{prop:topK:SL} and Proposition~\ref{prop:prime2}, 
both sides are of the form $(P\oplus Q[1])[\beta^{\pm 1}]$ for semisimple 
perverse sheaves $P, Q$ with generic supports 
contained in $B^L_{\geq 2} \cap (B^L)^{\rm{ell}}$.  
Therefore, it is enough to show the equivalence (\ref{equiv:Ktop:slpgl}) over the elliptic locus, 
which follows from the equivalence (\ref{equiv:slpgl}). 
Therefore (\ref{equiv:Ktop:slpgl}) is an equivalence. Taking the global 
section and using Proposition~\ref{prop:gsection}, 
we obtain the isomorphism
\begin{align*}
\Phi_{\overline{\mathcal{P}}} \colon 
    K^{\rm{top}}(\mathbb{T}^L_{\mathrm{PGL}(r)}(w+1-g^{\rm{sp}})_{-\chi-g^{\rm{sp}}+1})_{\mathbb{Q}}
    \stackrel{\sim}{\to} 
    K^{\rm{top}}(\mathbb{T}^L_{\mathrm{SL}(r)}(\chi)_w)_{\mathbb{Q}}
    \end{align*}
    The homotopy groups of both sides in (\ref{equiv:topK:SLPGL0})
    are torsion free by Lemma~\ref{lem:sl:tfree} and Proposition~\ref{prop:PGL:free}, 
    and the map (\ref{equiv:topK:SLPGL0})
    is an isomorphism after rationalization by the above argument. 
    Therefore, as in the proof of Theorem~\ref{thm:induceK}, 
    we have the equivalence (\ref{equiv:topK:SLPGL}). 
\end{proof}

We now prove part (2) of Theorem~\ref{thm:intro2} and part (2) of Theorem~\ref{thm:intro}.
Let $\mathcal{M}_{\mathrm{SL}(r)}^L$ be the
moduli stack of $L$-twisted $\mathrm{SL}(r)$-principal
Higgs bundles, namely: 
\begin{align*}
    \mathcal{M}_{\mathrm{SL}(r)}^L :=\mathcal{M}_{\mathrm{SL}(r)}^L(\chi=r(1-g)).
\end{align*}
    The quasi-BPS category for principal $\mathrm{SL}$-Higgs bundles 
    is 
    \begin{align*}
        \mathbb{T}^L_{\mathrm{SL}(r), w}:=\mathbb{T}^L_{\mathrm{SL}(r)}(\chi=r(1-g))_{w}
        \subset D^b(\mathcal{M}_{\mathrm{SL}(r)}^L)_{w}. 
    \end{align*}
We also set
\begin{align*}
    \mathbb{T}^L_{\mathrm{PGL}(r)}(\chi) :=
    \begin{cases}
\mathbb{T}_{\mathrm{PGL}(r)}^L(\chi)_{r/2}, & l \mbox{ is odd and } r\mbox{ is even} \\
\mathbb{T}_{\mathrm{PGL}(r)}^L(\chi)_0, & \mbox{otherwise. }        
    \end{cases}
\end{align*}
We say that $(r, w)$ satisfies the BPS condition if $(r, r(1-g), w)$ 
satisfies the BPS condition, or equivalently if
$(r, w+1-g^{\rm{sp}})$ are coprime. This is also 
equivalent to either one of the 
following conditions
\begin{enumerate}
    \item $l$ is even and $(r, w)$ are coprime,
    \item $l$ is odd and $(r, w)$ are coprime with $r\not\equiv 2 \pmod 4$,
    \item $l$ is odd and $(r, w)$ has divisibility $2$ with $r/2$ odd. 
\end{enumerate}
\begin{cor}\label{cor:topK:untsited}
Suppose that $l>2g-2$ and $(r, w)$ satisfies the BPS condition. 
Then there is an equivalence of spectra
\begin{align}\label{equiv:untwisted}
    K^{\rm{top}}(\mathbb{T}^L_{\mathrm{PGL}(r)}(w+1-g^{\rm{sp}})) \stackrel{\sim}{\to} 
    K^{\rm{top}}(\mathbb{T}^L_{\mathrm{SL}(r), w}).
\end{align}
\end{cor}
\begin{proof}
 For $\chi=r(1-g)$, we have 
    \begin{align*}
      -\chi-g^{\rm{sp}}+1 \equiv 
      \begin{cases}
          r/2 \pmod r, & l \mbox{ is odd and } r\mbox{ is even} \\
          0 \pmod r, & \mbox{otherwise. }
      \end{cases}
    \end{align*}
    It follows that we have 
    \begin{align*}
        \mathbb{T}_{\mathrm{PGL}(r)}^L(w+1-g^{\rm{sp}})_{-r(1-g)-g^{\rm{sp}}+1}
        \simeq \mathbb{T}_{\mathrm{PGL}(r)}^L(w+1-g^{\rm{sp}}). 
    \end{align*}
    Therefore we obtain the desired equivalence (\ref{equiv:untwisted})
    from Theorem~\ref{thm:topK:SLPGL}.  
\end{proof}

Recall that we denoted by 
$\mathbb{T}_{\mathrm{PGL}(r)}(w):=\mathbb{T}_{\mathrm{PGL}(r)}^{L=\Omega_C}(w)$, \ 
$\mathbb{T}_{\mathrm{SL}(r), w}:=\mathbb{T}_{\mathrm{SL}(r), w}^{L=\Omega_C}$. We finally obtain the 
following result: 

\begin{thm}\label{thm:topK:slpgl2}
Suppose that $(r, w)$ are coprime. 
Then there is an equivalence 
\begin{align*}
    K^{\rm{top}}(\mathbb{T}_{\mathrm{PGL}(r)}(w))_{\mathbb{Q}} \stackrel{\sim}{\to} 
    K^{\rm{top}}(\mathbb{T}_{\mathrm{SL}(r), w})_{\mathbb{Q}}. 
\end{align*}
\end{thm}
\begin{proof}
Note that $1-g^{\rm{sp}}$ is divisible by $r$ in the case of 
    $L=\Omega_C$,
    so \[\mathbb{T}_{\mathrm{PGL}(r)}(w+1-g^{\rm{sp}})\cong\mathbb{T}_{\mathrm{PGL}(r)}(w).\]
    The claim follows from 
    applying $p_{B\ast}\phi_{w_B}$ to the equivalence (\ref{equiv:topK:SLPGL}), and 
    using Proposition~\ref{prop:Ktop:phiB} and Proposition~\ref{prop:Ktop:vanish:pgl}. 
\end{proof}

     \bibliographystyle{amsalpha}
\bibliography{math}

\medskip

\textsc{\small Tudor P\u adurariu: Sorbonne Université and Université Paris Cité, CNRS, IMJ-PRG, F-75005 Paris, France.}\\
\textit{\small E-mail address:} \texttt{\small padurariu@imj-prg.fr}\\

\textsc{\small Yukinobu Toda: Kavli Institute for the Physics and Mathematics of the Universe (WPI), University of Tokyo, 5-1-5 Kashiwanoha, Kashiwa, 277-8583, Japan.}\\
\textit{\small E-mail address:} \texttt{\small yukinobu.toda@ipmu.jp}\\

 \end{document}